\documentclass[a4paper]{amsart}
\usepackage{amssymb,amsthm}
\usepackage[abbrev,nobysame]{amsrefs}

\usepackage{appendix}

\theoremstyle{plain}
\newtheorem{theorem}{Theorem}[section]
\newtheorem{lemma}[theorem]{Lemma}

\theoremstyle{definition}
\newtheorem{definition}[theorem]{Definition}
\newtheorem{assumption}[theorem]{Assumption}

\theoremstyle{remark}
\newtheorem{remark}[theorem]{Remark}

\numberwithin{equation}{section}

\begin{document}

\title[Navier--Stokes equations in a curved thin domain, Part II]{Navier--Stokes equations in a curved thin domain, Part II: global existence of a strong solution}

\author[T.-H. Miura]{Tatsu-Hiko Miura}
\address{Department of Mathematics, Kyoto University, Kitashirakawa Oiwake-cho, Sakyo-ku, Kyoto 606-8502, Japan}
\email{t.miura@math.kyoto-u.ac.jp}



\begin{abstract}
  We consider the Navier--Stokes equations in a three-dimensional curved thin domain around a given closed surface under Navier's slip boundary conditions.
  When the thickness of the thin domain is sufficiently small, we establish the global existence of a strong solution for large data.
  We also show several estimates for the strong solution with constants explicitly depending on the thickness of the thin domain.
  The proofs of these results are based on a standard energy method and a good product estimate for the convection and viscous terms following from a detailed study of average operators in the thin direction.
  We use the average operators to decompose a three-dimensional vector field on the thin domain into the almost two-dimensional average part and the residual part, and derive good estimates for them which play an important role in the proof of the product estimate.
\end{abstract}

\maketitle

\section{Introduction} \label{S:Intro}
\subsection{Problem and main results} \label{SS:Int_Pro}
This paper is the second part of a three-part series including~\cites{Miu_NSCTD_01,Miu_NSCTD_03} of the study of the Navier--Stokes equations in a three-dimensional curved thin domain.
Let $\Gamma$ be a closed surface in $\mathbb{R}^3$ with unit outward normal vector field $n$.
Also, let $g_0$ and $g_1$ be functions on $\Gamma$ satisfying
\begin{align*}
  g := g_1-g_0 \geq c \quad\text{on}\quad \Gamma
\end{align*}
with some constant $c>0$.
For a sufficiently small $\varepsilon\in(0,1]$ we define a curved thin domain $\Omega_\varepsilon$ in $\mathbb{R}^3$ with small thickness of order $\varepsilon$ by
\begin{align} \label{E:Def_CTD}
  \Omega_\varepsilon := \{y+rn(y) \mid y\in\Gamma,\, \varepsilon g_0(y) < r < \varepsilon g_1(y)\}
\end{align}
and consider the Navier--Stokes equations with Navier's slip boundary conditions
\begin{align} \label{E:NS_CTD}
  \left\{
  \begin{alignedat}{3}
    \partial_tu^\varepsilon+(u^\varepsilon\cdot\nabla)u^\varepsilon-\nu\Delta u^\varepsilon+\nabla p^\varepsilon &= f^\varepsilon &\quad &\text{in} &\quad &\Omega_\varepsilon\times(0,\infty), \\
    \mathrm{div}\,u^\varepsilon &= 0 &\quad &\text{in} &\quad &\Omega_\varepsilon\times(0,\infty), \\
    u^\varepsilon \cdot n_\varepsilon &= 0 &\quad &\text{on} &\quad &\Gamma_\varepsilon\times(0,\infty), \\
    [\sigma(u^\varepsilon,p^\varepsilon)]_{\mathrm{tan}}+\gamma_\varepsilon u^\varepsilon &= 0 &\quad &\text{on} &\quad &\Gamma_\varepsilon\times(0,\infty), \\
    u^\varepsilon|_{t=0} &= u_0^\varepsilon &\quad &\text{in} &\quad &\Omega_\varepsilon.
  \end{alignedat}
  \right.
\end{align}
Here $\nu>0$ is the viscosity coefficient independent of $\varepsilon$.
Also, $\Gamma_\varepsilon$ is the boundary of $\Omega_\varepsilon$ with unit outward normal vector field $n_\varepsilon$ that consists of the inner and outer boundaries $\Gamma_\varepsilon^0$ and $\Gamma_\varepsilon^1$ of the form
\begin{align*}
  \Gamma_\varepsilon^i := \{y+\varepsilon g_i(y)n(y) \mid y\in\Gamma\}, \quad i=0,1,
\end{align*}
and $\gamma_\varepsilon\geq0$ is the friction coefficient on $\Gamma_\varepsilon$ given by
\begin{align} \label{E:Def_Fric}
  \gamma_\varepsilon := \gamma_\varepsilon^i \quad\text{on}\quad \Gamma_\varepsilon^i,\, i=0,1
\end{align}
with nonnegative constants $\gamma_\varepsilon^0$ and $\gamma_\varepsilon^1$ depending on $\varepsilon$.
We denote by $I_3$ and $n_\varepsilon\otimes n_\varepsilon$ the $3\times 3$ identity matrix and the tensor product of $n_\varepsilon$ with itself, write
\begin{align*}
  D(u^\varepsilon) := \frac{\nabla u^\varepsilon+(\nabla u^\varepsilon)^T}{2}, \quad P_\varepsilon:=I_3-n_\varepsilon\otimes n_\varepsilon
\end{align*}
for the strain rate tensor and the orthogonal projection onto the tangent plane of $\Gamma_\varepsilon$, and define the stress tensor and the tangential component of the stress vector on $\Gamma_\varepsilon$ by
\begin{align*}
  \sigma(u^\varepsilon,p^\varepsilon) := 2\nu D(u^\varepsilon)-p^\varepsilon I_3, \quad [\sigma(u^\varepsilon,p^\varepsilon)n_\varepsilon]_{\mathrm{tan}} := P_\varepsilon[\sigma(u^\varepsilon,p^\varepsilon)n_\varepsilon].
\end{align*}
Note that $[\sigma(u^\varepsilon,p^\varepsilon)n_\varepsilon]_{\mathrm{tan}}=2\nu P_\varepsilon D(u^\varepsilon)n_\varepsilon$ is independent of the pressure $p^\varepsilon$ and thus the slip boundary conditions are of the form
\begin{align} \label{E:Slip_Intro}
  u^\varepsilon\cdot n_\varepsilon = 0, \quad 2\nu P_\varepsilon D(u^\varepsilon)n_\varepsilon+\gamma_\varepsilon u^\varepsilon = 0 \quad\text{on}\quad \Gamma_\varepsilon.
\end{align}
We refer to \eqref{E:Slip_Intro} as the slip boundary conditions in the sequel.
The fluid subject to \eqref{E:Slip_Intro} slips on the boundary with velocity proportional to the tangential component of the stress vector.
Such conditions were introduced by Navier~\cite{Na1823} and are seen as an appropriate model for flows with free boundaries and for flows past chemically reacting walls (see~\cite{Ve87}).
They also appear in the study of the atmosphere and ocean dynamics~\cites{LiTeWa92a,LiTeWa92b,LiTeWa95} and in the homogenization of the no-slip boundary condition on a rough boundary~\cites{Hi16,JaMi01}.

The purpose of this paper is to establish the global-in-time existence of a strong solution $u^\varepsilon$ to \eqref{E:NS_CTD} for large data $u_0^\varepsilon$ and $f^\varepsilon$ in the sense that
\begin{align*}
  \|u_0^\varepsilon\|_{H^1(\Omega_\varepsilon)},\, \|f^\varepsilon\|_{L^\infty(0,\infty;L^2(\Omega_\varepsilon))} = O(\varepsilon^{-1/2})
\end{align*}
when $\varepsilon$ is sufficiently small (see Theorem~\ref{T:GE} for the precise statement).
Our result generalizes the existence results of~\cites{Ho10,HoSe10,IfRaSe07} for the Navier--Stokes equations in flat thin domains under the slip boundary conditions (see Remark~\ref{R:GE_Prev}).
We also derive several estimates for $u^\varepsilon$ with constants explicitly depending on $\varepsilon$ (see Theorem~\ref{T:UE}).
Those estimates are essential for the study of the thin-film limit for \eqref{E:NS_CTD} carried out in the last part~\cite{Miu_NSCTD_03} of our study.
We prove in~\cite{Miu_NSCTD_03} that the average in the thin direction of $u^\varepsilon$ converges on $\Gamma$ as $\varepsilon\to0$ and characterize its limit as a solution to limit equations on $\Gamma$.
Moreover, we observe that the limit equations agree with the Navier--Stokes equations on a Riemannian manifold introduced in~\cites{EbMa70,MitYa02,Ta92} if the thickness of $\Omega_\varepsilon$ is $\varepsilon$ (i.e. $g\equiv1$) and we impose the perfect slip boundary conditions \eqref{E:Slip_Intro} with $\gamma_\varepsilon=0$.
We note that the last paper~\cite{Miu_NSCTD_03} gives the first result on a rigorous derivation of the surface Navier--Stokes equations on a general closed surface in $\mathbb{R}^3$ by the thin-film limit.

\subsection{Idea of the proof} \label{SS:Int_Ide}
To prove the global existence of a strong solution to \eqref{E:NS_CTD} we argue by a standard energy method as in the case of flat thin domains studied in~\cites{Ho10,HoSe10,IfRaSe07}.
In the first part~\cite{Miu_NSCTD_01} of our study we investigated the Stokes operator $A_\varepsilon$ for $\Omega_\varepsilon$ under the slip boundary conditions \eqref{E:Slip_Intro} and proved the uniform norm equivalence
\begin{align} \label{E:NE_Intro}
  c^{-1}\|u\|_{H^k(\Omega_\varepsilon)} \leq \|A_\varepsilon^{k/2}u\|_{L^2(\Omega_\varepsilon)} \leq c\|u\|_{H^k(\Omega_\varepsilon)}, \quad u\in D(A_\varepsilon^{k/2}),\, k=1,2
\end{align}
and the uniform difference estimate for $A_\varepsilon$ and $-\nu\Delta$ of the form
\begin{align} \label{E:Diff_Intro}
  \|A_\varepsilon u+\nu\Delta u\|_{L^2(\Omega_\varepsilon)} \leq c\|u\|_{H^1(\Omega_\varepsilon)}, \quad u\in D(A_\varepsilon)
\end{align}
with a constant independent of $\varepsilon$ (see Section~\ref{S:St_Op}).
Using these estimates and average operators in the thin direction introduced and studied in Section~\ref{S:Ave}, we show that the $L^2(\Omega_\varepsilon)$-norm of $A_\varepsilon^{1/2}u^\varepsilon$ for a strong solution $u^\varepsilon$ is bounded uniformly in time.
A key point is to apply the estimate for the trilinear term
\begin{multline} \label{E:Tri_Intro}
  \left|\bigl((u\cdot\nabla)u,A_\varepsilon u\bigr)_{L^2(\Omega_\varepsilon)}\right| \leq \left(\frac{1}{4}+d_1\varepsilon^{1/2}\|A_\varepsilon^{1/2}u\|_{L^2(\Omega_\varepsilon)}\right)\|A_\varepsilon u\|_{L^2(\Omega_\varepsilon)}^2 \\
  +d_2\left(\|u\|_{L^2(\Omega_\varepsilon)}^2\|A_\varepsilon^{1/2}u\|_{L^2(\Omega_\varepsilon)}^4+\varepsilon^{-1}\|u\|_{L^2(\Omega_\varepsilon)}^2\|A_\varepsilon^{1/2}u\|_{L^2(\Omega_\varepsilon)}^2\right)
\end{multline}
for $u\in D(A_\varepsilon)$ with constants $d_1,d_2>0$ independent of $\varepsilon$ (see Lemma~\ref{L:Tri_Est_A}).
Based on this estimate and the uniform Gronwall inequality (see Lemma~\ref{L:Uni_Gronwall}) we prove
\begin{align*}
  d_1\varepsilon^{1/2}\|A_\varepsilon^{1/2}u^\varepsilon(t)\|_{L^2(\Omega_\varepsilon)} < \frac{1}{4}
\end{align*}
for all $t\in[0,\infty)$ by contradiction, which implies the global existence of the strong solution $u^\varepsilon$ (see Section~\ref{S:GE} for details).
The proof of \eqref{E:Tri_Intro} relies on \eqref{E:NE_Intro}, \eqref{E:Diff_Intro}, and a good decomposition of $u$ into the average and residual parts.
Using the average operators and an extension of a vector field on $\Gamma$ to $\Omega_\varepsilon$ satisfying the impermeable boundary condition, i.e. the first condition of \eqref{E:Slip_Intro}, we decompose $u$ into the almost two-dimensional average part $u^a$ and the residual part $u^r$ satisfying the impermeable boundary condition, and show good estimates for them separately.
For the average part we derive the product estimate
\begin{align} \label{E:Prod_Intro}
  \bigl\|\,|u^a|\,\varphi\bigr\|_{L^2(\Omega_\varepsilon)} \leq c\varepsilon^{-1/2}\|\varphi\|_{L^2(\Omega_\varepsilon)}^{1/2}\|\varphi\|_{H^1(\Omega_\varepsilon)}^{1/2}\|u\|_{L^2(\Omega_\varepsilon)}^{1/2}\|u\|_{H^1(\Omega_\varepsilon)}^{1/2}
\end{align}
for $\varphi\in H^1(\Omega_\varepsilon)$ and a similar estimate for $\nabla u^a$ (see Lemma~\ref{L:Prod_Ua}).
Usually, this kind of estimate is valid only for a two-dimensional domain due to the dependence of the Gagliardo--Nirenberg inequality on the dimension of a domain.
In our case, however, since $\Omega_\varepsilon$ is close to the two-dimensional surface $\Gamma$, we can show a product estimate for a function on $\Omega_\varepsilon$ and that on $\Gamma$ similar to a two-dimensional one that implies \eqref{E:Prod_Intro} (see Lemma~\ref{L:Prod}).
For the residual part we prove the $L^\infty(\Omega_\varepsilon)$-estimate
\begin{align} \label{E:UrInf_Intro}
  \|u^r\|_{L^\infty(\Omega_\varepsilon)} \leq c\left(\varepsilon^{1/2}\|u\|_{H^2(\Omega_\varepsilon)}+\|u\|_{L^2(\Omega_\varepsilon)}^{1/2}\|u\|_{H^2(\Omega_\varepsilon)}^{1/2}\right)
\end{align}
in Lemma~\ref{L:Linf_Ur}.
This estimate follows from an anisotropic Agmon inequality on $\Omega_\varepsilon$ (see Lemma~\ref{L:Agmon}) and Poincar\'{e} type inequalities for $u^r$ and $\nabla u^r$ (see Lemmas~\ref{L:Po_Ur} and~\ref{L:Po_Grad_Ur}).
Here the impermeable boundary condition on $u^r$ plays a fundamental role in the proof of the Poincar\'{e} type inequality for $u^r$.
Also, we use the divergence-free and slip boundary conditions on the original vector field $u$ to estimate $\nabla u^r$.
In the proof of \eqref{E:Tri_Intro} we use the relations
\begin{align*}
  (u\cdot\nabla)u = \mathrm{curl}\,u\times u+\frac{1}{2}\nabla(|u|^2), \quad (\nabla(|u|^2),A_\varepsilon u)_{L^2(\Omega_\varepsilon)} = 0
\end{align*}
to split $\bigl((u\cdot\nabla)u,A_\varepsilon u\bigr)_{L^2(\Omega_\varepsilon)}=J_1+J_2+J_3$ into
\begin{align*}
  J_1 &= (\mathrm{curl}\,u\times u^r,A_\varepsilon u)_{L^2(\Omega_\varepsilon)}, \\
  J_2 &= (\mathrm{curl}\,u\times u^a,A_\varepsilon u+\nu\Delta u)_{L^2(\Omega_\varepsilon)}, \\
  J_3 &= (\mathrm{curl}\,u\times u^a,-\nu\Delta u)_{L^2(\Omega_\varepsilon)}
\end{align*}
and estimate $J_1$, $J_2$, and $J_3$ by using \eqref{E:NE_Intro}, \eqref{E:Diff_Intro}, \eqref{E:Prod_Intro}, \eqref{E:UrInf_Intro}, and other inequalities given in Sections~\ref{S:Tool}--\ref{S:Ave}.
Here the estimates for $J_1$ and $J_2$ are straightforward, but the estimate for $J_3$ is complicated and requires long calculations (see Section~\ref{S:Tri}).

\subsection{Literature overview} \label{SS:Int_Lit}
The Navier--Stokes equations in thin domains have been studied for a long time.
When a thin domain in $\mathbb{R}^3$ has a very small thickness, it can be seen as almost two-dimensional and we naturally expect to show the global existence of a strong solution to the Navier--Stokes equations in such a thin domain for large data.
Raugel and Sell~\cite{RaSe93} first studied this problem for a thin product domain $Q\times(0,\varepsilon)$ in $\mathbb{R}^3$ with a rectangle $Q$ and a sufficiently small $\varepsilon>0$ under the purely periodic or mixed Dirichlet-periodic boundary conditions.
Their approach was based on a scaling method developed by Hale and Raugel~\cites{HaRa92a,HaRa92b} in the study of damped wave and reaction-diffusion equations in thin domains: they established the global existence of a strong solution by dilating the thin domain $Q\times(0,\varepsilon)$ and analyzing scaled equations in the fixed domain $Q\times(0,1)$ as a perturbation of the two-dimensional Navier--Stokes equations.
The result of~\cite{RaSe93} was generalized by Temam and Ziane~\cite{TeZi96} to a thin product domain $\omega\times(0,\varepsilon)$ in $\mathbb{R}^3$ around a bounded domain $\omega$ in $\mathbb{R}^2$ under combinations of the Dirichlet, periodic, and Hodge boundary conditions.
Using an average operator in the thin direction instead of the scaling method, they worked in the actual thin domain to get the global existence of a strong solution.
Under suitable boundary conditions they also established the convergence as $\varepsilon\to0$ of the average in the thin direction of the strong solution to a solution of the two-dimensional Navier--Stokes equations in $\omega$.
We refer to~\cites{Hu07,If99,IfRa01,MoTeZi97,Mo99} and the references cited therein for further generalizations and improvements on the results of~\cites{RaSe93,TeZi96}.

The above cited papers studied the Navier--Stokes equations in flat thin domains which have flat top and bottom boundaries and shrink to domains in $\mathbb{R}^2$ as $\varepsilon\to0$.
However, it is important for applications to consider nonflat thin domains since they appear in many physical problems (see~\cite{Ra95} for examples of nonflat thin domains).
Temam and Ziane~\cite{TeZi97} first generalized the shape of a thin domain in the study of the Navier--Stokes equations.
They considered a thin spherical shell
\begin{align*}
  \{x\in\mathbb{R}^3 \mid a < |x| < a+\varepsilon a\}, \quad a > 0
\end{align*}
under the Hodge boundary conditions and proved the global existence of a strong solution and the convergence of its average towards a solution to limit equations on a sphere as $\varepsilon\to0$.
Also, a flat thin domain with nonflat top and bottom boundaries
\begin{align*}
  \{(x',x_3)\in\mathbb{R}^3 \mid x'\in(0,1)^2,\, \varepsilon g_0(x') < x_3 < \varepsilon g_1(x')\}, \quad g_0,g_1\colon(0,1)^2\to\mathbb{R}
\end{align*}
was studied by Iftimie, Raugel, and Sell~\cite{IfRaSe07} (with $g_0\equiv0$), Hoang~\cites{Ho10}, and Hoang and Sell~\cite{HoSe10}.
Under the laterally periodic and vertically slip boundary conditions, they proved the global existence of a strong solution by using average operators.
The authors of~\cite{IfRaSe07} also compared a solution of the original equations with that of limit equations in $(0,1)^2$.
We refer to~\cite{Ho13} for the study of two-phase flows in a flat thin domain with nonflat top and bottom boundaries.

In the series of this paper and~\cites{Miu_NSCTD_01,Miu_NSCTD_03} we deal with the curved thin domain $\Omega_\varepsilon$ of the form \eqref{E:Def_CTD}.
Our thin domain has a nonconstant thickness like the flat thin domain studied in~\cites{Ho10,HoSe10,IfRaSe07}.
Moreover, its limit set $\Gamma$ is a general closed surface in $\mathbb{R}^3$ including a sphere considered in~\cite{TeZi97} which may have nonconstant curvatures.
There are several works on the asymptotic behavior of eigenvalues of the Laplace operator on a curved thin domain around a hypersurface (see e.g.~\cites{JiKu16,Kr14,Sch96,Yac18}).
Also, curved thin domains around lower dimensional manifolds were considered in the study of reaction-diffusion equations~\cites{PrRiRy02,PrRy03,Yan90}.
However, the Navier--Stokes equations in a curved thin domain in $\mathbb{R}^3$ around a general closed surface have not been studied due to difficulties in analyzing vector fields on and surface quantities of the boundary of the curved thin domain which has a complicated geometry.
Our study is aimed at giving new methods for overcoming such difficulties.

\subsection{Organization of this paper} \label{SS:Int_Org}
The rest of this paper is organized as follows.
In Section~\ref{S:Main} we present the main results of this paper on the global existence and estimates of a strong solution to \eqref{E:NS_CTD}.
Section~\ref{S:Pre} provides notations and basic results on a closed surface and a curved thin domain.
Fundamental tools for the analysis of vector fields on the curved thin domain are given in Section~\ref{S:Tool}.
In Section~\ref{S:St_Op} we summarize the main results of our first paper~\cite{Miu_NSCTD_01} on the Stokes operator for the curved thin domain under the slip boundary conditions.
Section~\ref{S:Ave} is devoted to a detailed study of average operators in the thin direction.
The main purpose of that section is to give a good decomposition of a vector field on the curved thin domain into the average and residual parts with useful estimates.
In Section~\ref{S:Tri} we show the good estimate \eqref{E:Tri_Intro} for the trilinear term by using the estimates for the Stokes and average operators given in Sections~\ref{S:St_Op} and~\ref{S:Ave}.
Based on that estimate we prove the main results of this paper in Section~\ref{S:GE}.
In Appendix~\ref{S:Ap_Vec} we fix notations on vectors and matrices.
Appendix~\ref{S:Ap_Proofs} provides the proofs of lemmas given in Section~\ref{SS:Tool_Sob} and Lemma~\ref{L:Tan_Curl_Ua}.

Most results of this paper were obtained in the doctoral thesis of the author~\cite{Miu_DT}.
However, we add the new condition (A3) in Assumption~\ref{Assump_2} to consider some curved thin domains excluded in~\cite{Miu_DT} by showing new results on a uniform Korn inequality and the axial symmetry of a curved thin domain in the first part~\cite{Miu_NSCTD_01} of our study.
In particular, we can deal with the thin spherical shell
\begin{align*}
  \Omega_\varepsilon = \{x\in\mathbb{R}^3 \mid 1 < |x| < 1+\varepsilon\}
\end{align*}
under the perfect slip boundary conditions \eqref{E:Slip_Intro} with $\gamma_\varepsilon=0$ in this paper which was not considered in~\cite{Miu_DT}.
This kind of curved thin domain was studied by Temam and Ziane~\cite{TeZi97} under different boundary conditions (see Remark~\ref{R:Ex_Assump}).

\section{Main results} \label{S:Main}
In this section we fix some notations and make assumptions, and state the main results of this paper (see also Section~\ref{S:Pre} for notations).
The proofs of theorems in this section are presented in Section~\ref{S:GE}.

Let $\Gamma$ be a closed (i.e. compact and without boundary), connected, and oriented surface in $\mathbb{R}^3$ with unit outward normal vector field $n$.
We assume that $\Gamma$ is of class $C^5$ and the functions $g_0,g_1\in C^4(\Gamma)$ satisfy
\begin{align} \label{E:G_Inf}
  g := g_1-g_0 \geq c \quad\text{on}\quad \Gamma
\end{align}
with some constant $c>0$.
Note that we do not assume $g_0\leq0$ or $g_1\geq0$ on $\Gamma$.
For a sufficiently small $\varepsilon\in(0,1]$ we define the curved thin domain $\Omega_\varepsilon$ in $\mathbb{R}^3$ by \eqref{E:Def_CTD} and the standard $L^2$-solenoidal space on $\Omega_\varepsilon$ by
\begin{align*}
  L_\sigma^2(\Omega_\varepsilon) := \{u\in L^2(\Omega_\varepsilon)^3 \mid \text{$\mathrm{div}\,u=0$ in $\Omega_\varepsilon$, $u\cdot n_\varepsilon=0$ on $\Gamma_\varepsilon$}\}.
\end{align*}
To prove the global existence of a strong solution to \eqref{E:NS_CTD} we consider an abstract formulation for \eqref{E:NS_CTD} in an appropriate function space on $\Omega_\varepsilon$.
For this purpose, we fix notations and make assumptions as follows.

Let $\mathcal{R}$ be the space of all infinitesimal rigid displacements of $\mathbb{R}^3$ whose restrictions on $\Gamma$ are tangential, i.e.
\begin{align} \label{E:Def_R}
  \mathcal{R} := \{w(x)=a\times x+b,\,x\in\mathbb{R}^3 \mid a,b\in\mathbb{R}^3,\,\text{$w|_\Gamma\cdot n=0$ on $\Gamma$}\}.
\end{align}
Clearly, the space $\mathcal{R}$ is of finite dimension.
It describes the axial symmetry of the closed surface $\Gamma$, i.e. $\mathcal{R}\neq\{0\}$ if and only if $\Gamma$ is invariant under a rotation by any angle around some line (see~\cite{Miu_NSCTD_01}*{Lemma~E.1}).
Let $\nabla_\Gamma$ be the tangential gradient operator on $\Gamma$ (see Section~\ref{SS:Pre_Surf} for the definition) and
\begin{align} \label{E:Def_Rg}
  \begin{aligned}
    \mathcal{R}_i &:= \{w\in\mathcal{R} \mid \text{$w|_\Gamma\cdot\nabla_\Gamma g_i=0$ on $\Gamma$}\}, \quad i=0,1, \\
    \mathcal{R}_g &:= \{w\in\mathcal{R} \mid \text{$w|_\Gamma\cdot\nabla_\Gamma g=0$ on $\Gamma$}\} \quad (g=g_1-g_0).
  \end{aligned}
\end{align}
By definition, $\mathcal{R}_0\cap\mathcal{R}_1\subset\mathcal{R}_g$.
These spaces are related to the uniform axial symmetry and asymmetry of $\Omega_\varepsilon$: if $\mathcal{R}_0\cap\mathcal{R}_1\neq\{0\}$ then $\Omega_\varepsilon$ is axially symmetric around the same line for all $\varepsilon\in(0,1]$, while it is not axially symmetric around any line for all $\varepsilon>0$ sufficiently small if $\mathcal{R}_g=\{0\}$ (see~\cite{Miu_NSCTD_01}*{Lemmas~E.6 and~E.7}).

Next let $P:=I_3-n\otimes n$ be the orthogonal projection onto the tangent plane of $\Gamma$.
For a vector field $v\colon\Gamma\to\mathbb{R}^3$ we set
\begin{align*}
  (\nabla_\Gamma v)_S := \frac{\nabla_\Gamma v+(\nabla_\Gamma v)^T}{2}, \quad D_\Gamma(v) := P(\nabla_\Gamma v)_SP \quad\text{on}\quad \Gamma
\end{align*}
and define function spaces for tangential vector fields on $\Gamma$ by
\begin{align} \label{E:Def_Kil}
  \begin{aligned}
    \mathcal{K}(\Gamma) &:= \{v \in H^1(\Gamma)^3 \mid \text{$v\cdot n=0$, $D_\Gamma(v)=0$ on $\Gamma$}\}, \\
    \mathcal{K}_g(\Gamma) &:= \{v\in\mathcal{K}(\Gamma) \mid \text{$v\cdot\nabla_\Gamma g=0$ on $\Gamma$}\}.
  \end{aligned}
\end{align}
If $v\in\mathcal{K}(\Gamma)$, then for all tangential vector fields $X$ and $Y$ on $\Gamma$ we have
\begin{align*}
  \overline{\nabla}_Xv\cdot Y+X\cdot\overline{\nabla}_Yv = 0 \quad\text{on}\quad \Gamma,
\end{align*}
where $\overline{\nabla}_Xv:=P(X\cdot\nabla_\Gamma)v$ is the covariant derivative of $v$ along $X$.
Such a vector field is known as a Killing vector field on $\Gamma$ that generates a one-parameter group of isometries of $\Gamma$ (see~\cites{Jo11,Pe06} for details).
By direct calculations we see that
\begin{align*}
  \mathcal{R}|_\Gamma := \{w|_\Gamma \mid w\in\mathcal{R}\} \subset \mathcal{K}(\Gamma).
\end{align*}
The sets $\mathcal{K}(\Gamma)$ and $\mathcal{R}|_\Gamma$ represent the intrinsic and extrinsic infinitesimal symmetry of $\Gamma$, respectively.
If $\Gamma$ is axially symmetric then $\mathcal{R}|_\Gamma=\mathcal{K}(\Gamma)$ (see~\cite{Miu_NSCTD_01}*{Lemma~E.3}).
The same relation is valid if $\Gamma$ is closed and convex by the rigidity theorem of Cohn-Vossen (see~\cite{Sp79}).
However, it is not known whether $\mathcal{R}|_\Gamma$ agrees with $\mathcal{K}(\Gamma)$ for closed but nonconvex and not axially symmetric surfaces.

In the rest of this section and Sections~\ref{S:St_Op},~\ref{S:Tri}, and~\ref{S:GE} we make the following assumptions on the friction coefficients $\gamma_\varepsilon^0$ and $\gamma_\varepsilon^1$ appearing in \eqref{E:Def_Fric}, the closed surface $\Gamma$, and the functions $g_0$ and $g_1$.

\begin{assumption} \label{Assump_1}
  There exists a constant $c>0$ such that
  \begin{align} \label{E:Fric_Upper}
    \gamma_\varepsilon^0 \leq c\varepsilon, \quad \gamma_\varepsilon^1 \leq c\varepsilon
  \end{align}
  for all $\varepsilon\in(0,1]$.
\end{assumption}

\begin{assumption} \label{Assump_2}
  Either of the following conditions is satisfied:
  \begin{itemize}
    \item[(A1)] There exists a constant $c>0$ such that
    \begin{align*}
      \gamma_\varepsilon^0 \geq c\varepsilon \quad\text{for all}\quad \varepsilon\in(0,1] \quad\text{or}\quad \gamma_\varepsilon^1 \geq c\varepsilon \quad\text{for all}\quad \varepsilon\in(0,1].
    \end{align*}
    \item[(A2)] The space $\mathcal{K}_g(\Gamma)$ contains only a trivial vector field, i.e. $\mathcal{K}_g(\Gamma)=\{0\}$.
    \item[(A3)] The relations
    \begin{align*}
      \mathcal{R}_g=\mathcal{R}_0\cap\mathcal{R}_1, \quad \mathcal{R}_g|_\Gamma:=\{w|_\Gamma\mid w\in\mathcal{R}_g\}=\mathcal{K}_g(\Gamma)
    \end{align*}
    hold and $\gamma_\varepsilon^0=\gamma_\varepsilon^1=0$ for all $\varepsilon\in(0,1]$.
  \end{itemize}
\end{assumption}

\begin{remark} \label{R:Ex_Assump}
  The condition (A2) or (A3) is satisfied in the following examples:
  \begin{itemize}
    \item It is known (see e.g.~\cite{Sh_18pre}*{Proposition~2.2}) that $\mathcal{K}(\Gamma)=\{0\}$, i.e. there exists no nontrivial Killing vector field on $\Gamma$ if the genus of $\Gamma$ is greater than one.
    In this case, the condition (A2) is satisfied for any $g=g_1-g_0$.
    \item The condition (A3) is satisfied when $\Omega_\varepsilon$ is a thin spherical shell
    \begin{align*}
      \Omega_\varepsilon = \{x\in\mathbb{R}^3 \mid 1<|x|<1+\varepsilon\} \quad (\Gamma = S^2,\, g_0 \equiv 0,\, g_1 \equiv 1)
    \end{align*}
    around the unit sphere $S^2$ in $\mathbb{R}^3$ and the perfect slip boundary conditions
    \begin{align} \label{E:Per_Slip}
      u\cdot n_\varepsilon = 0, \quad 2\nu P_\varepsilon D(u)n_\varepsilon = 0 \quad\text{on}\quad \Gamma_\varepsilon
    \end{align}
    are imposed.
    The Navier--Stokes equations in this kind of thin domain were studied by Temam and Ziane~\cite{TeZi97} under the Hodge boundary conditions
    \begin{align} \label{E:Hodge_BC}
      u\cdot n_\varepsilon = 0, \quad \mathrm{curl}\,u\times n_\varepsilon = 0 \quad\text{on}\quad \Gamma_\varepsilon.
    \end{align}
    Note that the perfect slip boundary conditions \eqref{E:Per_Slip} are different from the Hodge boundary conditions \eqref{E:Hodge_BC} unless the boundary is flat.
    Indeed, if $u$ satisfies $u\cdot n_\varepsilon=0$ on $\Gamma_\varepsilon$, then (see~\cite{MitMon09}*{Section~2} and~\cite{Miu_NSCTD_01}*{Lemma~B.10})
    \begin{align*}
      2P_\varepsilon D(u)n_\varepsilon-\mathrm{curl}\,u\times n_\varepsilon = 2W_\varepsilon u \quad\text{on}\quad \Gamma_\varepsilon,
    \end{align*}
    where $W_\varepsilon$ is the Weingarten map (or the shape operator) of the boundary $\Gamma_\varepsilon$ (see Section~\ref{SS:Pre_CTD}) that represents the curvatures of $\Gamma_\varepsilon$.
  \end{itemize}
  For further discussions on Assumption~\ref{Assump_2} we refer to~\cite{Miu_NSCTD_01}*{Remarks~2.9 and~2.10}.
\end{remark}

Under Assumptions~\ref{Assump_1} and~\ref{Assump_2} we define subspaces of $L^2(\Omega_\varepsilon)^3$ and $H^1(\Omega_\varepsilon)^3$ by
\begin{align} \label{E:Def_Heps}
  \begin{aligned}
    \mathcal{H}_\varepsilon &:=
    \begin{cases}
      L_\sigma^2(\Omega_\varepsilon) &\text{if the condition (A1) or (A2) is satisfied}, \\
      L_\sigma^2(\Omega_\varepsilon)\cap\mathcal{R}_g^\perp &\text{if the condition (A3) is satisfied},
  \end{cases} \\
  \mathcal{V}_\varepsilon &:= \mathcal{H}_\varepsilon\cap H^1(\Omega_\varepsilon)^3,
    \end{aligned}
\end{align}
where $\mathcal{R}_g^\perp$ is the orthogonal complement of $\mathcal{R}_g$ in $L^2(\Omega_\varepsilon)^3$.
Here we consider vector fields in $\mathcal{R}_g$, which are defined on $\mathbb{R}^3$, in $L^2(\Omega_\varepsilon)^3$ just by restricting them on $\overline{\Omega}_\varepsilon$.
It is shown in~\cite{Miu_NSCTD_01}*{Lemma~E.8} that $\mathcal{R}_0\cap\mathcal{R}_1\subset L_\sigma^2(\Omega_\varepsilon)$.
Thus $\mathcal{R}_g\subset L_\sigma^2(\Omega_\varepsilon)$ under the condition (A3).
Moreover, $\mathcal{H}_\varepsilon$ and $\mathcal{V}_\varepsilon$ are closed in $L^2(\Omega_\varepsilon)^3$ and $H^1(\Omega_\varepsilon)^3$.
We denote by $\mathbb{P}_\varepsilon$ the orthogonal projection from $L^2(\Omega_\varepsilon)^3$ onto $\mathcal{H}_\varepsilon$.

As we will see below, the function space $\mathcal{H}_\varepsilon$ defined by \eqref{E:Def_Heps} gives an appropriate abstract formulation for \eqref{E:NS_CTD}.
The bilinear form for the Stokes problem in $\Omega_\varepsilon$ under the slip boundary conditions is of the form
\begin{align*}
  a_\varepsilon(u_1,u_2) := 2\nu\bigl(D(u_1),D(u_2)\bigr)_{L^2(\Omega_\varepsilon)}+\sum_{i=0,1}\gamma_\varepsilon^i(u_1,u_2)_{L^2(\Gamma_\varepsilon^i)}
\end{align*}
for $u_1,u_2\in H^1(\Omega_\varepsilon)^3$ (see Section~\ref{S:St_Op}).
In the first part~\cite{Miu_NSCTD_01} of our study we proved that $a_\varepsilon$ is bounded and coercive on $\mathcal{V}_\varepsilon$ uniformly in $\varepsilon$.

\begin{lemma}[{\cite{Miu_NSCTD_01}*{Theorem~2.4}}] \label{L:Uni_aeps}
  Under Assumptions~\ref{Assump_1} and~\ref{Assump_2}, there exist constants $\varepsilon_0\in(0,1]$ and $c>0$ such that
  \begin{align} \label{E:Uni_aeps}
    c^{-1}\|u\|_{H^1(\Omega_\varepsilon)}^2 \leq a_\varepsilon(u,u) \leq c\|u\|_{H^1(\Omega_\varepsilon)}^2
  \end{align}
  for all $\varepsilon\in(0,\varepsilon_0]$ and $u\in\mathcal{V}_\varepsilon$.
\end{lemma}

By Lemma~\ref{L:Uni_aeps} and the Lax--Milgram theory we see that $a_\varepsilon$ induces a bounded linear operator $A_\varepsilon$ from $\mathcal{V}_\varepsilon$ into its dual space $\mathcal{V}_\varepsilon'$.
We consider $A_\varepsilon$ as an unbounded operator on $\mathcal{H}_\varepsilon$ with domain $D(A_\varepsilon)$ and call it the Stokes operator for $\Omega_\varepsilon$ under the slip boundary conditions or simply the Stokes operator on $\mathcal{H}_\varepsilon$.
In Section~\ref{S:St_Op} we summarize the fundamental results on $A_\varepsilon$ established in~\cite{Miu_NSCTD_01}.

\begin{remark} \label{R:Assump}
  We impose Assumptions~\ref{Assump_1} and~\ref{Assump_2} in Sections~\ref{S:St_Op},~\ref{S:Tri}, and~\ref{S:GE} to employ the Stokes operator $A_\varepsilon$, but we do not use these assumptions explicitly in this paper except for the inequalities \eqref{E:Fric_Upper}.
  Also, the $C^5$-regularity of $\Gamma$ and the $C^4$-regularity of $g_0$ and $g_1$ on $\Gamma$ are not used explicitly in this paper since they are required just for the proof of the uniform norm equivalence for $A_\varepsilon$ (see Lemma~\ref{L:Stokes_H2}) given in the first part of our study (see~\cite{Miu_NSCTD_01}*{Section~6}).
\end{remark}

Now let us present the main results of this paper.
With the above notations we consider the abstract formulation for \eqref{E:NS_CTD} in $\mathcal{H}_\varepsilon$:
\begin{align} \label{E:NS_Abst}
  \partial_tu^\varepsilon+A_\varepsilon u^\varepsilon+\mathbb{P}_\varepsilon(u^\varepsilon\cdot\nabla)u^\varepsilon = \mathbb{P}_\varepsilon f^\varepsilon \quad\text{on}\quad (0,\infty), \quad u^\varepsilon|_{t=0} = u_0^\varepsilon.
\end{align}
We refer to~\cites{BoFa13,CoFo88,So01,Te79} and the references cited therein for the study of the abstract evolution equation \eqref{E:NS_Abst}.
For a vector field $u$ on $\Omega_\varepsilon$ we define the average of $u$ in the thin direction and its tangential component by
\begin{align*}
  Mu(y) := \frac{1}{\varepsilon g(y)}\int_{\varepsilon g_0(y)}^{\varepsilon g_1(y)} u(y+rn(y))\,dr, \quad M_\tau u(y) := P(y)Mu(y)
\end{align*}
for $y\in\Gamma$ (see Section~\ref{S:Ave}).
Also, for the Hilbert space
\begin{align*}
  H^1(\Gamma,T\Gamma) := \{v\in H^1(\Gamma)^3 \mid \text{$v\cdot n=0$ on $\Gamma$}\}
\end{align*}
we denote by $H^{-1}(\Gamma,T\Gamma)$ its dual space (see Section~\ref{SS:Pre_Surf}).
First we give the global existence of a strong solution to \eqref{E:NS_CTD} for large data.

\begin{theorem} \label{T:GE}
  Under Assumptions~\ref{Assump_1} and~\ref{Assump_2}, let $\varepsilon_0$ be the constant given in Lemma~\ref{L:Uni_aeps}.
  There exists a constant $c_0\in(0,1]$ such that the following statement holds: for each $\varepsilon\in(0,\varepsilon_0]$ suppose that the given data
  \begin{align*}
    u_0^\varepsilon\in\mathcal{V}_\varepsilon, \quad f^\varepsilon\in L^\infty(0,\infty;L^2(\Omega_\varepsilon)^3)
  \end{align*}
  satisfy
  \begin{multline} \label{E:GE_Data}
    \|u_0^\varepsilon\|_{H^1(\Omega_\varepsilon)}^2+\|\mathbb{P}_\varepsilon f^\varepsilon\|_{L^\infty(0,\infty;L^2(\Omega_\varepsilon))}^2 \\
    +\|M_\tau u_0^\varepsilon\|_{L^2(\Gamma)}^2+\|M_\tau\mathbb{P}_\varepsilon f^\varepsilon\|_{L^\infty(0,T;H^{-1}(\Gamma,T\Gamma))}^2 \leq c_0\varepsilon^{-1}.
  \end{multline}
  If the condition (A3) of Assumption~\ref{Assump_2} is imposed, suppose further that $f^\varepsilon(t)\in\mathcal{R}_g^\perp$ for a.a. $t\in(0,\infty)$.
  Then there exists a global-in-time strong solution
  \begin{align*}
    u^\varepsilon \in C([0,\infty);\mathcal{V}_\varepsilon)\cap L_{loc}^2([0,\infty);D(A_\varepsilon))\cap H_{loc}^1([0,\infty);\mathcal{H}_\varepsilon)
  \end{align*}
  to the Navier--Stokes equations \eqref{E:NS_CTD}.
\end{theorem}

We also establish several estimates for a strong solution to \eqref{E:NS_CTD} with constants explicitly depending on $\varepsilon$, which are fundamental for the study of the thin-film limit for \eqref{E:NS_CTD} carried out in the last part~\cite{Miu_NSCTD_03} of our study.

\begin{theorem} \label{T:UE}
  Under Assumptions~\ref{Assump_1} and~\ref{Assump_2}, let $\varepsilon_0$ be the constant given in Lemma~\ref{L:Uni_aeps}.
  Also, let $c_1$, $c_2$, $\alpha$, and $\beta$ be positive constants.
  Then there exists a constant $\varepsilon_1\in(0,\varepsilon_0]$ such that the following statement holds: for $\varepsilon\in(0,\varepsilon_1]$ suppose that the given data
  \begin{align*}
    u_0^\varepsilon\in \mathcal{V}_\varepsilon, \quad f^\varepsilon\in L^\infty(0,\infty;L^2(\Omega_\varepsilon)^3)
  \end{align*}
  satisfy
  \begin{align} \label{E:UE_Data}
    \begin{aligned}
      \|u_0^\varepsilon\|_{H^1(\Omega_\varepsilon)}^2+\|\mathbb{P}_\varepsilon f^\varepsilon\|_{L^\infty(0,\infty;L^2(\Omega_\varepsilon))}^2 &\leq c_1\varepsilon^{-1+\alpha}, \\
      \|M_\tau u_0^\varepsilon\|_{L^2(\Gamma)}^2+\|M_\tau\mathbb{P}_\varepsilon f^\varepsilon\|_{L^\infty(0,\infty;H^{-1}(\Gamma,T\Gamma))}^2 &\leq c_2\varepsilon^{-1+\beta}.
    \end{aligned}
  \end{align}
  If the condition (A3) of Assumption~\ref{Assump_2} is imposed, suppose further that $f^\varepsilon(t)\in\mathcal{R}_g^\perp$ for a.a. $t\in(0,\infty)$.
  Then there exists a global-in-time strong solution
  \begin{align*}
    u^\varepsilon \in C([0,\infty);\mathcal{V}_\varepsilon)\cap L_{loc}^2([0,\infty);D(A_\varepsilon))\cap H_{loc}^1([0,\infty);\mathcal{H}_\varepsilon)
  \end{align*}
  to \eqref{E:NS_CTD}.
  Moreover, there exists a constant $c>0$ independent of $\varepsilon$ and $u^\varepsilon$ such that
  \begin{align} \label{E:UE_L2}
    \begin{aligned}
      \|u^\varepsilon(t)\|_{L^2(\Omega_\varepsilon)}^2 &\leq c(\varepsilon^{1+\alpha}+\varepsilon^\beta), \\
      \int_0^t\|u^\varepsilon(s)\|_{H^1(\Omega_\varepsilon)}^2\,ds &\leq c(\varepsilon^{1+\alpha}+\varepsilon^\beta)(1+t)
    \end{aligned}
  \end{align}
  for all $t\geq0$ and
  \begin{align} \label{E:UE_H1}
    \begin{aligned}
      \|u^\varepsilon(t)\|_{H^1(\Omega_\varepsilon)}^2 &\leq c(\varepsilon^{-1+\alpha}+\varepsilon^{-1+\beta}), \\
      \int_0^t\|u^\varepsilon(s)\|_{H^2(\Omega_\varepsilon)}^2\,ds &\leq c(\varepsilon^{-1+\alpha}+\varepsilon^{-1+\beta})(1+t)
    \end{aligned}
  \end{align}
  for all $t\geq 0$.
\end{theorem}

\begin{remark} \label{R:Ext_Orth}
  The assumption $f^\varepsilon(t)\in\mathcal{R}_g^\perp$ for a.a. $t\in(0,\infty)$ under the condition (A3) is necessary in order to recover the momentum equations of the original problem \eqref{E:NS_CTD} properly from the abstract evolution equation \eqref{E:NS_Abst}.
  When the condition (A3) is imposed, Theorem~\ref{T:GE} a priori provides a global solution
  \begin{align*}
    u^\varepsilon \in C([0,\infty);\mathcal{V}_\varepsilon)\cap L_{loc}^2([0,\infty);D(A_\varepsilon))\cap H_{loc}^1([0,\infty);\mathcal{H}_\varepsilon)
  \end{align*}
  to the abstract evolution equation \eqref{E:NS_Abst} in $\mathcal{H}_\varepsilon=L_\sigma^2(\Omega_\varepsilon)\cap\mathcal{R}_g^\perp$.
  The function space $\mathcal{R}_g$ is of finite dimension and thus closed in $L^2(\Omega_\varepsilon)^3$.
  Moreover, under the condition (A3) we see that $\mathcal{R}_g=\mathcal{R}_0\cap\mathcal{R}_1$ is contained in $L_\sigma^2(\Omega_\varepsilon)$ by~\cite{Miu_NSCTD_01}*{Lemma~E.8}.
  Hence we have the orthogonal decomposition
  \begin{align*}
    L^2(\Omega_\varepsilon)^3 = \mathcal{H}_\varepsilon\oplus\mathcal{R}_g\oplus G^2(\Omega_\varepsilon) \quad (L_\sigma^2(\Omega_\varepsilon) = \mathcal{H}_\varepsilon\oplus\mathcal{R}_g)
  \end{align*}
  with $G^2(\Omega_\varepsilon)=\{\nabla p \mid p\in H^1(\Omega_\varepsilon)\}$.
  By this decomposition we find that the partial differential equations in $L^2(\Omega_\varepsilon)^3$ recovered from \eqref{E:NS_Abst} in $\mathcal{H}_\varepsilon$ are of the form
  \begin{align} \label{E:NS_Abst_L2}
    \partial_tu^\varepsilon(t)-\nu\Delta u^\varepsilon(t)+[(u^\varepsilon\cdot\nabla)u^\varepsilon](t)+\nabla p^\varepsilon(t)+w(t) = f^\varepsilon(t)
  \end{align}
  for a.a. $t\in(0,\infty)$ with $\nabla p^\varepsilon(t)\in G^2(\Omega_\varepsilon)$ and $w(t)\in\mathcal{R}_g$.
  These equations contain the additional vector field $w(t)$, but we can show $w(t)=0$ under the assumption $f^\varepsilon(t)\in\mathcal{R}_g^\perp$.
  Indeed, we take the $L^2(\Omega_\varepsilon)$-inner product of \eqref{E:NS_Abst_L2} with $w$ (here and hereafter we fix and suppress $t$) and use the fact that $\partial_tu^\varepsilon\in\mathcal{H}_\varepsilon$, $\nabla p^\varepsilon\in G^2(\Omega_\varepsilon)$, and $f^\varepsilon$ are orthogonal to $w\in\mathcal{R}_g$ to obtain
    \begin{align*}
    \|w\|_{L^2(\Omega_\varepsilon)}^2 = \nu(\Delta u^\varepsilon,w)_{L^2(\Omega_\varepsilon)}-\bigl((u^\varepsilon\cdot\nabla)u^\varepsilon,w\bigr)_{L^2(\Omega_\varepsilon)}.
  \end{align*}
  Noting that we impose the perfect slip boundary conditions \eqref{E:Per_Slip} on $u^\varepsilon$ under the condition (A3), we carry out integration by parts (see \eqref{E:IbP_St}) and use $\mathrm{div}\,u^\varepsilon=0$ in $\Omega_\varepsilon$ and $w\cdot n_\varepsilon=0$ on $\Gamma_\varepsilon$ (note that $w\in\mathcal{R}_g=\mathcal{R}_0\cap\mathcal{R}_1\subset L_\sigma^2(\Omega_\varepsilon)$) to get
  \begin{align*}
    (\Delta u^\varepsilon,w)_{L^2(\Omega_\varepsilon)} = -2\bigl(D(u^\varepsilon),D(w)\bigl)_{L^2(\Omega_\varepsilon)}.
  \end{align*}
  Moreover, since $w\in\mathcal{R}_g$ is of the form $w(x,t)=a(t)\times x+b(t)$ for $x\in\mathbb{R}^3$ with $a(t),b(t)\in\mathbb{R}^3$ independent of $x$, it follows that $D(w)=0$ in $\mathbb{R}^3$ and thus the above equality yields $(\Delta u^\varepsilon,w)_{L^2(\Omega_\varepsilon)}=0$.
  Also, by integration by parts and $u^\varepsilon\in L_\sigma^2(\Omega_\varepsilon)$,
  \begin{align*}
    \bigl((u^\varepsilon\cdot\nabla)u^\varepsilon,w\bigr)_{L^2(\Omega_\varepsilon)} = -(u^\varepsilon,(u^\varepsilon\cdot\nabla)w)_{L^2(\Omega_\varepsilon)} = -(u^\varepsilon,a\times u^\varepsilon)_{L^2(\Omega_\varepsilon)} = 0.
  \end{align*}
  Thus $\|w\|_{L^2(\Omega_\varepsilon)}^2=0$, i.e. $w=0$ in \eqref{E:NS_Abst_L2} and the momentum equations of \eqref{E:NS_CTD} are properly recovered from the abstract evolution equation \eqref{E:NS_Abst}.
\end{remark}

\begin{remark} \label{R:GE_Prev}
  Formally, when $\Gamma=\mathbb{T}^2$ is the flat torus, the function spaces given by \eqref{E:Def_Rg} and \eqref{E:Def_Kil} are of the form
  \begin{align*}
    \mathcal{R}_i &= \{(a_1,a_2,0)^T\in\mathbb{R}^2\times\{0\} \mid \text{$a_1\partial_1g_i+a_2\partial_2g_i=0$ on $\mathbb{T}^2$}\}, \quad i=0,1, \\
    \mathcal{R}_g &= \mathcal{K}_g(\Gamma) = \{(a_1,a_2,0)^T\in\mathbb{R}^2\times\{0\} \mid \text{$a_1\partial_1g+a_2\partial_2g=0$ on $\mathbb{T}^2$}\}.
  \end{align*}
  These function spaces appear in the study of the Navier--Stokes equations in flat thin domains around the flat torus~\cites{Ho10,HoSe10,IfRaSe07}.
  In~\cite{Ho10} and~\cites{HoSe10,IfRaSe07} the global existence of a strong solution was established under the conditions (A2) and (A3), respectively, and assumptions on the data $u_0^\varepsilon$ and $f^\varepsilon$ similar to those in Theorem~\ref{T:GE}.
  Therefore, Theorem~\ref{T:GE} generalizes the existence results of~\cites{Ho10,HoSe10,IfRaSe07} to the curved thin domain $\Omega_\varepsilon$ around the general closed surface $\Gamma$.
\end{remark}

\section{Preliminaries} \label{S:Pre}
We fix notations and give basic properties of a closed surface and a curved thin domain.
The proofs of lemmas in this section are given in the first part~\cite{Miu_NSCTD_01} of our study, so we omit them in this paper.

In what follows, we fix a coordinate system of $\mathbb{R}^3$ and write $x_i$, $i=1,2,3$ for the $i$-th component of a point $x\in\mathbb{R}^3$ under this coordinate system.
Also, we denote by $c$ a general positive constant independent of the parameter $\varepsilon$.
Notations on vectors and matrices used in this paper are presented in Appendix~\ref{S:Ap_Vec}.

\subsection{Closed surface} \label{SS:Pre_Surf}
Let $\Gamma$ be a two-dimensional closed, connected, and oriented surface in $\mathbb{R}^3$ of class $C^5$.
We write $n$ and $d$ for the unit outward normal vector field of $\Gamma$ and the signed distance function from $\Gamma$ increasing in the direction of $n$.
Also, by $\kappa_1$ and $\kappa_2$ we denote the principal curvatures of $\Gamma$.
By the $C^5$-regularity of $\Gamma$ we have $n\in C^4(\Gamma)^3$ and $\kappa_1,\kappa_2\in C^3(\Gamma)$, and $\kappa_1$ and $\kappa_2$ are bounded on the compact set $\Gamma$.
Hence there exists a tubular neighborhood
\begin{align*}
  N := \{x\in\mathbb{R}^3 \mid \mathrm{dist}(x,\Gamma) < \delta\}, \quad \delta > 0
\end{align*}
of $\Gamma$ such that for each $x\in N$ we have
\begin{align} \label{E:Nor_Coord}
  x = \pi(x)+d(x)n(\pi(x)), \quad \nabla d(x) = n(\pi(x)).
\end{align}
with a unique point $\pi(x)\in\Gamma$, and $d$ and $\pi$ are of class $C^5$ and $C^4$ on $\overline{N}$, respectively (see~\cite{GiTr01}*{Section~14.6}).
Also, by the boundedness of $\kappa_1$ and $\kappa_2$ we have
\begin{align} \label{E:Curv_Bound}
  c^{-1} \leq 1-r\kappa_i(y) \leq c \quad\text{for all}\quad y\in\Gamma,\,r\in(-\delta,\delta),\,i=1,2
\end{align}
if we take $\delta>0$ sufficiently small.

Let us give differential operators on $\Gamma$ and surface quantities of $\Gamma$.
We define the orthogonal projections onto the tangent plane and the normal direction of $\Gamma$ by
\begin{align*}
  P(y) := I_3-n(y)\otimes n(y), \quad Q(y) := n(y)\otimes n(y), \quad y\in\Gamma.
\end{align*}
They are of class $C^4$ on $\Gamma$ and satisfy $|P|=2$, $|Q|=1$, and
\begin{gather*}
  I_3 = P+Q, \quad PQ = QP = 0, \quad P^T = P^2 = P, \quad Q^T = Q^2 = Q, \\
  |a|^2 = |Pa|^2+|Qa|^2, \quad |Pa| \leq |a|, \quad Pa\cdot n = 0, \quad a\in\mathbb{R}^3
\end{gather*}
on $\Gamma$.
In the sequel we frequently use these relations (sometimes without mention).
For $\eta\in C^1(\Gamma)$ we define its tangential gradient and tangential derivatives by
\begin{align} \label{E:Def_TGr}
  \nabla_\Gamma\eta(y) := P(y)\nabla\tilde{\eta}(y), \quad \underline{D}_i\eta(y) := \sum_{j=1}^3P_{ij}(y)\partial_j\tilde{\eta}(y), \quad y\in\Gamma,\,i=1,2,3
\end{align}
so that $\nabla_\Gamma\eta=(\underline{D}_1\eta,\underline{D}_2\eta,\underline{D}_3\eta)^T$, where $\tilde{\eta}$ is a $C^1$-extension of $\eta$ to $N$ with $\tilde{\eta}|_\Gamma=\eta$.
By the definition of $\nabla_\Gamma\eta$ we immediately get
\begin{align} \label{E:P_TGr}
  P\nabla_\Gamma\eta = \nabla_\Gamma\eta, \quad n\cdot\nabla_\Gamma\eta = 0 \quad\text{on}\quad \Gamma.
\end{align}
Note that $\nabla_\Gamma\eta$ defined by \eqref{E:Def_TGr} agrees with the gradient on a Riemannian manifold expressed under a local coordinate system (see~\cite{Miu_NSCTD_01}*{Lemma~B.2}).
Hence the values of $\nabla_\Gamma\eta$ and $\underline{D}_i\eta$ are independent of the choice of an extension $\tilde{\eta}$.
In particular, the constant extension $\bar{\eta}:=\eta\circ\pi$ of $\eta$ in the normal direction of $\Gamma$ satisfies
\begin{align} \label{E:ConDer_Surf}
  \nabla\bar{\eta}(y) = \nabla_\Gamma\eta(y), \quad \partial_i\bar{\eta}(y) = \underline{D}_i\eta(y), \quad y\in\Gamma,\,i=1,2,3
\end{align}
since $\nabla\pi(y)=P(y)$ for $y\in\Gamma$ by \eqref{E:Nor_Coord} and $d(y)=0$.
Hereafter the notation $\bar{\eta}$ with an overline always means the constant extension of a function $\eta$ on $\Gamma$ in the normal direction of $\Gamma$.
When $\eta\in C^2(\Gamma)$ we denote its tangential Hessian matrix by
\begin{align*}
  \nabla_\Gamma^2\eta := (\underline{D}_i\underline{D}_j\eta)_{i,j} \quad\text{on}\quad \Gamma.
\end{align*}
For a (not necessarily tangential) vector field $v=(v_1,v_2,v_3)^T\in C^1(\Gamma)^3$ we define the tangential gradient matrix and the surface divergence of $v$ by
\begin{align*}
  \nabla_\Gamma v :=
  \begin{pmatrix}
    \underline{D}_1v_1 & \underline{D}_1v_2 & \underline{D}_1v_3 \\
    \underline{D}_2v_1 & \underline{D}_2v_2 & \underline{D}_2v_3 \\
    \underline{D}_3v_1 & \underline{D}_3v_2 & \underline{D}_3v_3
  \end{pmatrix}, \quad
  \mathrm{div}_\Gamma v := \mathrm{tr}[\nabla_\Gamma v] = \sum_{i=1}^3\underline{D}_iv_i \quad\text{on}\quad \Gamma.
\end{align*}
Then $\nabla_\Gamma v=P\nabla\tilde{v}$ on $\Gamma$ for any $C^1$-extension $\tilde{v}$ of $v$ to $N$ with $\tilde{v}|_\Gamma=v$.
We also set
\begin{align*}
  |\nabla_\Gamma^2v|^2 := \sum_{i,j,k=1}^3|\underline{D}_i\underline{D}_jv_k|^2 \quad\text{on}\quad \Gamma
\end{align*}
when $v\in C^2(\Gamma)^3$.
Let
\begin{align} \label{E:Def_WH}
  W := -\nabla_\Gamma n, \quad H := \mathrm{tr}[W] = -\mathrm{div}_\Gamma n \quad\text{on}\quad \Gamma.
\end{align}
We call $W$ and $H$ the Weingarten map of $\Gamma$ and (twice) the mean curvature of $\Gamma$, respectively.
They are of class $C^3$ on $\Gamma$ by the $C^5$-regularity of $\Gamma$.
Moreover, since
\begin{align*}
  W = -\nabla\bar{n} = -\nabla^2d, \quad Wn = -(\nabla_\Gamma n)n = -\frac{1}{2}\nabla_\Gamma(|n|^2) = 0 \quad\text{on}\quad \Gamma
\end{align*}
by \eqref{E:Nor_Coord}, \eqref{E:ConDer_Surf}, and $|n|=1$ on $\Gamma$, the Weingarten map $W$ is symmetric and satisfies
\begin{align} \label{E:Form_W}
  WP = PW = W \quad\text{on}\quad \Gamma.
\end{align}
The eigenvalues of $W$ are zero and the principal curvatures $\kappa_1$ and $\kappa_2$ (see e.g.~\cites{GiTr01,Lee18}).
By this fact, \eqref{E:Curv_Bound}, and \eqref{E:Form_W} we have the following lemma.

\begin{lemma}[{\cite{Miu_NSCTD_01}*{Lemma~3.3}}] \label{L:Wein}
  The matrix
  \begin{align*}
    I_3-d(x)\overline{W}(x) = I_3-rW(y)
  \end{align*}
  is invertible for all $x=y+rn(y)\in N$ with $y\in\Gamma$ and $r\in(-\delta,\delta)$.
  Moreover,
  \begin{align} \label{E:WReso_P}
    \{I_3-rW(y)\}^{-1}P(y) = P(y)\{I_3-rW(y)\}^{-1}
  \end{align}
  for all $y\in\Gamma$ and $r\in(-\delta,\delta)$ and there exists a constant $c>0$ such that
  \begin{gather}
    c^{-1}|a| \leq \bigl|\{I_3-rW(y)\}^ka\bigr| \leq c|a|, \quad k=\pm1, \label{E:Wein_Bound} \\
    \bigl|I_3-\{I_3-rW(y)\}^{-1}\bigr| \leq c|r| \label{E:Wein_Diff}
  \end{gather}
  for all $y\in\Gamma$, $r\in(-\delta,\delta)$, and $a\in\mathbb{R}^3$.
\end{lemma}

Moreover, the following relations hold by \eqref{E:Nor_Coord}, \eqref{E:P_TGr}, \eqref{E:ConDer_Surf}, and \eqref{E:WReso_P}--\eqref{E:Wein_Diff}.

\begin{lemma}[{\cite{Miu_NSCTD_01}*{Lemma~3.4}}] \label{L:Pi_Der}
  For all $x\in N$ we have
  \begin{align} \label{E:Pi_Der}
    \nabla\pi(x) &= \left\{I_3-d(x)\overline{W}(x)\right\}^{-1}\overline{P}(x).
  \end{align}
  Let $\eta\in C^1(\Gamma)$.
  Then its constant extension $\bar{\eta}=\eta\circ\pi$ satisfies
  \begin{align} \label{E:ConDer_Dom}
    \nabla\bar{\eta}(x) = \left\{I_3-d(x)\overline{W}(x)\right\}^{-1}\overline{\nabla_\Gamma\eta}(x),\quad x\in N
  \end{align}
  and there exists a constant $c>0$ independent of $\eta$ such that
  \begin{gather}
    c^{-1}\left|\overline{\nabla_\Gamma\eta}(x)\right| \leq |\nabla\bar{\eta}(x)| \leq c\left|\overline{\nabla_\Gamma\eta}(x)\right|, \label{E:ConDer_Bound} \\
    \left|\nabla\bar{\eta}(x)-\overline{\nabla_\Gamma\eta}(x)\right| \leq c\left|d(x)\overline{\nabla_\Gamma\eta}(x)\right| \label{E:ConDer_Diff}
  \end{gather}
  for all $x\in N$.
  Moreover, for $\eta\in C^2(\Gamma)$ we have
  \begin{align} \label{E:Con_Hess}
    |\nabla^2\bar{\eta}(x)| \leq c\left(\left|\overline{\nabla_\Gamma\eta}(x)\right|+\left|\overline{\nabla_\Gamma^2\eta}(x)\right|\right), \quad x\in N.
  \end{align}
\end{lemma}

It follows from $n\in C^4(\Gamma)^3$, $W=-\nabla_\Gamma n$ on $\Gamma$, and Lemma~\ref{L:Pi_Der} that, for $x\in N$,
\begin{gather} \
  \nabla\bar{n}(x) = -\left\{I_3-d(x)\overline{W}(x)\right\}^{-1}\overline{W}(x), \label{E:Nor_Grad} \\
  |\nabla\bar{n}(x)| \leq c, \quad |\nabla^2\bar{n}(x)| \leq c. \label{E:NorG_Bound}
\end{gather}
Next we define the Sobolev spaces on $\Gamma$.
For $\eta,\xi\in C^1(\Gamma)$ and $i=1,2,3$ we have an integration by parts formula (see~\cite{Miu_NSCTD_01}*{Lemma~3.5})
\begin{align*}
  \int_\Gamma(\eta\underline{D}_i\xi+\xi\underline{D}_i\eta)\,d\mathcal{H}^2 = -\int_\Gamma \eta\xi Hn_i\,d\mathcal{H}^2,
\end{align*}
where $\mathcal{H}^2$ stands for the two-dimensional Hausdorff measure.
Based on this formula, for $p\in[1,\infty]$ and $i=1,2,3$ we say that $\eta\in L^p(\Gamma)$ has the $i$-th weak tangential derivative if there exists $\eta_i\in L^p(\Gamma)$ such that
\begin{align*}
  \int_\Gamma \eta_i\xi\,d\mathcal{H}^2 = -\int_\Gamma \eta(\underline{D}_i\xi+\xi Hn_i)\,d\mathcal{H}^2
\end{align*}
for all $\xi\in C^1(\Gamma)$.
We denote this $\eta_i$ by $\underline{D}_i\eta$ and define the Sobolev space
\begin{align*}
  W^{1,p}(\Gamma) &:= \{\eta \in L^p(\Gamma) \mid \text{$\underline{D}_i\eta\in L^p(\Gamma)$ for all $i=1,2,3$}\}, \\
  \|\eta\|_{W^{1,p}(\Gamma)} &:=
  \begin{cases}
    \left(\|\eta\|_{L^p(\Gamma)}^p+\|\nabla_\Gamma\eta\|_{L^p(\Gamma)}^p\right)^{1/p} &\text{if}\quad p\in[1,\infty), \\
    \|\eta\|_{L^\infty(\Gamma)}+\|\nabla_\Gamma\eta\|_{L^\infty(\Gamma)} &\text{if}\quad p=\infty.
  \end{cases}
\end{align*}
Here $\nabla_\Gamma\eta:=(\underline{D}_1\eta,\underline{D}_2\eta,\underline{D}_3\eta)^T$ is the weak tangential gradient of $\eta\in W^{1,p}(\Gamma)$.
We also define the second order Sobolev space
\begin{align*}
  W^{2,p}(\Gamma) &:= \{\eta \in W^{1,p}(\Gamma) \mid \text{$\underline{D}_i\underline{D}_j\eta\in L^p(\Gamma)$ for all $i,j=1,2,3$}\}, \\
  \|\eta\|_{W^{2,p}(\Gamma)} &:=
  \begin{cases}
    \left(\|\eta\|_{W^{1,p}(\Gamma)}^p+\|\nabla_\Gamma^2\eta\|_{L^p(\Gamma)}^p\right)^{1/p} &\text{if}\quad p\in[1,\infty), \\
    \|\eta\|_{W^{1,\infty}(\Gamma)}+\|\nabla_\Gamma^2\eta\|_{L^\infty(\Gamma)} &\text{if}\quad p=\infty,
  \end{cases}
\end{align*}
where $\nabla_\Gamma^2\eta:=(\underline{D}_i\underline{D}_j\eta)_{i,j}$ for $\eta\in W^{2,p}(\Gamma)$, and write
\begin{align*}
  W^{0,p}(\Gamma) := L^p(\Gamma), \quad H^m(\Gamma) := W^{m,2}(\Gamma), \quad p\in[1,\infty], \, m=0,1,2.
\end{align*}
Note that $W^{m,p}(\Gamma)$ is a Banach space.
Also, $C^2(\Gamma)$ is dense in $W^{m,p}(\Gamma)$ if $p\neq\infty$ (see~\cite{Miu_NSCTD_01}*{Lemma~3.6}).
For a function space $\mathcal{X}(\Gamma)$ such as $C^m(\Gamma)$ and $W^{m,p}(\Gamma)$, let
\begin{align*}
  \mathcal{X}(\Gamma,T\Gamma) := \{v\in\mathcal{X}(\Gamma)^3 \mid \text{$v\cdot n=0$ on $\Gamma$}\}.
\end{align*}
It is the space of all tangential vector fields on $\Gamma$ whose components belong to $\mathcal{X}(\Gamma)$.
Note that $W^{m,p}(\Gamma,T\Gamma)$ is a closed subspace of $W^{m,p}(\Gamma)^3$.
Also, $C^1(\Gamma,T\Gamma)$ is dense in $L^2(\Gamma,T\Gamma)$ and $H^1(\Gamma,T\Gamma)$ (see~\cite{Miu_NSCTD_01}*{Lemma~3.7}).
We denote by $H^{-1}(\Gamma,T\Gamma)$ the dual space of $H^1(\Gamma,T\Gamma)$ and by $[\cdot,\cdot]_{T\Gamma}$ the duality product between $H^{-1}(\Gamma,T\Gamma)$ and $H^1(\Gamma,T\Gamma)$.
By setting
\begin{align*}
  [v,w]_{T\Gamma} := (v,w)_{L^2(\Gamma)}, \quad v\in L^2(\Gamma,T\Gamma),\, w\in H^1(\Gamma,T\Gamma)
\end{align*}
we consider vector fields in $L^2(\Gamma,T\Gamma)$ as elements of $H^{-1}(\Gamma,T\Gamma)$.

\subsection{Curved thin domain} \label{SS:Pre_CTD}
Let $g_0$ and $g_1$ be $C^4$ functions on $\Gamma$ such that the difference $g:=g_1-g_0$ satisfies \eqref{E:G_Inf}.
For $\varepsilon\in(0,1]$ we define a curved thin domain $\Omega_\varepsilon$ in $\mathbb{R}^3$ by \eqref{E:Def_CTD}, i.e.
\begin{align*}
  \Omega_\varepsilon := \{y+rn(y) \mid y\in\Gamma,\,\varepsilon g_0(y) < r < \varepsilon g_1(y)\}.
\end{align*}
The inner and outer boundaries $\Gamma_\varepsilon^0$ and $\Gamma_\varepsilon^1$ of $\Omega_\varepsilon$ are given by
\begin{align*}
  \Gamma_\varepsilon^i := \{y+\varepsilon g_i(y)n(y) \mid y\in\Gamma\}, \quad i=0,1
\end{align*}
and the whole boundary of $\Omega_\varepsilon$ is denoted by $\Gamma_\varepsilon:=\Gamma_\varepsilon^0\cup\Gamma_\varepsilon^1$.
Let $N$ be the tubular neighborhood of $\Gamma$ of radius $\delta>0$ given in Section~\ref{SS:Pre_Surf}.
Since $g_0$ and $g_1$ are bounded on $\Gamma$, there exists $\tilde{\varepsilon}\in(0,1]$ such that $\tilde{\varepsilon}|g_i|<\delta$ on $\Gamma$ for $i=0,1$.
Replacing $g_i$ with $\tilde{\varepsilon}g_i$ for $i=0,1$ we may assume $\tilde{\varepsilon}=1$.
Then for all $\varepsilon\in(0,1]$ we have $\overline{\Omega}_\varepsilon\subset N$ and we can apply the lemmas given in Section~\ref{SS:Pre_Surf} in $\overline{\Omega}_\varepsilon$.

Let $\tau_\varepsilon^i$ and $n_\varepsilon^i$ be vector fields on $\Gamma$ given by
\begin{align}
  \tau_\varepsilon^i(y) &:= \{I_3-\varepsilon g_i(y)W(y)\}^{-1}\nabla_\Gamma g_i(y), \label{E:Def_NB_Aux}\\
  n_\varepsilon^i(y) &:= (-1)^{i+1}\frac{n(y)-\varepsilon\tau_\varepsilon^i(y)}{\sqrt{1+\varepsilon^2|\tau_\varepsilon^i(y)|^2}} \label{E:Def_NB}
\end{align}
for $y\in\Gamma$ and $i=0,1$.
By \eqref{E:P_TGr} and \eqref{E:WReso_P}, $\tau_\varepsilon^i$ is tangential on $\Gamma$.

\begin{lemma}[{\cite{Miu_NSCTD_01}*{Lemma~3.8}}] \label{L:NB_Aux}
  For all $y\in\Gamma$, $i=0,1$, and $k,l=1,2,3$ we have
  \begin{gather}
    |\tau_\varepsilon^i(y)| \leq c, \quad |\underline{D}_k\tau_\varepsilon^i(y)| \leq c, \quad |\underline{D}_l\underline{D}_k\tau_\varepsilon^i(y)| \leq c, \label{E:Tau_Bound} \\
    |\tau_\varepsilon^i(y)-\nabla_\Gamma g_i(y)| \leq c\varepsilon, \quad |\nabla_\Gamma\tau_\varepsilon^i(y)-\nabla_\Gamma^2g_i(y)| \leq c\varepsilon \label{E:Tau_Diff}
  \end{gather}
  with a constant $c>0$ independent of $\varepsilon$.
\end{lemma}

Let $n_\varepsilon$ be the unit outward normal vector field of $\Gamma_\varepsilon$.
It satisfies
\begin{align} \label{E:Nor_Bo}
  n_\varepsilon(x) = \bar{n}_\varepsilon^i(x), \quad x\in\Gamma_\varepsilon^i,\,i=0,1,
\end{align}
where $\bar{n}_\varepsilon^i=n_\varepsilon^i\circ\pi$ is the constant extension of $n_\varepsilon^i$ (see~\cite{Miu_NSCTD_01}*{Lemma~3.9}).
We define the tangential gradient and the tangential derivatives of $\varphi\in C^1(\Gamma_\varepsilon)$ by
\begin{align*}
  \nabla_{\Gamma_\varepsilon}\varphi := P_\varepsilon\nabla\tilde{\varphi}, \quad \underline{D}_i^\varepsilon\varphi := \sum_{j=1}^3[P_\varepsilon]_{ij}\partial_j\tilde{\varphi} \quad\text{on}\quad \Gamma_\varepsilon,\, i=1,2,3,
\end{align*}
where $P_\varepsilon:=I_3-n_\varepsilon\otimes n_\varepsilon$ and $\tilde{\varphi}$ is any $C^1$-extension of $\varphi$ to an open neighborhood of $\Gamma_\varepsilon$ with $\tilde{\varphi}|_{\Gamma_\varepsilon}=\varphi$.
For $u=(u_1,u_2,u_3)^T\in C^1(\Gamma_\varepsilon)^3$ we set
\begin{align*}
  \nabla_{\Gamma_\varepsilon}u :=
  \begin{pmatrix}
    \underline{D}_1^\varepsilon u_1 & \underline{D}_1^\varepsilon u_2 & \underline{D}_1^\varepsilon u_3 \\
    \underline{D}_2^\varepsilon u_1 & \underline{D}_2^\varepsilon u_2 & \underline{D}_2^\varepsilon u_3 \\
    \underline{D}_3^\varepsilon u_1 & \underline{D}_3^\varepsilon u_2 & \underline{D}_3^\varepsilon u_3
  \end{pmatrix} \quad\text{on}\quad \Gamma_\varepsilon.
\end{align*}
Then we have
\begin{align} \label{E:TGr_Bo}
  \nabla_{\Gamma_\varepsilon}u = P_\varepsilon\nabla\tilde{u} \quad\text{on}\quad \Gamma_\varepsilon
\end{align}
for any $C^1$-extension $\tilde{u}$ of $u$ to an open neighborhood of $\Gamma_\varepsilon$ with $\tilde{u}|_{\Gamma_\varepsilon}=u$.
We also define the Weingarten map of $\Gamma_\varepsilon$ by
\begin{align*}
  W_\varepsilon := -\nabla_{\Gamma_\varepsilon}n_\varepsilon \quad\text{on}\quad \Gamma_\varepsilon
\end{align*}
and the Sobolev spaces on $\Gamma_\varepsilon$ as in Section~\ref{SS:Pre_Surf}.
By \eqref{E:Nor_Bo} the functions $n_\varepsilon$, $P_\varepsilon$, and $W_\varepsilon$ on $\Gamma_\varepsilon$ can be compared with the constant extensions of $n$, $P$, and $W$.

\begin{lemma}[{\cite{Miu_NSCTD_01}*{Lemma~3.10}}] \label{L:Comp_Nor}
  For $x\in\Gamma_\varepsilon^i$, $i=0,1$ we have
  \begin{align}
    \left|n_\varepsilon(x)-(-1)^{i+1}\left\{\bar{n}(x)-\varepsilon\overline{\nabla_\Gamma g_i}(x)\right\}\right| &\leq c\varepsilon^2, \label{E:Comp_N} \\
    \left|P_\varepsilon(x)-\overline{P}(x)\right| &\leq c\varepsilon, \label{E:Comp_P} \\
    \left|W_\varepsilon(x)-(-1)^{i+1}\overline{W}(x)\right| &\leq c\varepsilon, \label{E:Comp_W}
  \end{align}
  with a constant $c>0$ independent of $\varepsilon$.
\end{lemma}

Since $\nabla_\Gamma g_i$ is bounded on $\Gamma$ by $g_i\in C^4(\Gamma)$, it follows from \eqref{E:Comp_N} that
\begin{align} \label{E:Comp_N_Re}
  |n_\varepsilon(x)-(-1)^{i+1}\bar{n}(x)| \leq c\varepsilon
\end{align}
for all $x\in\Gamma_\varepsilon^i$, $i=0,1$.
In the sequel we use \eqref{E:Comp_N_Re} instead of \eqref{E:Comp_N}.

Let us give a change of variables formula for an integral over $\Omega_\varepsilon$.
For functions $\varphi$ on $\Omega_\varepsilon$ and $\eta$ on $\Gamma_\varepsilon^i$, $i=0,1$ we use the notations
\begin{alignat}{2}
  \varphi^\sharp(y,r) &:= \varphi(y+rn(y)), &\quad &y\in\Gamma,\,r\in(\varepsilon g_0(y),\varepsilon g_1(y)), \label{E:Pull_Dom} \\
  \eta_i^\sharp(y) &:= \eta(y+\varepsilon g_i(y)n(y)), &\quad &y\in\Gamma. \label{E:Pull_Bo}
\end{alignat}
We define a function $J=J(y,r)$ for $y\in\Gamma$ and $r\in(-\delta,\delta)$ by
\begin{align} \label{E:Def_Jac}
  J(y,r) := \mathrm{det}[I_3-rW(y)] = \{1-r\kappa_1(y)\}\{1-r\kappa_2(y)\}.
\end{align}
Then it follows from \eqref{E:Curv_Bound} and $\kappa_1,\kappa_2\in C^3(\Gamma)$ that
\begin{align} \label{E:Jac_Bound_02}
  c^{-1} \leq J(y,r) \leq c, \quad |\nabla_\Gamma J(y,r)| \leq c, \quad \left|\frac{\partial J}{\partial r}(y,r)\right| \leq c
\end{align}
for all $y\in\Gamma$ and $r\in(-\delta,\delta)$, where $\nabla_\Gamma J$ is the tangential gradient of $J$ with respect to $y\in\Gamma$.
By the boundedness of $\kappa_1$, $\kappa_2$, $g_0$, and $g_1$ on $\Gamma$ we also have
\begin{align} \label{E:Jac_Diff_02}
  |J(y,r)-1| \leq c\varepsilon \quad\text{for all}\quad y\in\Gamma,\,r\in[\varepsilon g_0(y),\varepsilon g_1(y)].
\end{align}
For a function $\varphi$ on $\Omega_\varepsilon$ the change of variables formula
\begin{align} \label{E:CoV_Dom}
  \int_{\Omega_\varepsilon}\varphi(x)\,dx = \int_\Gamma\int_{\varepsilon g_0(y)}^{\varepsilon g_1(y)}\varphi(y+rn(y))J(y,r)\,dr\,d\mathcal{H}^2(y)
\end{align}
holds (see e.g.~\cite{GiTr01}*{Section~14.6}).
By \eqref{E:Jac_Bound_02} and \eqref{E:CoV_Dom} we observe that
\begin{align} \label{E:CoV_Equiv}
  c^{-1}\|\varphi\|_{L^p(\Omega_\varepsilon)}^p \leq \int_\Gamma\int_{\varepsilon g_0(y)}^{\varepsilon g_1(y)}|\varphi^\sharp(y,r)|^p\,dr\,d\mathcal{H}^2(y) \leq c\|\varphi\|_{L^p(\Omega_\varepsilon)}^p
\end{align}
for $\varphi\in L^p(\Omega_\varepsilon)$, $p\in[1,\infty)$.
We also have the following estimates for the constant extension $\bar{\eta}=\eta\circ\pi$ of a function $\eta$ on $\Gamma$ and for a function on $\Gamma_\varepsilon^i$, $i=0,1$.

\begin{lemma}[{\cite{Miu_NSCTD_01}*{Lemma~3.12}}] \label{L:Con_Lp_W1p}
  For $p\in[1,\infty)$ we have $\eta\in L^p(\Gamma)$ if and only if $\bar{\eta}\in L^p(\Omega_\varepsilon)$, and there exists a constant $c>0$ independent of $\varepsilon$ and $\eta$ such that
  \begin{align} \label{E:Con_Lp}
    c^{-1}\varepsilon^{1/p}\|\eta\|_{L^p(\Gamma)} \leq \|\bar{\eta}\|_{L^p(\Omega_\varepsilon)} \leq c\varepsilon^{1/p}\|\eta\|_{L^p(\Gamma)}.
  \end{align}
  Moreover, $\eta\in W^{1,p}(\Gamma)$ if and only if $\bar{\eta}\in W^{1,p}(\Omega_\varepsilon)$ and we have
  \begin{align} \label{E:Con_W1p}
    c^{-1}\varepsilon^{1/p}\|\eta\|_{W^{1,p}(\Gamma)} \leq \|\bar{\eta}\|_{W^{1,p}(\Omega_\varepsilon)} \leq c\varepsilon^{1/p}\|\eta\|_{W^{1,p}(\Gamma)}.
  \end{align}
\end{lemma}

\begin{lemma}[{\cite{Miu_NSCTD_01}*{Lemma~3.13}}] \label{L:CoV_Surf}
  For $i=0,1$ and $p\in[1,\infty)$ let $\varphi\in L^p(\Gamma_\varepsilon^i)$ and $\varphi_i^\sharp$ be given by \eqref{E:Pull_Bo}.
  Then $\varphi_i^\sharp\in L^p(\Gamma)$ and
  \begin{align} \label{E:Lp_CoV_Surf}
    c^{-1}\|\varphi\|_{L^p(\Gamma_\varepsilon^i)} \leq \|\varphi_i^\sharp\|_{L^p(\Gamma)} \leq c\|\varphi\|_{L^p(\Gamma_\varepsilon^i)},
  \end{align}
  where $c>0$ is a constant independent of $\varepsilon$ and $\varphi$.
\end{lemma}

\section{Fundamental tools for analysis} \label{S:Tool}

\subsection{Sobolev inequalities} \label{SS:Tool_Sob}
Let us present Sobolev inequalities on $\Gamma$ and $\Omega_\varepsilon$.
The proofs of Lemmas~\ref{L:La_Surf} and~\ref{L:Agmon} are given in Appendix~\ref{S:Ap_Proofs}.
Also, we omit the proof of Lemma~\ref{L:Poincare} since it is given in the first part~\cite{Miu_NSCTD_01} of our study.

First we give Ladyzhenskaya's inequality on $\Gamma$.

\begin{lemma} \label{L:La_Surf}
  There exists a constant $c>0$ such that
  \begin{align} \label{E:La_Surf}
    \|\eta\|_{L^4(\Gamma)} \leq c\|\eta\|_{L^2(\Gamma)}^{1/2}\|\nabla_\Gamma\eta\|_{L^2(\Gamma)}^{1/2}
  \end{align}
  for all $\eta\in H^1(\Gamma)$.
\end{lemma}

Next we present Poincar\'{e} and trace type inequalities on $\Omega_\varepsilon$.
For a function $\varphi$ on $\Omega_\varepsilon$ and $x\in\Omega_\varepsilon$ we define the derivative of $\varphi$ in the normal direction of $\Gamma$ by
\begin{align} \label{E:Def_NorDer}
  \partial_n\varphi(x) := (\bar{n}(x)\cdot\nabla)\varphi(x) = \frac{d}{dr}\bigl(\varphi(y+rn(y))\bigr)\Big|_{r=d(x)} \quad (y=\pi(x)\in\Gamma).
\end{align}
Note that for the constant extension $\bar{\eta}=\eta\circ\pi$ of $\eta\in C^1(\Gamma)$ we have
\begin{align} \label{E:NorDer_Con}
  \partial_n\bar{\eta}(x) = (\bar{n}(x)\cdot\nabla)\bar{\eta}(x) = 0, \quad x\in \Omega_\varepsilon.
\end{align}

\begin{lemma}[{\cite{Miu_NSCTD_01}*{Lemma~4.1}}] \label{L:Poincare}
  There exists $c>0$ independent of $\varepsilon$ such that
  \begin{align}
    \|\varphi\|_{L^p(\Omega_\varepsilon)} &\leq c\left(\varepsilon^{1/p}\|\varphi\|_{L^p(\Gamma_\varepsilon^i)}+\varepsilon\|\partial_n\varphi\|_{L^p(\Omega_\varepsilon)}\right), \label{E:Poin_Dom} \\
    \|\varphi\|_{L^p(\Gamma_\varepsilon^i)} &\leq c\left(\varepsilon^{-1/p}\|\varphi\|_{L^p(\Omega_\varepsilon)}+\varepsilon^{1-1/p}\|\partial_n\varphi\|_{L^p(\Omega_\varepsilon)}\right) \label{E:Poin_Bo}
  \end{align}
  for all $\varphi\in W^{1,p}(\Omega_\varepsilon)$ with $p\in[1,\infty)$ and $i=0,1$.
\end{lemma}

By the Sobolev embedding theorem (see~\cite{AdFo03}) a function $\varphi$ in $H^2(\Omega_\varepsilon)$ is continuous and bounded on $\Omega_\varepsilon$.
The following anisotropic Agmon inequality gives the explicit dependence of a constant on $\varepsilon$ in the $L^\infty(\Omega_\varepsilon)$-estimate for $\varphi$.

\begin{lemma} \label{L:Agmon}
  There exists a constant $c>0$ independent of $\varepsilon$ such that
  \begin{multline} \label{E:Agmon}
    \|\varphi\|_{L^\infty(\Omega_\varepsilon)} \leq c\varepsilon^{-1/2}\|\varphi\|_{L^2(\Omega_\varepsilon)}^{1/4}\|\varphi\|_{H^2(\Omega_\varepsilon)}^{1/2} \\
    \times\left(\|\varphi\|_{L^2(\Omega_\varepsilon)}+\varepsilon\|\partial_n\varphi\|_{L^2(\Omega_\varepsilon)}+\varepsilon^2\|\partial_n^2\varphi\|_{L^2(\Omega_\varepsilon)}\right)^{1/4}
  \end{multline}
  for all $\varphi\in H^2(\Omega_\varepsilon)$.
\end{lemma}

\subsection{Consequences of the boundary conditions} \label{SS:Tool_Bo}
In this subsection we derive several properties of vector fields satisfying the impermeable boundary condition
\begin{align} \label{E:Bo_Imp}
  u\cdot n_\varepsilon = 0 \quad\text{on}\quad \Gamma_\varepsilon
\end{align}
or the slip boundary conditions
\begin{align} \label{E:Bo_Slip}
  u\cdot n_\varepsilon = 0, \quad 2\nu P_\varepsilon D(u)n_\varepsilon+\gamma_\varepsilon u = 0 \quad\text{on}\quad \Gamma_\varepsilon.
\end{align}
Here $\nu>0$ is the viscosity coefficient independent of $\varepsilon$ and $\gamma_\varepsilon\geq0$ is the friction coefficient on $\Gamma_\varepsilon$ given by \eqref{E:Def_Fric}.
Also, for a vector field $u$ on $\Omega_\varepsilon$ we denote by
\begin{align*}
  D(u) := (\nabla u)_S = \frac{\nabla u+(\nabla u)^T}{2}
\end{align*}
the strain rate tensor.
First we deal with vector fields satisfying \eqref{E:Bo_Imp}.

\begin{lemma} \label{L:Exp_Bo}
  For $i=0,1$ let $u\in C(\Gamma_\varepsilon^i)^3$ satisfy \eqref{E:Bo_Imp} on $\Gamma_\varepsilon^i$.
  Then
  \begin{align} \label{E:Exp_Bo}
    u\cdot\bar{n} = \varepsilon u\cdot\bar{\tau}_\varepsilon^i, \quad |u\cdot\bar{n}| \leq c\varepsilon|u| \quad\text{on}\quad \Gamma_\varepsilon^i,
  \end{align}
  where $\tau_\varepsilon^i$ is given by \eqref{E:Def_NB_Aux} and $c>0$ is a constant independent of $\varepsilon$ and $u$.
\end{lemma}

\begin{proof}
  By \eqref{E:Def_NB}, \eqref{E:Nor_Bo}, and \eqref{E:Bo_Imp} on $\Gamma_\varepsilon^i$ we immediately get the first equality of \eqref{E:Exp_Bo}.
  The second inequality of \eqref{E:Exp_Bo} follows from the first one and \eqref{E:Tau_Bound}.
\end{proof}

As a consequence of Lemma~\ref{L:Exp_Bo} we show Poincar\'{e} type inequalities for the normal component with respect to $\Gamma$ of a vector field on $\Omega_\varepsilon$.

\begin{lemma} \label{L:Poin_Nor}
  Let $p\in[1,\infty)$.
  There exists $c>0$ independent of $\varepsilon$ such that
  \begin{align} \label{E:Poin_Nor}
    \|u\cdot\bar{n}\|_{L^p(\Omega_\varepsilon)} \leq c\varepsilon\|u\|_{W^{1,p}(\Omega_\varepsilon)}
  \end{align}
  for all $u\in W^{1,p}(\Omega_\varepsilon)^3$ satisfying \eqref{E:Bo_Imp} on $\Gamma_\varepsilon^0$ or on $\Gamma_\varepsilon^1$.
  We also have
  \begin{align} \label{E:Poin_Dnor}
    \left\|\overline{P}\nabla(u\cdot\bar{n})\right\|_{L^p(\Omega_\varepsilon)} \leq c\varepsilon\|u\|_{W^{2,p}(\Omega_\varepsilon)}
  \end{align}
  for all $u\in W^{2,p}(\Omega_\varepsilon)^3$ satisfying \eqref{E:Bo_Imp} on $\Gamma_\varepsilon^0$ or on $\Gamma_\varepsilon^1$.
\end{lemma}

\begin{proof}
  Let $u\in W^{1,p}(\Omega_\varepsilon)^3$.
  We may assume that $u$ satisfies \eqref{E:Bo_Imp} on $\Gamma_\varepsilon^0$ without loss of generality.
  By \eqref{E:NorDer_Con} and \eqref{E:Poin_Dom} with $i=0$,
  \begin{align} \label{Pf_PN:Lp_Dom}
    \|u\cdot\bar{n}\|_{L^p(\Omega_\varepsilon)} \leq c\left(\varepsilon^{1/p}\|u\cdot\bar{n}\|_{L^p(\Gamma_\varepsilon^0)}+\varepsilon\|\partial_nu\|_{L^p(\Omega_\varepsilon)}\right).
  \end{align}
  Moreover, we apply the second inequality of \eqref{E:Exp_Bo} and then use \eqref{E:Poin_Bo} with $i=0$ to the first term on the right-hand side of \eqref{Pf_PN:Lp_Dom} to get
  \begin{align} \label{Pf_PN:Lp_Bo}
    \|u\cdot\bar{n}\|_{L^p(\Gamma_\varepsilon^0)} \leq c\varepsilon\|u\|_{L^p(\Gamma_\varepsilon^0)} \leq c\varepsilon^{1-1/p}\|u\|_{W^{1,p}(\Omega_\varepsilon)}.
  \end{align}
  Combining \eqref{Pf_PN:Lp_Dom} and \eqref{Pf_PN:Lp_Bo} we obtain \eqref{E:Poin_Nor}.

  Next suppose that $u\in W^{2,p}(\Omega_\varepsilon)^3$ satisfies \eqref{E:Bo_Imp} on $\Gamma_\varepsilon^0$.
  Noting that
  \begin{align*}
    \left|\partial_n\Bigl[\overline{P}\nabla(u\cdot\bar{n})\Bigr]\right| \leq c(|\nabla u|+|\nabla^2u|) \quad\text{in}\quad \Omega_\varepsilon
  \end{align*}
  by \eqref{E:NorG_Bound} and \eqref{E:NorDer_Con}, we apply \eqref{E:Poin_Dom} with $i=0$ to get
  \begin{align} \label{Pf_PN:Gr_Lp}
    \left\|\overline{P}\nabla(u\cdot\bar{n})\right\|_{L^p(\Omega_\varepsilon)} \leq c\left(\varepsilon^{1/p}\left\|\overline{P}\nabla(u\cdot\bar{n})\right\|_{L^p(\Gamma_\varepsilon^0)}+\varepsilon\|u\|_{W^{2,p}(\Omega_\varepsilon)}\right).
  \end{align}
  Since the tangential gradient on $\Gamma_\varepsilon^0$ depends only on the values of a function on $\Gamma_\varepsilon^0$, we observe by \eqref{E:TGr_Bo} and the first equality of \eqref{E:Exp_Bo} that
  \begin{align*}
    P_\varepsilon\nabla(u\cdot\bar{n}) = \nabla_{\Gamma_\varepsilon}(u\cdot\bar{n}) = \varepsilon\nabla_{\Gamma_\varepsilon}(u\cdot\bar{\tau}_\varepsilon^i) = \varepsilon P_\varepsilon\nabla(u\cdot\bar{\tau}_\varepsilon^i) \quad\text{on}\quad \Gamma_\varepsilon^0.
  \end{align*}
  Hence we have
  \begin{align*}
    \overline{P}\nabla(u\cdot\bar{n}) &= P_\varepsilon\nabla(u\cdot\bar{n})+\Bigl(\overline{P}-P_\varepsilon\Bigr)\nabla(u\cdot\bar{n}) \\
    &= \varepsilon P_\varepsilon\nabla(u\cdot\bar{\tau}_\varepsilon^0)+\Bigl(\overline{P}-P_\varepsilon\Bigr)\nabla(u\cdot\bar{n})
  \end{align*}
  on $\Gamma_\varepsilon^0$.
  By this formula, \eqref{E:ConDer_Bound}, \eqref{E:NorG_Bound}, \eqref{E:Tau_Bound}, and \eqref{E:Comp_P},
  \begin{align*}
    \left|\overline{P}\nabla(u\cdot\bar{n})\right| \leq c\varepsilon(|u|+|\nabla u|) \quad\text{on}\quad \Gamma_\varepsilon^0.
  \end{align*}
  From this inequality and \eqref{E:Poin_Bo} it follows that
  \begin{align*}
    \left\|\overline{P}\nabla(u\cdot\bar{n})\right\|_{L^p(\Gamma_\varepsilon^0)} \leq c\varepsilon\left(\|u\|_{L^p(\Gamma_\varepsilon^0)}+\|\nabla u\|_{L^p(\Gamma_\varepsilon^0)}\right) \leq c\varepsilon^{1-1/p}\|u\|_{W^{2,p}(\Omega_\varepsilon)}.
  \end{align*}
  Applying this inequality to the right-hand side of \eqref{Pf_PN:Gr_Lp} we obtain \eqref{E:Poin_Dnor}.
\end{proof}

Next we consider the slip boundary conditions \eqref{E:Bo_Slip}.

\begin{lemma} \label{L:NSl}
  Let $i=0,1$.
  If $u\in C^2(\overline{\Omega}_\varepsilon)^3$ satisfies \eqref{E:Bo_Slip} on $\Gamma_\varepsilon^i$, then
  \begin{align} \label{E:NSl_ND}
    P_\varepsilon(n_\varepsilon\cdot\nabla)u = -W_\varepsilon u-\frac{\gamma_\varepsilon}{\nu}u \quad\text{on}\quad \Gamma_\varepsilon^i.
  \end{align}
\end{lemma}

\begin{proof}
  Since $u\cdot n_\varepsilon=0$ on $\Gamma_\varepsilon^i$, we have
  \begin{align*}
    0 = \nabla_{\Gamma_\varepsilon}(u\cdot n_\varepsilon) = (\nabla_{\Gamma_\varepsilon}u)n_\varepsilon+(\nabla_{\Gamma_\varepsilon}n_\varepsilon)u = (\nabla_{\Gamma_\varepsilon}u)n_\varepsilon-W_\varepsilon u \quad\text{on}\quad \Gamma_\varepsilon^i
  \end{align*}
  and thus $(\nabla_{\Gamma_\varepsilon}u)n_\varepsilon=W_\varepsilon u$ on $\Gamma_\varepsilon^i$.
  By this equality and \eqref{E:TGr_Bo},
  \begin{align*}
    2P_\varepsilon D(u)n_\varepsilon &= P_\varepsilon\{(\nabla u)n_\varepsilon+(\nabla u)^Tn_\varepsilon\} = (\nabla_{\Gamma_\varepsilon}u)n_\varepsilon+P_\varepsilon(n_\varepsilon\cdot\nabla)u \\
    &= W_\varepsilon u+P_\varepsilon(n_\varepsilon\cdot\nabla)u
  \end{align*}
  on $\Gamma_\varepsilon^i$.
  Hence it follows from the second equality of \eqref{E:Bo_Slip} that
  \begin{align*}
    P_\varepsilon(n_\varepsilon\cdot\nabla)u = 2P_\varepsilon D(u)n_\varepsilon-W_\varepsilon u = -\frac{\gamma_\varepsilon}{\nu}u-W_\varepsilon u \quad\text{on}\quad \Gamma_\varepsilon^i.
  \end{align*}
  Thus \eqref{E:NSl_ND} is valid.
\end{proof}

Using \eqref{E:NSl_ND} we show an estimate which is essential for a Poincar\'{e} type inequality for the derivatives of the residual part in the decomposition of a vector field on $\Omega_\varepsilon$ based on the average in the thin direction (see Lemma~\ref{L:Po_Grad_Ur}).

\begin{lemma} \label{L:PDnU_WU}
  Let $p\in[1,\infty)$.
  Suppose that the inequalities \eqref{E:Fric_Upper} are valid.
  Then there exists a constant $c>0$ independent of $\varepsilon$ such that
  \begin{align} \label{E:PDnU_WU}
    \left\|\overline{P}\partial_nu+\overline{W}u\right\|_{L^p(\Omega_\varepsilon)} \leq c\varepsilon\|u\|_{W^{2,p}(\Omega_\varepsilon)}
  \end{align}
  for all $u\in W^{2,p}(\Omega_\varepsilon)^3$ satisfying \eqref{E:Bo_Slip} on $\Gamma_\varepsilon^0$ or on $\Gamma_\varepsilon^1$.
\end{lemma}

\begin{proof}
  For $i=0$ or $i=1$ let $u\in W^{2,p}(\Omega_\varepsilon)^3$ satisfy \eqref{E:Bo_Slip} on $\Gamma_\varepsilon^i$.
  Since $\partial_nu=(\bar{n}\cdot\nabla)u$ in $\Omega_\varepsilon$ and $n$, $P$, and $W$ are bounded on $\Gamma$, we observe by \eqref{E:NorDer_Con} that
  \begin{align*}
    \left|\partial_n\Bigl[\overline{P}\partial_n u+\overline{W}u\Bigr]\right| \leq c(|\nabla u|+|\nabla^2u|) \quad\text{in}\quad \Omega_\varepsilon.
  \end{align*}
  We apply \eqref{E:Poin_Dom} and the above inequality to get
  \begin{align} \label{Pf_PuWu:L2_Dom}
    \left\|\overline{P}\partial_nu+\overline{W}u\right\|_{L^p(\Omega_\varepsilon)} \leq c\left(\varepsilon^{1/p}\left\|\overline{P}\partial_nu+\overline{W}u\right\|_{L^p(\Gamma_\varepsilon^i)}+\varepsilon\|u\|_{W^{2,p}(\Omega_\varepsilon)}\right).
  \end{align}
  Moreover, since $\partial_nu=(\bar{n}\cdot\nabla)u$ and $u$ satisfies \eqref{E:Bo_Slip} on $\Gamma_\varepsilon^i$,
  \begin{align*}
    \overline{P}\partial_nu &= (-1)^{i+1}P_\varepsilon(n_\varepsilon\cdot\nabla)u+P_\varepsilon[\{\bar{n}-(-1)^{i+1}n_\varepsilon\}\cdot\nabla]u+\Bigl(\overline{P}-P_\varepsilon\Bigr)(\bar{n}\cdot\nabla)u \\
    &= (-1)^{i+1}\left(-W_\varepsilon u-\frac{\gamma_\varepsilon}{\nu}u\right)+P_\varepsilon[\{\bar{n}-(-1)^{i+1}n_\varepsilon\}\cdot\nabla]u+\Bigl(\overline{P}-P_\varepsilon\Bigr)(\bar{n}\cdot\nabla)u
  \end{align*}
  on $\Gamma_\varepsilon^i$ by \eqref{E:NSl_ND}.
  Hence by \eqref{E:Fric_Upper} and \eqref{E:Comp_P}--\eqref{E:Comp_N_Re} we get
  \begin{align*}
    \left|\overline{P}\partial_nu+\overline{W}u\right| \leq \left|\overline{W}u-(-1)^{i+1}W_\varepsilon u\right|+c\varepsilon(|u|+|\nabla u|) \leq c\varepsilon(|u|+|\nabla u|)
  \end{align*}
  on $\Gamma_\varepsilon^i$, which together with \eqref{E:Poin_Bo} implies that
  \begin{align*}
    \left\|\overline{P}\partial_nu+\overline{W}u\right\|_{L^p(\Gamma_\varepsilon^i)} \leq c\varepsilon\left(\|u\|_{L^p(\Gamma_\varepsilon^i)}+\|\nabla u\|_{L^p(\Gamma_\varepsilon^i)}\right) \leq c\varepsilon^{1-1/p}\|u\|_{W^{2,p}(\Omega_\varepsilon)}.
  \end{align*}
  Applying this inequality to \eqref{Pf_PuWu:L2_Dom} we obtain \eqref{E:PDnU_WU}.
\end{proof}

\subsection{Impermeable extension of surface vector fields} \label{SS:Tool_IE}
In the analysis of integrals over $\Omega_\varepsilon$ involving a vector field on $\Gamma$ it is convenient to consider its extension to $\Omega_\varepsilon$ satisfying the impermeable boundary condition \eqref{E:Bo_Imp}.
Let $\tau_\varepsilon^0$ and $\tau_\varepsilon^1$ be the vector fields on $\Gamma$ given by \eqref{E:Def_NB_Aux}.
We define a vector field $\Psi_\varepsilon$ on $N$ by
\begin{align} \label{E:Def_ExAux}
  \Psi_\varepsilon(x) := \frac{1}{\bar{g}(x)}\bigl\{\bigl(d(x)-\varepsilon\bar{g}_0(x)\bigr)\bar{\tau}_\varepsilon^1(x)+\bigl(\varepsilon\bar{g}_1(x)-d(x)\bigr)\bar{\tau}_\varepsilon^0(x)\bigr\}, \quad x\in N.
\end{align}
By definition, $\Psi_\varepsilon=\varepsilon\bar{\tau}_\varepsilon^i$ on $\Gamma_\varepsilon^i$, $i=0,1$.
Let us give several estimates for $\Psi_\varepsilon$.

\begin{lemma} \label{L:ExAux_Bound}
  There exists a constant $c>0$ independent of $\varepsilon$ such that
  \begin{align} \label{E:ExAux_Bound}
    |\Psi_\varepsilon| \leq c\varepsilon, \quad |\nabla\Psi_\varepsilon| \leq c, \quad |\nabla^2\Psi_\varepsilon| \leq c \quad\text{in}\quad \Omega_\varepsilon.
  \end{align}
  Moreover, we have
  \begin{align} \label{E:ExAux_TNDer}
    \left|\overline{P}\nabla\Psi_\varepsilon\right| \leq c\varepsilon, \quad \left|\partial_n\Psi_\varepsilon-\frac{1}{\bar{g}}\overline{\nabla_\Gamma g}\right| \leq c\varepsilon \quad\text{in}\quad \Omega_\varepsilon.
  \end{align}
\end{lemma}

\begin{proof}
  Applying \eqref{E:Tau_Bound} and
  \begin{align} \label{Pf_EAB:Width}
    0 \leq d(x)-\varepsilon\bar{g}_0(x) \leq \varepsilon\bar{g}(x), \quad 0 \leq \varepsilon\bar{g}_1(x)-d(x) \leq \varepsilon\bar{g}(x), \quad x\in \Omega_\varepsilon
  \end{align}
  to \eqref{E:Def_ExAux} we get the first inequality of \eqref{E:ExAux_Bound}.
  Also, by $\nabla d=\bar{n}$ in $N$ we have
  \begin{align} \label{Pf_EAB:Grad}
    \nabla\Psi_\varepsilon = \frac{1}{\bar{g}}\{\bar{n}\otimes(\bar{\tau}_\varepsilon^1-\bar{\tau}_\varepsilon^0)+F_\varepsilon\} \quad\text{in}\quad N,
  \end{align}
  where $F_\varepsilon$ is a $3\times3$ matrix-valued function on $N$ given by
  \begin{align*}
    F_\varepsilon := -\nabla\bar{g}\otimes\Psi_\varepsilon+\varepsilon(\nabla\bar{g}_1\otimes\bar{\tau}_\varepsilon^0-\nabla\bar{g}_0\otimes\bar{\tau}_\varepsilon^1)+(d-\varepsilon\bar{g}_0)\nabla\bar{\tau}_\varepsilon^1+(\varepsilon\bar{g}_1-d)\nabla\bar{\tau}_\varepsilon^0.
  \end{align*}
  By \eqref{E:G_Inf}, \eqref{E:Tau_Bound}, and $|n|=1$ on $\Gamma$ we see that the first term on the right-hand side of \eqref{Pf_EAB:Grad} is bounded on $N$ uniformly in $\varepsilon$.
  Moreover, from \eqref{E:G_Inf}, \eqref{E:ConDer_Bound}, \eqref{E:Tau_Bound}, the first inequality of \eqref{E:ExAux_Bound}, \eqref{Pf_EAB:Width}, and $g_0,g_1\in C^4(\Gamma)$ we deduce that
  \begin{align} \label{Pf_EAB:Est_F}
    |F_\varepsilon| \leq c\varepsilon \quad\text{in}\quad \Omega_\varepsilon.
  \end{align}
  Hence the second inequality of \eqref{E:ExAux_Bound} follows.
  Similarly, differentiating both sides of \eqref{Pf_EAB:Grad} and using \eqref{E:G_Inf}, \eqref{E:ConDer_Bound}, \eqref{E:Con_Hess}, \eqref{E:NorG_Bound}, \eqref{E:Tau_Bound}, the first and second inequalities of \eqref{E:ExAux_Bound}, and $g_0,g_1\in C^4(\Gamma)$ we can derive the last inequality of \eqref{E:ExAux_Bound}.

  Let us prove \eqref{E:ExAux_TNDer}.
  First note that
  \begin{align*}
    P[n\otimes(\tau_\varepsilon^1-\tau_\varepsilon^0)] &= (Pn)\otimes(\tau_\varepsilon^1-\tau_\varepsilon^0) = 0, \\
    [(\tau_\varepsilon^1-\tau_\varepsilon^0)\otimes n]n &= |n|^2(\tau_\varepsilon^1-\tau_\varepsilon^0) = \tau_\varepsilon^1-\tau_\varepsilon^0
  \end{align*}
  on $\Gamma$.
  From these equalities, \eqref{E:Def_NorDer}, and \eqref{Pf_EAB:Grad} we deduce that
  \begin{align*}
    \overline{P}\nabla\Psi_\varepsilon = \frac{1}{\bar{g}}\overline{P}F_\varepsilon, \quad \partial_n\Psi_\varepsilon = (\nabla\Psi_\varepsilon)^T\bar{n} = \frac{1}{\bar{g}}(\bar{\tau}_\varepsilon^1-\bar{\tau}_\varepsilon^0)+F_\varepsilon^T\bar{n} \quad\text{in}\quad N.
  \end{align*}
  Hence we observe by \eqref{E:G_Inf} and \eqref{Pf_EAB:Est_F} that
  \begin{align*}
    \left|\overline{P}\nabla\Psi_\varepsilon\right| \leq c\left|\overline{P}F_\varepsilon\right| \leq c|F_\varepsilon| \leq c\varepsilon \quad\text{in}\quad \Omega_\varepsilon.
  \end{align*}
  Also, applying \eqref{E:G_Inf}, \eqref{E:Tau_Diff}, and \eqref{Pf_EAB:Est_F} to the equality for $\partial_n\Psi_\varepsilon$ we obtain
  \begin{align*}
    \left|\partial_n\Psi_\varepsilon-\frac{1}{\bar{g}}\overline{\nabla_\Gamma g}\right| \leq \frac{1}{\bar{g}}\sum_{i=0,1}\left|\bar{\tau}_\varepsilon^i-\overline{\nabla_\Gamma g_i}\right|+|F_\varepsilon| \leq c\varepsilon \quad\text{in}\quad \Omega_\varepsilon.
  \end{align*}
  Thus the inequalities \eqref{E:ExAux_TNDer} are valid.
\end{proof}

For a tangential vector field $v$ on $\Gamma$ (i.e. $v\cdot n=0$ on $\Gamma$) we define
\begin{align} \label{E:Def_ExImp}
  E_\varepsilon v(x) := \bar{v}(x)+\{\bar{v}(x)\cdot\Psi_\varepsilon(x)\}\bar{n}(x), \quad x\in N,
\end{align}
where $\bar{v}$ and $\bar{n}$ are the constant extensions of $v$ and $n$.
By the definition of $\Psi_\varepsilon$ we easily see that $E_\varepsilon v$ satisfies the impermeable boundary condition \eqref{E:Bo_Imp}.

\begin{lemma} \label{L:ExImp_Bo}
  For all $v\in C(\Gamma,T\Gamma)$ we have $E_\varepsilon v\cdot n_\varepsilon=0$ on $\Gamma_\varepsilon$.
\end{lemma}

\begin{proof}
  For $i=0,1$ we observe by \eqref{E:Def_NB}, \eqref{E:Nor_Bo}, and $v\cdot n=0$ on $\Gamma$ that
  \begin{align*}
    \bar{v}\cdot n_\varepsilon = (-1)^i\frac{\varepsilon\bar{v}\cdot\bar{\tau}_\varepsilon^i}{\sqrt{1+\varepsilon^2|\bar{\tau}_\varepsilon^i|^2}}, \quad \bar{n}\cdot n_\varepsilon = \frac{(-1)^{i+1}}{\sqrt{1+\varepsilon^2|\bar{\tau}_\varepsilon^i|^2}} \quad\text{on}\quad \Gamma_\varepsilon^i.
  \end{align*}
  From these equalities and $\Psi_\varepsilon=\varepsilon\bar{\tau}_\varepsilon^i$ on $\Gamma_\varepsilon^i$ by \eqref{E:Def_ExAux} we get $E_\varepsilon v\cdot n_\varepsilon=0$ on $\Gamma_\varepsilon^i$.
\end{proof}

Also, it is easy to show that $E_\varepsilon v\in W^{m,p}(\Omega_\varepsilon)$ for $v\in W^{m,p}(\Gamma,T\Gamma)$.

\begin{lemma} \label{L:ExImp_Wmp}
  There exists a constant $c>0$ independent of $\varepsilon$ such that
  \begin{align} \label{E:ExImp_Wmp}
    \|E_\varepsilon v\|_{W^{m,p}(\Omega_\varepsilon)} \leq c\varepsilon^{1/p}\|v\|_{W^{m,p}(\Gamma)}
  \end{align}
  for all $v\in W^{m,p}(\Gamma,T\Gamma)$ with $m=0,1,2$ and $p\in[1,\infty)$.
\end{lemma}

\begin{proof}
  By \eqref{E:ConDer_Bound}, \eqref{E:Con_Hess}, \eqref{E:NorG_Bound}, and \eqref{E:ExAux_Bound} we have
  \begin{gather*}
    |E_\varepsilon v| \leq c|\bar{v}|, \quad |\nabla E_\varepsilon v| \leq c\left(|\bar{v}|+\left|\overline{\nabla_\Gamma v}\right|\right), \quad |\nabla^2 E_\varepsilon v| \leq c\left(|\bar{v}|+\left|\overline{\nabla_\Gamma v}\right|+\left|\overline{\nabla_\Gamma^2 v}\right|\right)
  \end{gather*}
  in $\Omega_\varepsilon$.
  These inequalities and \eqref{E:Con_Lp} imply \eqref{E:ExImp_Wmp}.
\end{proof}

If $\widetilde{\Omega}_\varepsilon$ is a flat thin domain of the form
\begin{align*}
  \widetilde{\Omega}_\varepsilon = \{x=(x',x_3)\in\mathbb{R}^3 \mid x'\in\omega,\,\varepsilon\tilde{g}_0(x')<x_3<\varepsilon\tilde{g}_1(x')\},
\end{align*}
where $\omega$ is a domain in $\mathbb{R}^2$ and $\tilde{g}_0$ and $\tilde{g}_1$ are functions on $\omega$, then we have
\begin{align*}
  \mathrm{div}(E_\varepsilon v)(x) = \frac{1}{\tilde{g}(x')}\mathrm{div}(\tilde{g}v)(x'), \quad x=(x',x_3)\in\widetilde{\Omega}_\varepsilon \quad (\tilde{g}:=\tilde{g}_1-\tilde{g}_0)
\end{align*}
for $v\colon\omega\to\mathbb{R}^2$ (see~\cite{HoSe10}*{Lemma~4.24} and~\cite{IfRaSe07}*{Remark~3.1}).
This is not the case for the curved thin domain $\Omega_\varepsilon$ of the form \eqref{E:Def_CTD} since the principal curvatures of the limit surface $\Gamma$ do not vanish in general.
However, we can show that the difference between $\mathrm{div}(E_\varepsilon v)$ and $g^{-1}\mathrm{div}_\Gamma(gv)$ is of order $\varepsilon$ in $\Omega_\varepsilon$.

\begin{lemma} \label{L:ExImp_Div}
  There exists a constant $c>0$ independent of $\varepsilon$ such that
  \begin{align} \label{E:ExImp_Grad}
    \left|\nabla E_\varepsilon v-\left\{\overline{\nabla_\Gamma v}+\frac{1}{\bar{g}}\Bigl(\bar{v}\cdot\overline{\nabla_\Gamma g}\Bigr)\overline{Q}\right\}\right| \leq c\varepsilon\left(|\bar{v}|+\left|\overline{\nabla_\Gamma v}\right|\right) \quad\text{in}\quad \Omega_\varepsilon
  \end{align}
  for all $v\in C^1(\Gamma,T\Gamma)$.
  Moreover, we have
  \begin{align} \label{E:ExImp_Div}
    \left|\mathrm{div}(E_\varepsilon v)-\frac{1}{\bar{g}}\overline{\mathrm{div}_\Gamma(gv)}\right| \leq c\varepsilon\left(|\bar{v}|+\left|\overline{\nabla_\Gamma v}\right|\right) \quad\text{in}\quad \Omega_\varepsilon.
  \end{align}
\end{lemma}

\begin{proof}
  Since
  \begin{align*}
    \nabla E_\varepsilon v = \nabla\bar{v}+[(\nabla\bar{v})\Psi_\varepsilon+(\nabla\Psi_\varepsilon)\bar{v}]\otimes\bar{n}+(\bar{v}\cdot\Psi_\varepsilon)\nabla\bar{n} \quad\text{in}\quad N
  \end{align*}
  by \eqref{E:Def_ExImp} and $(v\cdot\nabla_\Gamma g)Q=[(v\cdot\nabla_\Gamma g)n]\otimes n=[(n\otimes\nabla_\Gamma g)v]\otimes n$ on $\Gamma$, we have
  \begin{multline} \label{Pf_EID:Diff}
    \left|\nabla E_\varepsilon v-\left\{\overline{\nabla_\Gamma v}+\frac{1}{\bar{g}}\Bigl(\bar{v}\cdot\overline{\nabla_\Gamma g}\Bigr)\overline{Q}\right\}\right| \\
    \leq \left|\nabla\bar{v}-\overline{\nabla_\Gamma v}\right|+|[(\nabla\bar{v})\Psi_\varepsilon]\otimes\bar{n}|+|(\bar{v}\cdot\Psi_\varepsilon)\nabla\bar{n}| \\
    +\left|\left[\left(\nabla\Psi_\varepsilon-\frac{1}{\bar{g}}\bar{n}\otimes\overline{\nabla_\Gamma g}\right)\bar{v}\right]\otimes\bar{n}\right| \quad\text{in}\quad N.
  \end{multline}
  Also, by $\nabla\Psi_\varepsilon=\overline{P}\nabla\Psi_\varepsilon+\overline{Q}\nabla\Psi_\varepsilon=\overline{P}\nabla\Psi_\varepsilon+\bar{n}\otimes\partial_n\Psi_\varepsilon$ in $N$ and \eqref{E:ExAux_TNDer} we get
  \begin{align*}
    \left|\nabla\Psi_\varepsilon-\frac{1}{\bar{g}}\bar{n}\otimes\overline{\nabla_\Gamma g}\right| \leq \left|\overline{P}\nabla\Psi_\varepsilon\right|+\left|\bar{n}\otimes\left(\partial_n\Psi_\varepsilon-\frac{1}{\bar{g}}\overline{\nabla_\Gamma g}\right)\right| \leq c\varepsilon \quad\text{in}\quad \Omega_\varepsilon.
  \end{align*}
  Applying this inequality, \eqref{E:ConDer_Bound}, \eqref{E:ConDer_Diff}, \eqref{E:NorG_Bound}, \eqref{E:ExAux_Bound}, and $|d|\leq c\varepsilon$ in $\Omega_\varepsilon$ to the right-hand side of \eqref{Pf_EID:Diff} we obtain \eqref{E:ExImp_Grad}.
  Moreover, since
  \begin{align*}
    \mathrm{tr}\left[\nabla_\Gamma v+\frac{1}{g}(v\cdot\nabla_\Gamma g)Q\right] = \mathrm{div}_\Gamma v+\frac{1}{g}(v\cdot\nabla_\Gamma g) = \frac{1}{g}\mathrm{div}_\Gamma(gv) \quad\text{on}\quad \Gamma
  \end{align*}
  by $\mathrm{tr}[Q]=n\cdot n=1$ on $\Gamma$, the inequality \eqref{E:ExImp_Div} follows from \eqref{E:ExImp_Grad}.
\end{proof}

\section{Stokes operator under the slip boundary conditions} \label{S:St_Op}
We summarize the fundamental results of the first part~\cite{Miu_NSCTD_01} of our study on the Stokes operator for $\Omega_\varepsilon$ under the slip boundary conditions.
Throughout this section we impose Assumptions~\ref{Assump_1} and~\ref{Assump_2} and fix the constant $\varepsilon_0$ given in Lemma~\ref{L:Uni_aeps}.

Integration by parts shows that
\begin{multline} \label{E:IbP_St}
  \int_{\Omega_\varepsilon}\{\Delta u_1+\nabla(\mathrm{div}\,u_1)\}\cdot u_2\,dx \\
  = -2\int_{\Omega_\varepsilon}D(u_1):D(u_2)\,dx+2\int_{\Gamma_\varepsilon}[D(u_1)n_\varepsilon]\cdot u_2\,d\mathcal{H}^2
\end{multline}
for $u_1\in H^2(\Omega_\varepsilon)^3$ and $u_2\in H^1(\Omega_\varepsilon)^3$ (see also~\cite{Miu_NSCTD_01}*{Lemma~7.1}).
From this formula it follows that, if $u_1$ satisfies $\mathrm{div}\,u_1=0$ in $\Omega_\varepsilon$ and the slip boundary conditions \eqref{E:Bo_Slip} and $u_2$ satisfies the impermeable boundary condition \eqref{E:Bo_Imp}, then
\begin{align*}
  \nu\int_{\Omega_\varepsilon}\Delta u_1\cdot u_2\,dx = -2\nu\int_{\Omega_\varepsilon}D(u_1):D(u_2)\,dx-\sum_{i=0,1}\gamma_\varepsilon^i\int_{\Gamma_\varepsilon^i}u_1\cdot u_2\,d\mathcal{H}^2.
\end{align*}
Hence the bilinear form for the Stokes problem
\begin{align*}
  \left\{
  \begin{alignedat}{3}
    -\nu\Delta u+\nabla p &= f &\quad &\text{in} &\quad &\Omega_\varepsilon, \\
    \mathrm{div}\,u &= 0 &\quad &\text{in} &\quad &\Omega_\varepsilon, \\
    u\cdot n_\varepsilon &= 0 &\quad &\text{on} &\quad &\Gamma_\varepsilon, \\
    2\nu P_\varepsilon D(u)n_\varepsilon+\gamma_\varepsilon u &= 0 &\quad &\text{on} &\quad &\Gamma_\varepsilon
  \end{alignedat}
  \right.
\end{align*}
is of the form
\begin{align*}
  a_\varepsilon(u_1,u_2) := 2\nu\bigl(D(u_1),D(u_2)\bigr)_{L^2(\Omega_\varepsilon)}+\sum_{i=0,1}\gamma_\varepsilon^i(u_1,u_2)_{L^2(\Gamma_\varepsilon^i)}
\end{align*}
for $u_1,u_2\in H^1(\Omega_\varepsilon)^3$.
Let $\mathcal{H}_\varepsilon$ and $\mathcal{V}_\varepsilon$ be the function spaces defined by \eqref{E:Def_Heps}.
Then for each $\varepsilon\in(0,\varepsilon_0]$ the bilinear form $a_\varepsilon$ is uniformly bounded and coercive on $\mathcal{V}_\varepsilon$ by Lemma~\ref{L:Uni_aeps}.
Hence by the Lax--Milgram theorem it induces a bounded linear operator $A_\varepsilon$ from $\mathcal{V}_\varepsilon$ into its dual space $\mathcal{V}_\varepsilon'$ such that
\begin{align*}
  {}_{\mathcal{V}_\varepsilon'}\langle A_\varepsilon u_1,u_2\rangle_{\mathcal{V}_\varepsilon} = a_\varepsilon(u_1,u_2), \quad u_1,u_2\in \mathcal{V}_\varepsilon,
\end{align*}
where ${}_{\mathcal{V}_\varepsilon'}\langle\cdot,\cdot\rangle_{\mathcal{V}_\varepsilon}$ is the duality product between $\mathcal{V}_\varepsilon'$ and $\mathcal{V}_\varepsilon$.
If we consider $A_\varepsilon$ as an unbounded operator on $\mathcal{H}_\varepsilon$ with domain
\begin{align*}
  D(A_\varepsilon) = \{u\in\mathcal{V}_\varepsilon\mid A_\varepsilon u\in\mathcal{H}_\varepsilon\},
\end{align*}
then $A_\varepsilon$ is positive and self-adjoint on $\mathcal{H}_\varepsilon$ by the Lax--Milgram theory.
Hence the square root $A_\varepsilon^{1/2}$ of $A_\varepsilon$ is well-defined on $D(A_\varepsilon^{1/2})=\mathcal{V}_\varepsilon$ and
\begin{align} \label{E:L2in_Ahalf}
  (A_\varepsilon u_1,u_2)_{L^2(\Omega_\varepsilon)} = (A_\varepsilon^{1/2}u_1,A_\varepsilon^{1/2} u_2)_{L^2(\Omega_\varepsilon)}
\end{align}
for all $u_1\in D(A_\varepsilon)$ and $u_2\in \mathcal{V}_\varepsilon$ (see~\cites{BoFa13,So01}).
Moreover, by a regularity result for a solution to the Stokes problem (see~\cites{AmRe14,Be04,SoSc73}) we observe that
\begin{align} \label{E:Dom_St}
  D(A_\varepsilon) = \{u\in \mathcal{V}_\varepsilon\cap H^2(\Omega_\varepsilon)^3 \mid \text{$2\nu P_\varepsilon D(u)n_\varepsilon+\gamma_\varepsilon u=0$ on $\Gamma_\varepsilon$}\}
\end{align}
and $A_\varepsilon u=-\nu\mathbb{P}_\varepsilon\Delta u$ for $u\in D(A_\varepsilon)$, where $\mathbb{P}_\varepsilon$ is the orthogonal projection from $L^2(\Omega_\varepsilon)^3$ onto $\mathcal{H}_\varepsilon$.
In what follows, we call $A_\varepsilon$ the Stokes operator for $\Omega_\varepsilon$ under the slip boundary conditions or simply the Stokes operator on $\mathcal{H}_\varepsilon$.

In our first paper~\cite{Miu_NSCTD_01} we studied the Stokes operator $A_\varepsilon$ in detail.
We present the main results of~\cite{Miu_NSCTD_01} below which are essential for the proof of the global existence of a strong solution to \eqref{E:NS_CTD}.

\begin{lemma}[{\cite{Miu_NSCTD_01}*{Lemma~2.5}}] \label{L:Stokes_H1}
  There exists a constant $c>0$ such that
  \begin{align} \label{E:Stokes_H1}
    c^{-1}\|u\|_{H^1(\Omega_\varepsilon)} \leq \|A_\varepsilon^{1/2}u\|_{L^2(\Omega_\varepsilon)} \leq c\|u\|_{H^1(\Omega_\varepsilon)}
  \end{align}
  for all $\varepsilon\in(0,\varepsilon_0]$ and $u\in \mathcal{V}_\varepsilon$.
  Moreover, if $u\in D(A_\varepsilon)$, then we have
  \begin{align} \label{E:Stokes_Po}
    \|A_\varepsilon^{1/2}u\|_{L^2(\Omega_\varepsilon)} \leq c\|A_\varepsilon u\|_{L^2(\Omega_\varepsilon)}.
  \end{align}
\end{lemma}

\begin{lemma}[{\cite{Miu_NSCTD_01}*{Theorem~2.6}}] \label{L:Comp_Sto_Lap}
  There exists a constant $c>0$ such that
  \begin{align} \label{E:Comp_Sto_Lap}
    \|A_\varepsilon u+\nu\Delta u\|_{L^2(\Omega_\varepsilon)} \leq c\|u\|_{H^1(\Omega_\varepsilon)}
  \end{align}
  for all $\varepsilon\in(0,\varepsilon_0]$ and $u\in D(A_\varepsilon)$.
\end{lemma}

Note that the $H^1(\Omega_\varepsilon)$-norm of $u$ appears in the right-hand side of \eqref{E:Stokes_H2}, not its $H^2(\Omega_\varepsilon)$-norm.
This is essential for a good estimate for the trilinear term
\begin{align*}
  \bigl((u\cdot\nabla)u,A_\varepsilon u\bigr)_{L^2(\Omega_\varepsilon)}, \quad u\in D(A_\varepsilon)
\end{align*}
derived in Section~\ref{S:Tri}.

\begin{lemma}[{\cite{Miu_NSCTD_01}*{Theorem~2.7}}] \label{L:Stokes_H2}
  There exists a constant $c>0$ such that
  \begin{align} \label{E:Stokes_H2}
    c^{-1}\|u\|_{H^2(\Omega_\varepsilon)} \leq \|A_\varepsilon u\|_{L^2(\Omega_\varepsilon)} \leq c\|u\|_{H^2(\Omega_\varepsilon)}
  \end{align}
  for all $\varepsilon\in(0,\varepsilon_0]$ and $u\in D(A_\varepsilon)$.
\end{lemma}

\begin{lemma}[{\cite{Miu_NSCTD_01}*{Corollary~2.8}}] \label{L:St_Inter}
  There exists a constant $c>0$ such that
  \begin{align} \label{E:St_Inter}
    \|u\|_{H^1(\Omega_\varepsilon)} \leq c\|u\|_{L^2(\Omega_\varepsilon)}^{1/2}\|u\|_{H^2(\Omega_\varepsilon)}^{1/2}
  \end{align}
  for all $\varepsilon\in(0,\varepsilon_0]$ and $u\in D(A_\varepsilon)$.
\end{lemma}

\section{Average operators in the thin direction} \label{S:Ave}
The purpose of this section is to study average operators in the thin direction which play a fundamental role in the analysis of the Navier--Stokes equations \eqref{E:NS_CTD}.
Throughout this section we assume $\varepsilon\in(0,1]$ and denote by $\bar{\eta}=\eta\circ\pi$ the constant extension of a function $\eta$ on $\Gamma$ in the normal direction of $\Gamma$.
We also write $\partial_n\varphi=(\bar{n}\cdot\nabla)\varphi$ for the derivative of a function $\varphi$ on $\Omega_\varepsilon$ in the normal direction of $\Gamma$.

\subsection{Definition and basic properties of the average operators} \label{SS:Ave_Def}
Let us give the definition of average operators and show their basic properties.

\begin{definition} \label{D:Average}
  We define the average operator $M$ as
  \begin{align} \label{E:Def_Ave}
    M\varphi(y) := \frac{1}{\varepsilon g(y)}\int_{\varepsilon g_0(y)}^{\varepsilon g_1(y)}\varphi(y+rn(y))\,dr, \quad y\in\Gamma
  \end{align}
  for a function $\varphi$ on $\Omega_\varepsilon$.
  The operator $M$ is also applied to a vector field $u\colon\Omega_\varepsilon\to\mathbb{R}^3$ and we define the averaged tangential component $M_\tau u$ of $u$ by
  \begin{align} \label{E:Def_Tan_Ave}
    M_\tau u(y) := P(y)Mu(y) = \frac{1}{\varepsilon g(y)}\int_{\varepsilon g_0(y)}^{\varepsilon g_1(y)}P(y)u(y+rn(y))\,dr, \quad y\in\Gamma.
  \end{align}
\end{definition}

For the sake of simplicity, we denote the tangential and normal components with respect to the surface $\Gamma$ of a vector field $u\colon\Omega_\varepsilon\to\mathbb{R}^3$ by
\begin{align} \label{E:Def_U_TN}
  u_\tau(x) := \overline{P}(x)u(x), \quad u_n(x) := \{u(x)\cdot\bar{n}(x)\}\bar{n}(x), \quad x\in\Omega_\varepsilon
\end{align}
so that $u=u_\tau+u_n$ and $u_\tau\cdot u_n=0$ in $\Omega_\varepsilon$ (note that $u_n$ is a vector field).
Moreover, we use the notations \eqref{E:Pull_Dom} and \eqref{E:Pull_Bo} and sometimes suppress the arguments of functions.
For example, we write
\begin{align*}
  M\varphi = \frac{1}{\varepsilon g}\int_{\varepsilon g_0}^{\varepsilon g_1}\varphi^\sharp\,dr, \quad M_\tau u = \frac{1}{\varepsilon g}\int_{\varepsilon g_0}^{\varepsilon g_1}u_\tau^\sharp\,dr \quad\text{on}\quad \Gamma.
\end{align*}

\begin{lemma} \label{L:Ave_Lp}
  Let $p\in[1,\infty)$.
  There exists $c>0$ independent of $\varepsilon$ such that
  \begin{align}
    \|M\varphi\|_{L^p(\Gamma)} &\leq c\varepsilon^{-1/p}\|\varphi\|_{L^p(\Omega_\varepsilon)}, \label{E:Ave_Lp_Surf} \\
    \left\|\overline{M\varphi}\right\|_{L^p(\Omega_\varepsilon)} &\leq c\|\varphi\|_{L^p(\Omega_\varepsilon)} \label{E:Ave_Lp_Dom}
  \end{align}
  for all $\varphi\in L^p(\Omega_\varepsilon)$.
\end{lemma}

\begin{proof}
  By H\"{o}lder's inequality and \eqref{E:G_Inf},
  \begin{align*}
    |M\varphi(y)|^p = \left|\frac{1}{\varepsilon g(y)}\int_{\varepsilon g_0(y)}^{\varepsilon g_1(y)}\varphi^\sharp(y,r)\,dr\right|^p \leq c\varepsilon^{-1}\int_{\varepsilon g_0(y)}^{\varepsilon g_1(y)}|\varphi^\sharp(y,r)|^p\,dr
  \end{align*}
  for all $y\in\Gamma$.
  Integrating both sides of the above inequality over $\Gamma$ and using \eqref{E:CoV_Equiv} we obtain \eqref{E:Ave_Lp_Surf}.
  The inequality \eqref{E:Ave_Lp_Dom} follows from \eqref{E:Con_Lp} and \eqref{E:Ave_Lp_Surf}.
\end{proof}

\begin{lemma} \label{L:AveT_Lp}
  Let $p\in[1,\infty)$.
  There exists $c>0$ independent of $\varepsilon$ such that
  \begin{align} \label{E:AveT_Lp_Surf}
    \|M_\tau u\|_{L^p(\Gamma)} &\leq c\varepsilon^{-1/p}\|u\|_{L^p(\Omega_\varepsilon)}
  \end{align}
  for all $u\in L^p(\Omega_\varepsilon)^3$.
\end{lemma}

\begin{proof}
  Since $M_\tau u=Mu_\tau$ on $\Gamma$ and $|u_\tau|\leq|u|$ in $\Omega_\varepsilon$ by \eqref{E:Def_Tan_Ave} and \eqref{E:Def_U_TN}, the inequality \eqref{E:AveT_Lp_Surf} immediately follows from \eqref{E:Ave_Lp_Surf}.
\end{proof}

\begin{lemma} \label{L:Ave_Diff}
  Let $p\in[1,\infty)$.
  There exists $c>0$ independent of $\varepsilon$ such that
  \begin{align}
    \left\|\varphi-\overline{M\varphi}\right\|_{L^p(\Omega_\varepsilon)} &\leq c\varepsilon\|\partial_n\varphi\|_{L^p(\Omega_\varepsilon)}, \label{E:Ave_Diff_Dom}\\
    \left\|\varphi-\overline{M\varphi}\right\|_{L^p(\Gamma_\varepsilon^i)} &\leq c\varepsilon^{1-1/p}\|\partial_n\varphi\|_{L^p(\Omega_\varepsilon)}, \quad i=0,1 \label{E:Ave_Diff_Bo}
  \end{align}
  for all $\varphi\in W^{1,p}(\Omega_\varepsilon)$.
\end{lemma}

\begin{proof}
  For $y\in\Gamma$ and $r\in(\varepsilon g_0(y),\varepsilon g_1(y))$ we have
  \begin{align} \label{Pf_ADiff:Diff}
    \varphi^\sharp(y,r)-M\varphi(y) &= \frac{1}{\varepsilon g(y)}\int_{\varepsilon g_0(y)}^{\varepsilon g_1(y)}\{\varphi^\sharp(y,r)-\varphi^\sharp(y,r_1)\}\,dr_1.
  \end{align}
  Since $\partial\varphi^\sharp/\partial r=(\partial_n\varphi)^\sharp$ by \eqref{E:Pull_Dom} and \eqref{E:Def_NorDer},
  \begin{align*}
    \begin{aligned}
      |\varphi^\sharp(y,r)-\varphi^\sharp(y,r_1)| &= \left|\int_{r_1}^r\frac{\partial}{\partial r_2}\bigl(\varphi^\sharp(y,r_2)\bigr)\,dr_2\right| \leq \int_{\varepsilon g_0(y)}^{\varepsilon g_1(y)}|(\partial_n\varphi)^\sharp(y,r_2)|\,dr_2.
    \end{aligned}
  \end{align*}
  Noting that the right-hand side is independent of $r_1$, we apply this inequality to the right-hand side of \eqref{Pf_ADiff:Diff} and then use H\"{o}lder's inequality to get
  \begin{align*}
    |\varphi^\sharp(y,r)-M\varphi(y)| &\leq \int_{\varepsilon g_0(y)}^{\varepsilon g_1(y)}|(\partial_n\varphi)^\sharp(y,r_2)|\,dr_2 \\
    &\leq c\varepsilon^{1-1/p}\left(\int_{\varepsilon g_0(y)}^{\varepsilon g_1(y)}|(\partial_n\varphi)^\sharp(y,r_2)|^p\,dr_2\right)^{1/p}.
  \end{align*}
  Since the last term is independent of $r$, this inequality and \eqref{E:CoV_Equiv} imply that
  \begin{align*}
    \left\|\varphi-\overline{M\varphi}\right\|_{L^p(\Omega_\varepsilon)}^p &\leq c\int_\Gamma\int_{\varepsilon g_0(y)}^{\varepsilon g_1(y)}|\varphi^\sharp(y,r)-M\varphi(y)|^p\,dr\,d\mathcal{H}^2(y) \\
    &\leq c\int_\Gamma\varepsilon g(y)\left(c\varepsilon^{p-1}\int_{\varepsilon g_0(y)}^{\varepsilon g_1(y)}|(\partial_n\varphi)^\sharp(y,r_2)|^p\,dr_2\right)d\mathcal{H}^2(y) \\
    &\leq c\varepsilon^p\|\partial_n\varphi\|_{L^p(\Omega_\varepsilon)}^p.
  \end{align*}
  Thus \eqref{E:Ave_Diff_Dom} follows.
  We also have \eqref{E:Ave_Diff_Bo} by applying \eqref{E:Poin_Bo} to $\varphi-\overline{M\varphi}$ and using \eqref{E:NorDer_Con} and \eqref{E:Ave_Diff_Dom} to the resulting inequality.
\end{proof}

\begin{lemma} \label{L:Ave_N_Lp}
  Let $p\in[1,\infty)$.
  There exists $c>0$ independent of $\varepsilon$ such that
  \begin{align} \label{E:Ave_N_Lp}
    \|Mu\cdot n\|_{L^p(\Gamma)} &\leq c\varepsilon^{1-1/p}\|u\|_{W^{1,p}(\Omega_\varepsilon)}
  \end{align}
  for all $u\in W^{1,p}(\Omega_\varepsilon)^3$ satisfying \eqref{E:Bo_Imp} on $\Gamma_\varepsilon^0$ or on $\Gamma_\varepsilon^1$.
\end{lemma}

\begin{proof}
  By \eqref{E:Ave_Lp_Surf} we observe that
  \begin{align*}
    \|Mu\cdot n\|_{L^p(\Gamma)} = \|M(u\cdot\bar{n})\|_{L^p(\Gamma)} \leq c\varepsilon^{-1/p}\|u\cdot\bar{n}\|_{L^p(\Omega_\varepsilon)}.
  \end{align*}
  Applying \eqref{E:Poin_Nor} to the right-hand side we obtain \eqref{E:Ave_N_Lp}.
\end{proof}

\begin{lemma} \label{L:AveT_Diff}
  Let $p\in[1,\infty)$.
  There exists $c>0$ independent of $\varepsilon$ such that
  \begin{align} \label{E:AveT_Diff_Dom}
    \left\|u-\overline{M_\tau u}\right\|_{L^p(\Omega_\varepsilon)} \leq c\varepsilon\|u\|_{W^{1,p}(\Omega_\varepsilon)}
  \end{align}
  for all $u\in W^{1,p}(\Omega_\varepsilon)^3$ satisfying \eqref{E:Bo_Imp} on $\Gamma_\varepsilon^0$ or on $\Gamma_\varepsilon^1$.
\end{lemma}

\begin{proof}
  Noting that
  \begin{gather*}
    u-\overline{M_\tau u} = (u\cdot\bar{n})\bar{n}+\Bigl(u_\tau-\overline{Mu_\tau}\Bigr), \quad |(u\cdot\bar{n})\bar{n}| = |u\cdot\bar{n}|, \\
    |\partial_nu_\tau| = \left|\partial_n\Bigl(\overline{P}u\Bigr)\right| = \left|\overline{P}\partial_nu\right| \leq |\partial_nu| \leq c|\nabla u|
  \end{gather*}
  in $\Omega_\varepsilon$ by \eqref{E:NorDer_Con}, \eqref{E:Def_Tan_Ave}, and \eqref{E:Def_U_TN}, we deduce from \eqref{E:Poin_Nor} and \eqref{E:Ave_Diff_Dom} that
  \begin{align*}
    \left\|u-\overline{M_\tau u}\right\|_{L^p(\Omega_\varepsilon)} &\leq \|u\cdot\bar{n}\|_{L^p(\Omega_\varepsilon)}+\left\|u_\tau-\overline{Mu_\tau}\right\|_{L^p(\Omega_\varepsilon)} \\
    &\leq c\varepsilon(\|u\|_{W^{1,p}(\Omega_\varepsilon)}+\|\partial_nu_\tau\|_{L^p(\Omega_\varepsilon)}) \\
    &\leq c\varepsilon\|u\|_{W^{1,p}(\Omega_\varepsilon)}.
  \end{align*}
  Hence \eqref{E:AveT_Diff_Dom} is valid.
\end{proof}

Unlike the case of a flat thin domain (see~\cites{Ho10,HoSe10,IfRaSe07}), the constant extension of $M$ on $L^2(\Omega_\varepsilon)$ is not symmetric since the Jacobian $J\not\equiv1$ appears in the change of variables formula \eqref{E:CoV_Dom}.
However, its skew-symmetric part is small of order $\varepsilon$.

\begin{lemma} \label{L:Ave_Inner}
  There exists a constant $c>0$ independent of $\varepsilon$ such that
  \begin{align} \label{E:Ave_Inner}
    \left|\Bigl(\overline{M\varphi_1},\varphi_2\Bigr)_{L^2(\Omega_\varepsilon)}-\Bigl(\varphi_1,\overline{M\varphi_2}\Bigr)_{L^2(\Omega_\varepsilon)}\right| \leq c\varepsilon\|\varphi_1\|_{L^2(\Omega_\varepsilon)}\|\varphi_2\|_{L^2(\Omega_\varepsilon)}
  \end{align}
  for all $\varphi_1,\varphi_2\in L^2(\Omega_\varepsilon)$ and
  \begin{align} \label{E:AveT_Inner}
    \left|\Bigl(\overline{M_\tau u_1},u_2\Bigr)_{L^2(\Omega_\varepsilon)}-\Bigl(u_1,\overline{M_\tau u_2}\Bigr)_{L^2(\Omega_\varepsilon)}\right| \leq c\varepsilon\|u_1\|_{L^2(\Omega_\varepsilon)}\|u_2\|_{L^2(\Omega_\varepsilon)}
  \end{align}
  for all $u_1,u_2\in L^2(\Omega_\varepsilon)^3$.
\end{lemma}

\begin{proof}
  For $\varphi_1,\varphi_2\in L^2(\Omega_\varepsilon)$ we have
  \begin{align*}
    \Bigl(\overline{M\varphi_1},\varphi_2\Bigr)_{L^2(\Omega_\varepsilon)} &= \int_\Gamma M\varphi_1\left(\int_{\varepsilon g_0}^{\varepsilon g_1}\varphi_2^\sharp J\,dr\right)d\mathcal{H}^2 \\
    &= \varepsilon(M\varphi_1,gM\varphi_2)_{L^2(\Gamma)}+\int_\Gamma M\varphi_1\left(\int_{\varepsilon g_0}^{\varepsilon g_1}\varphi_2^\sharp(J-1)\,dr\right)d\mathcal{H}^2
  \end{align*}
  by \eqref{E:CoV_Dom} and \eqref{E:Def_Ave}.
  Similarly,
  \begin{align*}
    \Bigl(\varphi_1,\overline{M\varphi_2}\Bigr)_{L^2(\Omega_\varepsilon)} = \varepsilon(gM\varphi_1,M\varphi_2)_{L^2(\Gamma)}+\int_\Gamma\left(\int_{\varepsilon g_0}^{\varepsilon g_1}\varphi_1^\sharp(J-1)\,dr\right)M\varphi_2\,d\mathcal{H}^2.
  \end{align*}
  Since $(M\varphi_1,gM\varphi_2)_{L^2(\Gamma)}=(gM\varphi_1,M\varphi_2)_{L^2(\Gamma)}$, we see by \eqref{E:Jac_Diff_02} that
  \begin{multline} \label{Pf_AI:Est}
    \left|\Bigl(\overline{M\varphi_1},\varphi_2\Bigr)_{L^2(\Omega_\varepsilon)}-\Bigl(\varphi_1,\overline{M\varphi_2}\Bigr)_{L^2(\Omega_\varepsilon)}\right| \\
    \leq c\varepsilon\left\{\int_\Gamma|M\varphi_1|\left(\int_{\varepsilon g_0}^{\varepsilon g_1}|\varphi_2^\sharp|\,dr\right)d\mathcal{H}^2+\int_\Gamma\left(\int_{\varepsilon g_0}^{\varepsilon g_1}|\varphi_1^\sharp|\,dr\right)|M\varphi_2|\,d\mathcal{H}^2\right\}.
  \end{multline}
  Moreover, applying H\"{o}lder's inequality twice and using \eqref{E:CoV_Equiv} and \eqref{E:Ave_Lp_Surf} we get
  \begin{align*}
    \int_\Gamma|M\varphi_1|\left(\int_{\varepsilon g_0}^{\varepsilon g_1}|\varphi_2^\sharp|\,dr\right)d\mathcal{H}^2 &\leq \|M\varphi_1\|_{L^2(\Gamma)}\left\{\int_\Gamma\varepsilon g\left(\int_{\varepsilon g_0}^{\varepsilon g_1}|\varphi_2^\sharp|^2\,dr\right)d\mathcal{H}^2\right\}^{1/2} \\
    &\leq c\|\varphi_1\|_{L^2(\Omega_\varepsilon)}\|\varphi_2\|_{L^2(\Omega_\varepsilon)}
  \end{align*}
  and a similar inequality for the last term of \eqref{Pf_AI:Est}.
  Hence \eqref{E:Ave_Inner} follows.

  Let $u_1,u_2\in L^2(\Omega_\varepsilon)^3$.
  By \eqref{E:Def_Tan_Ave}, \eqref{E:Def_U_TN}, and $P^T=P$ on $\Gamma$ we have
  \begin{align*}
    \overline{M_\tau u_1}\cdot u_2 = \overline{PMu_1}\cdot u_2 = \overline{Mu_1}\cdot\Bigl(\overline{P}u_2\Bigr) = \overline{Mu_1}\cdot u_{2,\tau}, \quad \overline{M_\tau u_2} = \overline{Mu_{2,\tau}}
  \end{align*}
  in $\Omega_\varepsilon$ (here $u_{2,\tau}=\overline{P}u_2$ on $\Omega_\varepsilon$).
  Thus it follows from \eqref{E:Ave_Inner} that
  \begin{multline*}
    \left|\Bigl(\overline{M_\tau u_1},u_2\Bigr)_{L^2(\Omega_\varepsilon)}-\Bigl(u_1,\overline{M_\tau u_2}\Bigr)_{L^2(\Omega_\varepsilon)}\right| \\
    = \left|\Bigl(\overline{Mu_1},u_{2,\tau}\Bigr)_{L^2(\Omega_\varepsilon)}-\Bigl(u_1,\overline{Mu_{2,\tau}}\Bigr)_{L^2(\Omega_\varepsilon)}\right| \\
    \leq c\varepsilon\|u_1\|_{L^2(\Omega_\varepsilon)}\|u_{2,\tau}\|_{L^2(\Omega_\varepsilon)}
  \end{multline*}
  and we obtain \eqref{E:AveT_Inner} by this inequality and $|u_{2,\tau}|\leq|u_2|$ in $\Omega_\varepsilon$.
\end{proof}

\subsection{Tangential derivatives of averaged functions} \label{SS:Ave_Grad}
In this subsection we give formulas and inequalities for the tangential derivatives of the average operators.

\begin{lemma} \label{L:Ave_Der}
  For $\varphi\in C^1(\Omega_\varepsilon)$ we have
  \begin{align} \label{E:Ave_Der}
    \nabla_\Gamma M\varphi = M(B\nabla\varphi)+M\bigl((\partial_n\varphi)\psi_\varepsilon\bigr) \quad\text{on}\quad \Gamma,
  \end{align}
  where the matrix-valued function $B$ and the vector field $\psi_\varepsilon$ are given by
  \begin{align} \label{E:Ave_Der_Aux}
    \begin{aligned}
      B(x) &:= \left\{I_3-d(x)\overline{W}(x)\right\}\overline{P}(x), \\
      \psi_\varepsilon(x) &:= \frac{1}{\bar{g}(x)}\left\{\bigl(d(x)-\varepsilon\bar{g}_0(x)\bigr)\overline{\nabla_\Gamma g_1}(x)+\bigl(\varepsilon\bar{g}_1(x)-d(x)\bigr)\overline{\nabla_\Gamma g_0}(x)\right\}
    \end{aligned}
  \end{align}
  for $x\in N$.
\end{lemma}

\begin{proof}
  The constant extension of $M\varphi$ is given by
  \begin{align*}
    \overline{M\varphi}(x) = \frac{1}{\varepsilon\bar{g}(x)}\int_{\varepsilon\bar{g}_0(x)}^{\varepsilon\bar{g}_1(x)}\varphi(\pi(x)+r\bar{n}(x))\,dr, \quad x \in N.
  \end{align*}
  We differentiate both sides of this equality with respect to $x\in N$ and set $x=y\in\Gamma$.
  Then by \eqref{E:ConDer_Surf}, \eqref{E:Form_W}, \eqref{E:Pi_Der}, and \eqref{E:Nor_Grad} with $d(y)=0$ we get
  \begin{align} \label{Pf_ADer:Der}
    \nabla_\Gamma M\varphi(y) &= \frac{I(y)}{\varepsilon g(y)}+\frac{1}{\varepsilon g(y)}\int_{\varepsilon g_0(y)}^{\varepsilon g_1(y)}\{I_3-rW(y)\}P(y)(\nabla\varphi)^\sharp(y,r)\,dr
  \end{align}
  for $y\in\Gamma$.
  Here we use the notations \eqref{E:Pull_Dom} and \eqref{E:Pull_Bo} and set
  \begin{align*}
    I(y) := -\frac{\nabla_\Gamma g(y)}{g(y)}\int_{\varepsilon g_0(y)}^{\varepsilon g_1(y)}\varphi^\sharp(y,r)\,dr+\varepsilon\varphi_1^\sharp(y)\nabla_\Gamma g_1(y)-\varepsilon\varphi_0^\sharp(y)\nabla_\Gamma g_0(y).
  \end{align*}
  To the right-hand side we apply
  \begin{align*}
    \varepsilon\varphi_1^\sharp(y)\nabla_\Gamma g_1(y)-\varepsilon\varphi_0^\sharp(y)\nabla_\Gamma g_0(y) &= \Bigl[(\varphi\psi_\varepsilon)^\sharp(y,r)\Bigr]_{r=\varepsilon g_0(y)}^{\varepsilon g_1(y)} \\
    &= \int_{\varepsilon g_0(y)}^{\varepsilon g_1(y)}\frac{\partial}{\partial r}\bigl((\varphi\psi_\varepsilon)^\sharp(y,r)\bigr)\,dr
  \end{align*}
  and the equalities $\partial\varphi^\sharp/\partial r=(\partial_n\varphi)^\sharp$ and $\partial\psi_\varepsilon^\sharp/\partial r=\nabla_\Gamma g/g$ by
  \begin{align*}
    \psi_\varepsilon^\sharp(y,r) = \frac{1}{g(y)}\bigl\{\bigl(r-\varepsilon g_0(y)\bigr)\nabla_\Gamma g_1(y)+\bigl(\varepsilon g_1(y)-r\bigr)\nabla_\Gamma g_0(y)\bigr\}.
  \end{align*}
  Then we have
  \begin{align*}
    I(y) = \int_{\varepsilon g_0(y)}^{\varepsilon g_1(y)}\bigl((\partial_n\varphi)\psi_\varepsilon\bigr)^\sharp(y,r)\,dr = \varepsilon g(y)\bigl[M\bigl((\partial_n\varphi)\psi_\varepsilon\bigr)\bigr](y), \quad y\in\Gamma.
  \end{align*}
  Applying this and $\{I_3-rW(y)\}P(y)=B^\sharp(y,r)$ to the right-hand side of \eqref{Pf_ADer:Der} we obtain \eqref{E:Ave_Der}.
\end{proof}

\begin{remark} \label{R:Ave_Der}
  There exists a constant $c>0$ independent of $\varepsilon$ such that
  \begin{align} \label{E:ADA_Bound}
    |B| \leq c, \quad |\psi_\varepsilon| \leq c\varepsilon \quad\text{in}\quad \Omega_\varepsilon
  \end{align}
  by \eqref{Pf_EAB:Width} and the boundedness of $W$, $P$, $\nabla_\Gamma g_0$, and $\nabla_\Gamma g_1$ on $\Gamma$.
  Moreover,
  \begin{align} \label{E:ADA_Grad_Bound}
    |\partial_kB| \leq c, \quad |\nabla\psi_\varepsilon| \leq c, \quad \left|\nabla\psi_\varepsilon-\frac{1}{\bar{g}}\bar{n}\otimes\overline{\nabla_\Gamma g}\right| \leq c\varepsilon \quad\text{in}\quad \Omega_\varepsilon,\,k=1,2,3
  \end{align}
  by \eqref{E:G_Inf}, \eqref{E:ConDer_Bound}, \eqref{Pf_EAB:Width}, $\nabla d=\bar{n}$ in $N$, and the regularity of $W$, $P$, $g_0$, and $g_1$ on $\Gamma$.
\end{remark}

\begin{lemma} \label{L:Ave_Wmp}
  There exists a constant $c>0$ independent of $\varepsilon$ such that
  \begin{align} \label{E:Ave_Wmp_Surf}
    \|M\varphi\|_{W^{m,p}(\Gamma)} &\leq c\varepsilon^{-1/p}\|\varphi\|_{W^{m,p}(\Omega_\varepsilon)}
  \end{align}
  for all $\varphi\in W^{m,p}(\Omega_\varepsilon)$ with $m=1,2$ and $p\in[1,\infty)$.
\end{lemma}

\begin{proof}
  Let $\varphi\in W^{1,p}(\Omega_\varepsilon)$.
  From \eqref{E:Ave_Lp_Surf} and \eqref{E:Ave_Der} it follows that
  \begin{align*}
    \|\nabla_\Gamma M\varphi\|_{L^p(\Gamma)} &\leq c\left(\|M(B\nabla\varphi)\|_{L^p(\Gamma)}+\left\|M\bigl((\partial_n\varphi)\psi_\varepsilon\bigr)\right\|_{L^p(\Gamma)}\right) \\
    &\leq c\varepsilon^{-1/p}\left(\|B\nabla\varphi\|_{L^p(\Omega_\varepsilon)}+\|(\partial_n\varphi)\psi_\varepsilon\|_{L^p(\Omega_\varepsilon)}\right).
  \end{align*}
  Here $B$ and $\psi_\varepsilon$ are bounded on $\Omega_\varepsilon$ uniformly in $\varepsilon$ (see Remark~\ref{R:Ave_Der}).
  Hence
  \begin{align} \label{Pf_AW:Lp_Grad}
    \|\nabla_\Gamma M\varphi\|_{L^p(\Gamma)} \leq c\varepsilon^{-1/p}\|\nabla\varphi\|_{L^p(\Omega_\varepsilon)} \leq c\varepsilon^{-1/p}\|\varphi\|_{W^{1,p}(\Omega_\varepsilon)}.
  \end{align}
  Combining \eqref{Pf_AW:Lp_Grad} with \eqref{E:Ave_Lp_Surf} we obtain \eqref{E:Ave_Wmp_Surf} with $m=1$.
  When $\varphi\in W^{2,p}(\Omega_\varepsilon)$ we apply \eqref{Pf_AW:Lp_Grad} with $\varphi$ replaced by $B\nabla\varphi$ and $(\partial_n\varphi)\psi_\varepsilon$.
  Then by \eqref{E:ADA_Bound} and \eqref{E:ADA_Grad_Bound},
  \begin{align*}
    \|\nabla_\Gamma M(B\nabla\varphi)\|_{L^p(\Gamma)}+\left\|\nabla_\Gamma M\bigl((\partial_n\varphi)\psi_\varepsilon\bigr)\right\|_{L^p(\Gamma)} \leq c\varepsilon^{-1/p}\|\varphi\|_{W^{2,p}(\Omega_\varepsilon)}.
  \end{align*}
  Therefore, applying $\nabla_\Gamma$ to \eqref{E:Ave_Der} and using the above inequality we get
  \begin{align*}
    \|\nabla_\Gamma^2 M\varphi\|_{L^p(\Gamma)} \leq c\varepsilon^{-1/p}\|\varphi\|_{W^{2,p}(\Omega_\varepsilon)}
  \end{align*}
  and \eqref{E:Ave_Wmp_Surf} with $m=2$ follows from this inequality, \eqref{E:Ave_Lp_Surf}, and \eqref{Pf_AW:Lp_Grad}.
\end{proof}

\begin{lemma} \label{L:AveT_Wmp}
  There exists a constant $c>0$ independent of $\varepsilon$ such that
  \begin{align} \label{E:AveT_Wmp_Surf}
    \|M_\tau u\|_{W^{m,p}(\Gamma)} &\leq c\varepsilon^{-1/p}\|u\|_{W^{m,p}(\Omega_\varepsilon)}
  \end{align}
  for all $u\in W^{m,p}(\Omega_\varepsilon)^3$ with $m=1,2$ and $p\in[1,\infty)$.
\end{lemma}

\begin{proof}
  We observe by the $C^4$-regularity of $P$ on $\Gamma$ and \eqref{E:Ave_Wmp_Surf} that
  \begin{align*}
    \|M_\tau u\|_{W^{m,p}(\Gamma)} = \|PMu\|_{W^{m,p}(\Gamma)} \leq c\|Mu\|_{W^{m,p}(\Gamma)} \leq c\varepsilon^{-1/p}\|u\|_{W^{m,p}(\Omega_\varepsilon)}.
  \end{align*}
  Hence \eqref{E:AveT_Wmp_Surf} is valid.
\end{proof}

\begin{lemma} \label{L:Ave_Der_Diff}
  There exists a constant $c>0$ independent of $\varepsilon$ such that
  \begin{align}
    \left\|\overline{P}\nabla\varphi-\overline{\nabla_\Gamma M\varphi}\right\|_{L^p(\Omega_\varepsilon)} &\leq c\varepsilon\|\varphi\|_{W^{2,p}(\Omega_\varepsilon)} \label{E:ADD_Dom} \\
    \left\|\overline{P}\nabla\varphi-\overline{\nabla_\Gamma M\varphi}\right\|_{L^p(\Gamma_\varepsilon^i)} &\leq c\varepsilon^{1-1/p}\|\varphi\|_{W^{2,p}(\Omega_\varepsilon)}, \quad i=0,1 \label{E:ADD_Bo}
  \end{align}
  for all $\varphi\in W^{2,p}(\Omega_\varepsilon)$ with $p\in[1,\infty)$.
\end{lemma}

\begin{proof}
  By \eqref{E:Form_W}, \eqref{E:Ave_Der}, and \eqref{E:Ave_Der_Aux} we have $\overline{P}\nabla\varphi-\overline{\nabla_\Gamma M\varphi}=u+\bar{v}$ in $\Omega_\varepsilon$, where
  \begin{align*}
    u(x) &:= \overline{P}(x)\nabla\varphi(x)-\left[M\Bigr(\overline{P}\nabla\varphi\Bigl)\right](\pi(x)), \quad x\in\Omega_\varepsilon, \\
    v(y) &:= \left[M\Bigl(d\overline{W}\nabla\varphi\Bigr)\right](y)-\bigl[M\bigl((\partial_n\varphi)\psi_\varepsilon\bigr)\bigr](y), \quad y\in\Gamma.
  \end{align*}
  We apply \eqref{E:Ave_Diff_Dom} to $u$ and use \eqref{E:NorDer_Con} to get
  \begin{align} \label{Pf_ADD:Est_Diff}
    \|u\|_{L^p(\Omega_\varepsilon)} \leq c\varepsilon\left\|\partial_n\Bigl(\overline{P}\nabla\varphi\Bigr)\right\|_{L^p(\Omega_\varepsilon)} \leq c\varepsilon\|\varphi\|_{W^{2,p}(\Omega_\varepsilon)}.
  \end{align}
  Also, from \eqref{E:Ave_Lp_Dom} and
  \begin{align*}
    \left|d\overline{W}\nabla\varphi\right| \leq c\varepsilon|\nabla\varphi|, \quad |(\partial_n\varphi)\psi_\varepsilon| \leq c\varepsilon|\nabla\varphi| \quad\text{in}\quad \Omega_\varepsilon
  \end{align*}
  by \eqref{E:ADA_Bound} and $|d|\leq c\varepsilon$ in $\Omega_\varepsilon$ it follows that
  \begin{align} \label{Pf_ADD:Est_Res}
    \|\bar{v}\|_{L^p(\Omega_\varepsilon)} \leq c\left(\left\|d\overline{W}\nabla\varphi\right\|_{L^p(\Omega_\varepsilon)}+\|(\partial_n\varphi)\psi_\varepsilon\|_{L^p(\Omega_\varepsilon)}\right) \leq c\varepsilon\|\nabla\varphi\|_{L^p(\Omega_\varepsilon)}.
  \end{align}
  Combining \eqref{Pf_ADD:Est_Diff} and \eqref{Pf_ADD:Est_Res} we obtain
  \begin{align*}
    \left\|\overline{P}\nabla\varphi-\overline{\nabla_\Gamma M\varphi}\right\|_{L^p(\Omega_\varepsilon)} \leq \|u\|_{L^p(\Omega_\varepsilon)}+\|\bar{v}\|_{L^p(\Omega_\varepsilon)} \leq c\varepsilon\|\varphi\|_{W^{2,p}(\Omega_\varepsilon)}.
  \end{align*}
  Hence \eqref{E:ADD_Dom} holds.
  Also, \eqref{E:ADD_Bo} follows from \eqref{E:NorDer_Con}, \eqref{E:Poin_Bo}, and \eqref{E:ADD_Dom}.
\end{proof}

\begin{lemma} \label{L:Ave_N_W1p}
  There exists a constant $c>0$ independent of $\varepsilon$ such that
  \begin{align} \label{E:Ave_N_W1p}
    \|Mu\cdot n\|_{W^{1,p}(\Gamma)} \leq c\varepsilon^{1-1/p}\|u\|_{W^{2,p}(\Omega_\varepsilon)}
  \end{align}
  for all $u\in W^{2,p}(\Omega_\varepsilon)^3$ with $p\in[1,\infty)$ satisfying \eqref{E:Bo_Imp} on $\Gamma_\varepsilon^0$  or on $\Gamma_\varepsilon^1$.
\end{lemma}

\begin{proof}
  By \eqref{E:Ave_Der} we have
  \begin{align*}
    \nabla_\Gamma(Mu\cdot n) = \nabla_\Gamma M(u\cdot\bar{n}) = M\bigl(B\nabla(u\cdot\bar{n})\bigr)+M(\partial_n(u\cdot\bar{n})\psi_\varepsilon) \quad\text{on}\quad \Gamma.
  \end{align*}
  To the right-hand side we apply \eqref{E:Ave_Lp_Surf} and
  \begin{align*}
    |B\nabla(u\cdot\bar{n})| \leq c\left|\overline{P}\nabla(u\cdot\bar{n})\right|, \quad |\partial_n(u\cdot\bar{n})\psi_\varepsilon| \leq c\varepsilon|\nabla u| \quad\text{in}\quad \Omega_\varepsilon
  \end{align*}
  by \eqref{E:Wein_Bound}, \eqref{E:NorDer_Con}, and \eqref{E:ADA_Bound} to deduce that
  \begin{align*}
    \|\nabla_\Gamma(Mu\cdot n)\|_{L^p(\Gamma)} &\leq c\left(\left\|M\bigl(B\nabla(u\cdot\bar{n})\bigr)\right\|_{L^p(\Gamma)}+\|M(\partial_n(u\cdot\bar{n})\psi_\varepsilon)\|_{L^p(\Gamma)}\right) \\
    &\leq c\varepsilon^{-1/p}\left(\|B\nabla(u\cdot\bar{n})\|_{L^p(\Omega_\varepsilon)}+\|\partial_n(u\cdot\bar{n})\psi_\varepsilon\|_{L^p(\Omega_\varepsilon)}\right) \\
    &\leq c\left(\varepsilon^{-1/p}\left\|\overline{P}\nabla(u\cdot\bar{n})\right\|_{L^p(\Omega_\varepsilon)}+\varepsilon^{1-1/p}\|\nabla u\|_{L^p(\Omega_\varepsilon)}\right).
  \end{align*}
  Applying \eqref{E:Poin_Dnor} to the first term on the last line we obtain
  \begin{align*}
    \|\nabla_\Gamma(Mu\cdot n)\|_{L^p(\Gamma)} \leq c\varepsilon^{1-1/p}\|u\|_{W^{2,p}(\Omega_\varepsilon)}.
  \end{align*}
  The inequality \eqref{E:Ave_N_W1p} follows from this inequality and \eqref{E:Ave_N_Lp}.
\end{proof}

Next we estimate the weighted surface divergence of the average of a vector field on $\Omega_\varepsilon$ satisfying the divergence-free condition in $\Omega_\varepsilon$ and the impermeable boundary condition \eqref{E:Bo_Imp} on $\Gamma_\varepsilon$.

\begin{lemma} \label{L:Ave_Div}
  For $p\in[1,\infty)$ let $u\in W^{1,p}(\Omega_\varepsilon)^3$ satisfy $\mathrm{div}\,u=0$ in $\Omega_\varepsilon$ and \eqref{E:Bo_Imp}.
  Then there exists a constant $c>0$ independent of $\varepsilon$ and $u$ such that
  \begin{align} \label{E:Ave_Div_Lp}
    \|\mathrm{div}_\Gamma(gMu)\|_{L^p(\Gamma)} \leq c\varepsilon^{1-1/p}\|u\|_{W^{1,p}(\Omega_\varepsilon)}.
  \end{align}
  If in addition $u\in W^{2,p}(\Omega_\varepsilon)^3$, then we have
  \begin{align} \label{E:Ave_Div_W1p}
    \|\mathrm{div}_\Gamma(gMu)\|_{W^{1,p}(\Gamma)} \leq c\varepsilon^{1-1/p}\|u\|_{W^{2,p}(\Omega_\varepsilon)}.
  \end{align}
\end{lemma}

\begin{proof}
  Let $u\in W^{1,p}(\Omega_\varepsilon)^3$ satisfy $\mathrm{div}\,u=0$ in $\Omega_\varepsilon$ and \eqref{E:Bo_Imp}.
  We use the notations \eqref{E:Pull_Dom} and \eqref{E:Pull_Bo} and define functions on $\Gamma$ by
  \begin{align} \label{Pf_ADiv:Phi_Aux}
    \begin{gathered}
      \eta_1 := \sum_{i=0,1}(-1)^i(u_i^\sharp-Mu)\cdot\tau_\varepsilon^i, \quad \eta_2 := \sum_{i=0,1}(-1)^iMu\cdot(\tau_\varepsilon^i-\nabla_\Gamma g_i), \\
      \eta_3 := -gM\left(d\,\mathrm{tr}\Bigl[\overline{W}\nabla u\Bigr]\right), \quad \eta_4 := gM(\partial_nu\cdot\psi_\varepsilon),
    \end{gathered}
  \end{align}
  where $\tau_\varepsilon^0$ and $\tau_\varepsilon^1$ are given by \eqref{E:Def_NB_Aux}.
  First we show that
  \begin{align} \label{Pf_ADiv:Sum}
    \mathrm{div}_\Gamma(gMu) = \eta_1+\eta_2+\eta_3+\eta_4 \quad\text{on}\quad \Gamma.
  \end{align}
  By \eqref{E:Form_W}, \eqref{E:Ave_Der}, and \eqref{E:Ave_Der_Aux} we have
  \begin{align*}
    g\,\mathrm{div}_\Gamma(Mu) &= g\,\mathrm{tr}[\nabla_\Gamma Mu] = gM(\mathrm{tr}[B\nabla u])+gM(\mathrm{tr}[\psi_\varepsilon\otimes\partial_n u]) \\
    &= gM\left(\mathrm{tr}\Bigl[\overline{P}\nabla u\Bigr]\right)+\eta_3+\eta_4
  \end{align*}
  on $\Gamma$.
  Moreover, since $\mathrm{div}\,u=0$ and $(\bar{n}\otimes\bar{n})\nabla u=\bar{n}\otimes\partial_nu$ in $\Omega_\varepsilon$,
  \begin{align*}
    \mathrm{tr}\Bigl[\overline{P}\nabla u\Bigr] = \mathrm{div}\,u-\mathrm{tr}[(\bar{n}\otimes\bar{n})\nabla u] = -\bar{n}\cdot\partial_nu \quad\text{in}\quad \Omega_\varepsilon.
  \end{align*}
  By these equalities and $\mathrm{div}_\Gamma(gMu)=\nabla_\Gamma g\cdot Mu+g\,\mathrm{div}_\Gamma(Mu)$ on $\Gamma$ we get
  \begin{align} \label{Pf_ADiv:For1}
    \mathrm{div}_\Gamma(gMu) = \nabla_\Gamma g\cdot Mu-gM(\partial_nu\cdot\bar{n})+\eta_3+\eta_4 \quad\text{on}\quad \Gamma.
  \end{align}
  Let us calculate the second term on the right-hand side.
  Since
  \begin{align*}
    (\partial_nu\cdot \bar{n})^\sharp(y,r) = [\partial_n(u\cdot\bar{n})]^\sharp(y,r) = \frac{\partial}{\partial r}\Bigl((u\cdot\bar{n})^\sharp(y,r)\Bigr)
  \end{align*}
  for $y\in\Gamma$ and $r\in(\varepsilon g_0(y),\varepsilon g_1(y))$ by $\partial_n\bar{n}=0$ in $\Omega_\varepsilon$, we have
  \begin{align*}
    g(y)[M(\partial_nu\cdot\bar{n})](y) &= \frac{1}{\varepsilon}\int_{\varepsilon g_0(y)}^{\varepsilon g_1(y)}\frac{\partial}{\partial r}\Bigl((u\cdot\bar{n})^\sharp(y,r)\Bigr)\,dr \\
    &= \frac{1}{\varepsilon}\{(u\cdot\bar{n})^\sharp(y,\varepsilon g_1(y))-(u\cdot\bar{n})^\sharp(y,\varepsilon g_0(y))\}
  \end{align*}
  for $y\in\Gamma$.
  Moreover, since $u$ satisfies \eqref{E:Bo_Imp} on $\Gamma_\varepsilon$,
  \begin{align*}
    (u\cdot\bar{n})^\sharp(y,\varepsilon g_i(y)) = \varepsilon(u\cdot\bar{\tau}_\varepsilon^i)^\sharp(y,\varepsilon g_i(y)) = \varepsilon(u_i^\sharp\cdot\tau_\varepsilon^i)(y), \quad y\in\Gamma
  \end{align*}
  by \eqref{E:Exp_Bo}.
  Hence
  \begin{align*}
    gM(\partial_nu\cdot\bar{n}) = u_1^\sharp\cdot\tau_\varepsilon^1-u_0^\sharp\cdot\tau_\varepsilon^0 = \nabla_\Gamma g\cdot Mu-\eta_1-\eta_2 \quad\text{on}\quad \Gamma.
  \end{align*}
  Combining this with \eqref{Pf_ADiv:For1} we obtain \eqref{Pf_ADiv:Sum}.

  Let us estimate $\eta_1,\dots,\eta_4$ in $L^p(\Gamma)$.
  Noting that
  \begin{align} \label{Pf_ADiv:Pull_Bar}
    \eta(y) = \bar{\eta}(y+\varepsilon g_i(y)n(y)) = \bar{\eta}_i^\sharp(y), \quad y\in\Gamma,\,i=0,1
  \end{align}
  for a function $\eta$ on $\Gamma$, we apply \eqref{E:Tau_Bound}, \eqref{E:Lp_CoV_Surf}, and \eqref{E:Ave_Diff_Bo} to $\eta_1$ to deduce that
  \begin{align} \label{Pf_ADiv:Lp_Phi1}
    \|\eta_1\|_{L^p(\Gamma)} \leq c\sum_{i=0,1}\left\|u-\overline{Mu}\right\|_{L^p(\Gamma_\varepsilon^i)} \leq c\varepsilon^{1-1/p}\|u\|_{W^{1,p}(\Omega_\varepsilon)}.
  \end{align}
  The first inequality of \eqref{E:Tau_Diff} and \eqref{E:Ave_Lp_Surf} imply that
  \begin{align} \label{Pf_ADiv:Lp_Phi2}
    \|\eta_2\|_{L^p(\Gamma)} \leq c\varepsilon\|Mu\|_{L^p(\Gamma)} \leq c\varepsilon^{1-1/p}\|u\|_{L^p(\Omega_\varepsilon)}.
  \end{align}
  To $\eta_3$ and $\eta_4$ we apply \eqref{E:Ave_Lp_Surf} and
  \begin{align*}
    \left|d\,\mathrm{tr}\Bigl[\overline{W}\nabla u\Bigr]\right|\leq c\varepsilon|\nabla u|, \quad |\partial_nu\cdot\psi_\varepsilon|\leq c\varepsilon|\nabla u| \quad\text{in}\quad \Omega_\varepsilon
  \end{align*}
  by \eqref{E:ADA_Bound}, $|d|\leq c\varepsilon$ in $\Omega_\varepsilon$, and the boundedness of $W$ on $\Gamma$ to get
  \begin{align} \label{Pf_ADiv:Lp_Phi34}
    \begin{aligned}
      \|\eta_3\|_{L^p(\Gamma)} &\leq c\varepsilon^{-1/p}\left\|d\,\mathrm{tr}\Bigl[\overline{W}\nabla u\Bigr]\right\|_{L^p(\Omega_\varepsilon)} \leq c\varepsilon^{1-1/p}\|\nabla u\|_{L^p(\Omega_\varepsilon)}, \\
      \|\eta_4\|_{L^p(\Gamma)} &\leq c\varepsilon^{-1/p}\|\partial_nu\cdot\psi_\varepsilon\|_{L^p(\Omega_\varepsilon)} \leq c\varepsilon^{1-1/p}\|\nabla u\|_{L^p(\Omega_\varepsilon)}.
    \end{aligned}
  \end{align}
  Applying \eqref{Pf_ADiv:Lp_Phi1}--\eqref{Pf_ADiv:Lp_Phi34} to \eqref{Pf_ADiv:Sum} we obtain \eqref{E:Ave_Div_Lp}.

  Now we assume $u\in W^{2,p}(\Omega_\varepsilon)^3$ and estimate $\nabla_\Gamma\eta_1,\dots,\nabla_\Gamma\eta_4$ in $L^p(\Gamma)$.
  Let
  \begin{align*}
    G_i(y) := \nabla_\Gamma g_i(y)\otimes n(y)-g_i(y)W(y), \quad y\in\Gamma,\,i=0,1.
  \end{align*}
  Then $\nabla_\Gamma u_i^\sharp=(P+\varepsilon G_i)(\nabla u)_i^\sharp$ on $\Gamma$ for $i=0,1$ by \eqref{E:Pull_Bo} and
  \begin{align*}
    \nabla_\Gamma\bigl(y+\varepsilon g_i(y)n(y)\bigr) = P(y)+\varepsilon G_i(y), \quad y\in\Gamma, \, i=0,1.
  \end{align*}
  Hence
  \begin{align*}
    \nabla_\Gamma\eta_1 &= \sum_{i=0,1}(-1)^i\{(\nabla_\Gamma u_i^\sharp-\nabla_\Gamma Mu)\tau_\varepsilon^i+(\nabla_\Gamma\tau_\varepsilon^i)(u_i^\sharp-Mu)\} \\
    &= \sum_{i=0,1}(-1)^i\Bigl[\bigl\{\bigl(P(\nabla u)_i^\sharp-\nabla_\Gamma Mu\bigr)+\varepsilon G_i(\nabla u)_i^\sharp\bigr\}\tau_\varepsilon^i+(\nabla_\Gamma\tau_\varepsilon^i)(u_i^\sharp-Mu)\Bigr]
  \end{align*}
  on $\Gamma$.
  Since $G_0$ and $G_1$ are bounded on $\Gamma$, we see by \eqref{E:Tau_Bound} that
  \begin{align*}
    |\nabla_\Gamma\eta_1| \leq c\sum_{i=0,1}(|P(\nabla u)_i^\sharp-\nabla_\Gamma Mu|+\varepsilon|(\nabla u)_i^\sharp|+|u_i^\sharp-Mu|) \quad\text{on}\quad \Gamma.
  \end{align*}
  From this inequality, \eqref{E:Lp_CoV_Surf}, and \eqref{Pf_ADiv:Pull_Bar} we deduce that
  \begin{align*}
    &\|\nabla_\Gamma\eta_1\|_{L^p(\Gamma)} \\
    &\qquad \leq c\sum_{i=0,1}\left(\|P(\nabla u)_i^\sharp-\nabla_\Gamma Mu\|_{L^p(\Gamma)}+\varepsilon\|(\nabla u)_i^\sharp\|_{L^p(\Gamma)}+\|u_i^\sharp-Mu\|_{L^p(\Gamma)}\right) \\
    &\qquad \leq c\sum_{i=0,1}\left(\left\|\overline{P}\nabla u-\overline{\nabla_\Gamma Mu}\right\|_{L^p(\Gamma_\varepsilon^i)}+\varepsilon\|\nabla u\|_{L^p(\Gamma_\varepsilon^i)}+\left\|u-\overline{Mu}\right\|_{L^p(\Gamma_\varepsilon^i)}\right).
  \end{align*}
  To the right-hand side we apply \eqref{E:Poin_Bo}, \eqref{E:Ave_Diff_Bo}, and \eqref{E:ADD_Bo} to obtain
  \begin{align} \label{Pf_ADiv:Lp_Gr_Phi1}
    \|\nabla_\Gamma\eta_1\|_{L^p(\Gamma)} \leq c\varepsilon^{1-1/p}\|u\|_{W^{2,p}(\Omega_\varepsilon)}.
  \end{align}
  Next we estimate the $L^p(\Gamma)$-norm of $\nabla_\Gamma\eta_2$.
  From \eqref{E:Tau_Diff} and
  \begin{align*}
    \nabla_\Gamma\eta_2 = \sum_{i=0,1}(-1)^i\{(\nabla_\Gamma Mu)(\tau_\varepsilon^i-\nabla_\Gamma g_i)+(\nabla_\Gamma\tau_\varepsilon^i-\nabla_\Gamma^2 g_i)Mu\} \quad\text{on}\quad \Gamma
  \end{align*}
  it follows that $|\nabla_\Gamma\eta_2|\leq c\varepsilon(|Mu|+|\nabla_\Gamma Mu|)$ on $\Gamma$.
  Hence by \eqref{E:Ave_Wmp_Surf} with $m=1$,
  \begin{align} \label{Pf_ADiv:Lp_Gr_Phi2}
    \|\nabla_\Gamma\eta_2\|_{L^p(\Gamma)} \leq c\varepsilon\|Mu\|_{W^{1,p}(\Gamma)} \leq c\varepsilon^{1-1/p}\|u\|_{W^{1,p}(\Omega_\varepsilon)}.
  \end{align}
  To estimate the tangential gradient of $\eta_3$ in $L^p(\Gamma)$, we define
  \begin{align*}
    \phi := \mathrm{tr}\Bigl[\overline{W}\nabla u\Bigr] \quad\text{on}\quad \Omega_\varepsilon
  \end{align*}
  so that $\eta_3=-gM(d\phi)$ on $\Gamma$.
  Then by \eqref{E:Ave_Der} we have
  \begin{align*}
    \nabla_\Gamma\eta_3 = -M(d\phi)\nabla_\Gamma g-g\bigl\{M(\phi B\nabla d)+M(dB\nabla\phi)+M\bigl(\partial_n(d\phi)\psi_\varepsilon\bigr)\bigr\} \quad\text{on}\quad \Gamma.
  \end{align*}
  Here the second term on the right-hand side vanishes since $B\nabla d=B\bar{n}=0$ in $\Omega_\varepsilon$ by \eqref{E:Ave_Der_Aux} and $Pn=0$ on $\Gamma$.
  Thus \eqref{E:Ave_Lp_Surf} implies that
  \begin{align*}
    \|\nabla_\Gamma\eta_3\|_{L^p(\Gamma)} &\leq c\left(\|M(d\phi)\|_{L^p(\Gamma)}+\|M(dB\nabla\phi)\|_{L^p(\Gamma)}+\left\|M\bigl(\partial_n(d\phi)\psi_\varepsilon\bigr)\right\|_{L^p(\Gamma)}\right) \\
    &\leq c\varepsilon^{-1/p}\left(\|d\phi\|_{L^p(\Omega_\varepsilon)}+\|dB\nabla\phi\|_{L^p(\Omega_\varepsilon)}+\|\partial_n(d\phi)\psi_\varepsilon\|_{L^p(\Omega_\varepsilon)}\right).
  \end{align*}
  Moreover, noting that $W$ is of class $C^3$ on $\Gamma$ and
  \begin{align*}
    |d| \leq c\varepsilon \quad\text{in}\quad \Omega_\varepsilon, \quad \partial_nd = \bar{n}\cdot\nabla d = |\bar{n}|^2 = 1 \quad\text{in}\quad N,
  \end{align*}
  we use \eqref{E:ConDer_Bound}, \eqref{E:NorDer_Con}, and \eqref{E:ADA_Bound} to get
  \begin{align*}
    |d\phi| \leq c\varepsilon|\nabla u|, \quad |dB\nabla\phi| \leq c\varepsilon(|\nabla u|+|\nabla^2u|), \quad |\partial_n(d\phi)\psi_\varepsilon| \leq c\varepsilon(|\nabla u|+|\nabla^2u|)
  \end{align*}
  in $\Omega_\varepsilon$.
  Hence we obtain
  \begin{align} \label{Pf_ADiv:Lp_Gr_Phi3}
    \|\nabla_\Gamma\eta_3\|_{L^p(\Gamma)} \leq c\varepsilon^{1-1/p}\|u\|_{W^{2,p}(\Omega_\varepsilon)}.
  \end{align}
  Let us estimate $\nabla_\Gamma\eta_4$ in $L^p(\Gamma)$.
  Setting $\xi:=\partial_nu\cdot\psi_\varepsilon$ on $\Omega_\varepsilon$ we have
  \begin{align*}
    \nabla_\Gamma\eta_4=(M\xi)\nabla_\Gamma g+g\bigl\{M(B\nabla\xi)+M\bigl((\partial_n\xi)\psi_\varepsilon\bigr)\bigr\} \quad\text{on}\quad \Gamma
  \end{align*}
  by \eqref{E:Ave_Der}.
  From this equality and \eqref{E:Ave_Lp_Surf} we deduce that
  \begin{align} \label{Pf_ADiv:GP4_Aux_1}
    \begin{aligned}
      \|\nabla_\Gamma\eta_4\|_{L^p(\Gamma)} &\leq c\left(\|M\xi\|_{L^p(\Gamma)}+\|M(B\nabla\xi)\|_{L^p(\Gamma)}+\left\|M\bigl((\partial_n\xi)\psi_\varepsilon\bigr)\right\|_{L^p(\Gamma)}\right) \\
      &\leq c\varepsilon^{-1/p}\left(\|\xi\|_{L^p(\Omega_\varepsilon)}+\|B\nabla\xi\|_{L^p(\Omega_\varepsilon)}+\|(\partial_n\xi)\psi_\varepsilon\|_{L^p(\Omega_\varepsilon)}\right).
    \end{aligned}
  \end{align}
  We apply \eqref{E:ADA_Bound} and \eqref{E:ADA_Grad_Bound} to $\xi=\partial_nu\cdot\psi_\varepsilon$ and $(\partial_n\xi)\psi_\varepsilon$ to get
  \begin{align} \label{Pf_ADiv:GP4_Aux_2}
    \begin{aligned}
      \|\xi\|_{L^p(\Omega_\varepsilon)} &\leq c\varepsilon\|\nabla u\|_{L^p(\Omega_\varepsilon)}, \\
      \|(\partial_n\xi)\psi_\varepsilon\|_{L^p(\Omega_\varepsilon)} &\leq c\varepsilon\left(\|\nabla u\|_{L^p(\Omega_\varepsilon)}+\|\nabla^2u\|_{L^p(\Omega_\varepsilon)}\right).
      \end{aligned}
  \end{align}
  Moreover, by \eqref{E:Wein_Bound}, \eqref{E:Ave_Der_Aux}, \eqref{E:ADA_Grad_Bound}, and $P(n\otimes\nabla_\Gamma g)=(Pn)\otimes\nabla_\Gamma g=0$ on $\Gamma$,
  \begin{align*}
    |B\nabla\psi_\varepsilon| \leq c\left|\overline{P}\nabla\psi_\varepsilon\right| = c\left|\overline{P}\left(\nabla\psi_\varepsilon-\frac{1}{\bar{g}}\bar{n}\otimes\overline{\nabla_\Gamma g}\right)\right| \leq c\varepsilon \quad\text{in}\quad \Omega_\varepsilon.
  \end{align*}
  Using this inequality and \eqref{E:ADA_Bound} to $B\nabla\xi=B\{\nabla(\partial_nu)\}\psi_\varepsilon+B(\nabla\psi_\varepsilon)\partial_nu$ we get
  \begin{align} \label{Pf_ADiv:GP4_Aux_3}
    \|B\nabla\xi\|_{L^p(\Omega_\varepsilon)} \leq c\varepsilon\left(\|\nabla u\|_{L^p(\Omega_\varepsilon)}+\|\nabla^2u\|_{L^p(\Omega_\varepsilon)}\right).
  \end{align}
  From \eqref{Pf_ADiv:GP4_Aux_1}--\eqref{Pf_ADiv:GP4_Aux_3} it follows that
  \begin{align} \label{Pf_ADiv:Lp_Gr_Phi4}
    \|\nabla_\Gamma\eta_4\|_{L^p(\Omega_\varepsilon)}\leq c\varepsilon^{1-1/p}\|u\|_{W^{2,p}(\Omega_\varepsilon)}.
  \end{align}
  Finally, from \eqref{Pf_ADiv:Sum}, \eqref{Pf_ADiv:Lp_Gr_Phi1}--\eqref{Pf_ADiv:Lp_Gr_Phi3}, and \eqref{Pf_ADiv:Lp_Gr_Phi4} we deduce that
  \begin{align*}
    \left\|\nabla_\Gamma\bigl(\mathrm{div}_\Gamma(gMu)\bigr)\right\|_{L^p(\Gamma)} \leq \sum_{j=1}^4\|\nabla_\Gamma\eta_j\|_{L^p(\Gamma)} \leq c\varepsilon^{1-1/p}\|u\|_{W^{2,p}(\Omega_\varepsilon)}
  \end{align*}
  and conclude that \eqref{E:Ave_Div_W1p} is valid.
\end{proof}

\begin{lemma} \label{L:ADiv_Tan}
  For $p\in[1,\infty)$ let $u\in W^{1,p}(\Omega_\varepsilon)^3$ satisfy $\mathrm{div}\,u=0$ in $\Omega_\varepsilon$ and \eqref{E:Bo_Imp}.
  Then there exists a constant $c>0$ independent of $\varepsilon$ and $u$ such that
  \begin{align} \label{E:ADiv_Tan_Lp}
    \|\mathrm{div}_\Gamma(gM_\tau u)\|_{L^p(\Gamma)} \leq c\varepsilon^{1-1/p}\|u\|_{W^{1,p}(\Omega_\varepsilon)}.
  \end{align}
  If in addition $u\in W^{2,p}(\Omega_\varepsilon)^3$, then we have
  \begin{align} \label{E:ADiv_Tan_W1p}
    \|\mathrm{div}_\Gamma(gM_\tau u)\|_{W^{1,p}(\Gamma)} \leq c\varepsilon^{1-1/p}\|u\|_{W^{2,p}(\Omega_\varepsilon)}.
  \end{align}
\end{lemma}

\begin{proof}
  First note that, for a function $\eta$ on $\Gamma$ we have
  \begin{align*}
    \mathrm{div}_\Gamma(\eta n) = \nabla_\Gamma\eta\cdot n+\eta\,\mathrm{div}_\Gamma n = -\eta H \quad\text{on}\quad \Gamma
  \end{align*}
  by \eqref{E:P_TGr} and \eqref{E:Def_WH}.
  By this equality and $M_\tau u=Mu-(Mu\cdot n)n$ on $\Gamma$ we get
  \begin{align*}
    \mathrm{div}_\Gamma(gM_\tau u) = \mathrm{div}_\Gamma(gMu)+g(Mu\cdot n)H \quad\text{on}\quad \Gamma.
  \end{align*}
  From this equality and the regularity of $g$ and $H$ on $\Gamma$ it follows that
  \begin{align*}
    |\mathrm{div}_\Gamma(gM_\tau u)| &\leq c(|\mathrm{div}_\Gamma(gMu)|+|Mu\cdot n|), \\
    \left|\nabla_\Gamma\bigl(\mathrm{div}_\Gamma(gM_\tau u)\bigr)\right| &\leq c\left(\left|\nabla_\Gamma\bigl(\mathrm{div}_\Gamma(gMu)\bigr)\right|+|Mu\cdot n|+|\nabla_\Gamma(Mu\cdot n)|\right)
  \end{align*}
  on $\Gamma$.
  These inequalities, \eqref{E:Ave_N_Lp}, and \eqref{E:Ave_N_W1p}--\eqref{E:Ave_Div_W1p} imply \eqref{E:ADiv_Tan_Lp} and \eqref{E:ADiv_Tan_W1p}.
\end{proof}

For a vector field $u$ on $\Omega_\varepsilon$ let $\partial_nu$ be its derivative in the normal direction of $\Gamma$.
In Lemma~\ref{L:PDnU_WU} we observed that the tangential component of $\partial_nu$ with respect to $\Gamma$ is compared with the vector field $-\overline{W}u$.
Next we derive a similar relation for the normal component of $\partial_nu$ by using Lemmas~\ref{L:Ave_Der_Diff} and~\ref{L:Ave_Div}.

\begin{lemma} \label{L:DnU_N_Ave}
  There exists a constant $c>0$ independent of $\varepsilon$ such that
  \begin{align} \label{E:DnU_N_Ave}
    \left\|\partial_nu\cdot\bar{n}-\frac{1}{\bar{g}}\overline{M_\tau u}\cdot\overline{\nabla_\Gamma g}\right\|_{L^p(\Omega_\varepsilon)} \leq c\varepsilon\|u\|_{W^{2,p}(\Omega_\varepsilon)}
  \end{align}
  for all $u\in W^{2,p}(\Omega_\varepsilon)^3$ with $p\in[1,\infty)$ satisfying $\mathrm{div}\,u=0$ in $\Omega_\varepsilon$ and \eqref{E:Bo_Imp}.
\end{lemma}

\begin{proof}
  Since $\partial_nu=(\bar{n}\cdot\nabla)u=(\nabla u)^Tn$ in $\Omega_\varepsilon$, we have
  \begin{align*}
    \mathrm{tr}[(\bar{n}\otimes\bar{n})\nabla u] = \mathrm{tr}[\bar{n}\otimes\{(\nabla u)^T\bar{n}\}] = \mathrm{tr}[\bar{n}\otimes\partial_nu] = \partial_nu\cdot\bar{n} \quad\text{in}\quad \Omega_\varepsilon.
  \end{align*}
  By this equality, $n\otimes n=I_3-P$ on $\Gamma$, and $\mathrm{div}\,u=\mathrm{tr}[\nabla u]=0$ in $\Omega_\varepsilon$,
  \begin{align*}
    \partial_nu\cdot\bar{n} = \mathrm{tr}[\nabla u]-\mathrm{tr}\Bigl[\overline{P}\nabla u\Bigr] = -\mathrm{tr}\Bigl[\overline{P}\nabla u\Bigr] \quad\text{in}\quad \Omega_\varepsilon.
  \end{align*}
  Also, since $\nabla_\Gamma g$ is tangential on $\Gamma$,
  \begin{align*}
    \frac{1}{g}M_\tau u\cdot\nabla_\Gamma g &= \frac{1}{g}Mu\cdot\nabla_\Gamma g = \frac{1}{g}\mathrm{div}_\Gamma(gMu)-\mathrm{div}_\Gamma(Mu) \\
    &= \frac{1}{g}\mathrm{div}_\Gamma(gMu)-\mathrm{tr}[\nabla_\Gamma Mu]
  \end{align*}
  on $\Gamma$.
  From these equalities and \eqref{E:G_Inf} we deduce that
  \begin{align*}
    \left\|\partial_nu\cdot\bar{n}-\frac{1}{\bar{g}}\overline{M_\tau u}\cdot\overline{\nabla_\Gamma g}\right\|_{L^p(\Omega_\varepsilon)} \leq \left\|\overline{P}\nabla u-\overline{\nabla_\Gamma Mu}\right\|_{L^p(\Omega_\varepsilon)}+c\left\|\overline{\mathrm{div}_\Gamma(gMu)}\right\|_{L^p(\Omega_\varepsilon)}
  \end{align*}
  and we apply \eqref{E:Con_Lp}, \eqref{E:ADD_Dom}, and \eqref{E:Ave_Div_Lp} to the right-hand side to get \eqref{E:DnU_N_Ave}.
\end{proof}

\subsection{Decomposition of a vector field into the average and residual parts} \label{SS:Ave_Decom}
In the study of the Navier--Stokes equations in a three-dimensional thin domain it is convenient to decompose a three-dimensional vector field into an almost two-dimensional one and a residual term and analyze them separately.
To give a good decomposition of a vector field on $\Omega_\varepsilon$ we use the impermeable extension operator $E_\varepsilon$ given by \eqref{E:Def_ExImp} and the averaged tangential component of a vector field on $\Omega_\varepsilon$.

\begin{definition} \label{D:Def_ExAve}
  For a vector field $u$ on $\Omega_\varepsilon$ we define the average part of $u$ by
  \begin{align} \label{E:Def_ExAve}
    u^a(x) := E_\varepsilon M_\tau u(x) = \overline{M_\tau u}(x)+\left\{\overline{M_\tau u}(x)\cdot\Psi_\varepsilon(x)\right\}\bar{n}(x), \quad x\in N,
  \end{align}
  where $\Psi_\varepsilon$ is the vector field given by \eqref{E:Def_ExAux} and $M_\tau u$ is the averaged tangential component of $u$ given by \eqref{E:Def_Tan_Ave}.
  We also call $u^r:=u-u^a$ the residual part of $u$.
\end{definition}

By Lemmas~\ref{L:ExImp_Wmp}, \ref{L:AveT_Lp}, and \ref{L:AveT_Wmp} we see that if $u\in W^{m,p}(\Omega_\varepsilon)^3$ with $m=0,1,2$ and $p\in[1,\infty)$ then $u^a$ and $u^r$ belong to the same space.

\begin{lemma} \label{L:Wmp_UaUr}
  There exists a constant $c>0$ independent of $\varepsilon$ such that
  \begin{align} \label{E:Wmp_UaUr}
    \|u^a\|_{W^{m,p}(\Omega_\varepsilon)} \leq c\|u\|_{W^{m,p}(\Omega_\varepsilon)}, \quad \|u^r\|_{W^{m,p}(\Omega_\varepsilon)} \leq c\|u\|_{W^{m,p}(\Omega_\varepsilon)}
  \end{align}
  for all $u\in W^{m,p}(\Omega_\varepsilon)^3$ with $m=0,1,2$ and $p\in[1,\infty)$.
\end{lemma}

Since the average part $u^a$ can be seen as almost two-dimensional, we expect to have a good $L^2(\Omega_\varepsilon)$-estimate for the product of $u^a$ and a function on $\Omega_\varepsilon$.
Indeed, we can apply the following product estimate on $\Omega_\varepsilon$ to $u^a$.

\begin{lemma} \label{L:Prod}
  There exists a constant $c>0$ independent of $\varepsilon$ such that
  \begin{align} \label{E:Prod_Surf}
    \|\bar{\eta}\varphi\|_{L^2(\Omega_\varepsilon)} &\leq c\|\eta\|_{L^2(\Gamma)}^{1/2}\|\eta\|_{H^1(\Gamma)}^{1/2}\|\varphi\|_{L^2(\Omega_\varepsilon)}^{1/2}\|\varphi\|_{H^1(\Omega_\varepsilon)}^{1/2}
  \end{align}
  for all $\eta\in H^1(\Gamma)$ and $\varphi\in H^1(\Omega_\varepsilon)$.
\end{lemma}

\begin{proof}
  Throughout the proof we use the notation \eqref{E:Pull_Dom} and suppress the arguments of functions.
  By \eqref{E:CoV_Equiv} and \eqref{E:Def_Ave} we have
  \begin{align*}
    \|\bar{\eta}\varphi\|_{L^2(\Omega_\varepsilon)}^2 \leq c\int_\Gamma|\eta|^2\left(\int_{\varepsilon g_0}^{\varepsilon g_1}|\varphi^\sharp|^2\,dr\right)d\mathcal{H}^2 = c\varepsilon\int_\Gamma g|\eta|^2M(|\varphi|^2)\,d\mathcal{H}^2.
  \end{align*}
  Noting that $g$ is bounded on $\Gamma$, we apply H\"{o}lder's inequality to the last term to get
  \begin{align} \label{Pf_Pr:Est_L2}
    \|\bar{\eta}\varphi\|_{L^2(\Omega_\varepsilon)}^2 \leq c\varepsilon\|\eta\|_{L^4(\Gamma)}^2\|M(|\varphi|^2)\|_{L^2(\Gamma)}.
  \end{align}
  The $L^4(\Gamma)$-norm of $\eta$ is estimated by \eqref{E:La_Surf}.
  To estimate the last term of \eqref{Pf_Pr:Est_L2} let us show $M(|\varphi|^2)\in W^{1,1}(\Gamma)$.
  By $M(|\varphi|^2)\geq0$ on $\Gamma$, \eqref{E:G_Inf}, and \eqref{E:CoV_Equiv},
  \begin{align} \label{Pf_Pr:L1_MPhi2}
    \|M(|\varphi|^2)\|_{L^1(\Gamma)} = \int_\Gamma\frac{1}{\varepsilon g}\left(\int_{\varepsilon g_0}^{\varepsilon g_1}|\varphi^\sharp|^2\,dr\right)d\mathcal{H}^2 \leq c\varepsilon^{-1}\|\varphi\|_{L^2(\Omega_\varepsilon)}^2.
  \end{align}
  Also, by \eqref{E:Ave_Der}, \eqref{E:ADA_Bound}, and $\nabla(|\varphi|^2)=2\varphi\nabla\varphi$ in $\Omega_\varepsilon$ we get
  \begin{align*}
    |\nabla_\Gamma M(|\varphi|^2)| \leq |M\bigl(B\nabla(|\varphi|^2)\bigr)|+|M\bigl(\partial_n(|\varphi|^2)\psi_\varepsilon\bigr)| \leq cM(|\varphi\nabla\varphi|) \quad\text{on}\quad \Gamma.
  \end{align*}
  Hence from \eqref{E:Ave_Lp_Surf} and H\"{o}lder's inequality we deduce that
  \begin{align} \label{Pf_Pr:L1_Grad}
    \begin{aligned}
      \|\nabla_\Gamma M(|\varphi|^2)\|_{L^1(\Gamma)} &\leq c\|M(|\varphi\nabla\varphi|)\|_{L^1(\Gamma)} \leq c\varepsilon^{-1}\|\varphi\nabla\varphi\|_{L^1(\Omega_\varepsilon)} \\
      &\leq c\varepsilon^{-1}\|\varphi\|_{L^2(\Omega_\varepsilon)}\|\nabla\varphi\|_{L^2(\Omega_\varepsilon)}.
    \end{aligned}
  \end{align}
  Now we observe that the Sobolev embedding $W^{1,1}(\Gamma)\hookrightarrow L^2(\Gamma)$ is valid since $\Gamma\subset\mathbb{R}^3$ is a two-dimensional compact surface without boundary (see e.g.~\cite{Au98}*{Theorem~2.20}).
  By this fact, \eqref{Pf_Pr:L1_MPhi2}, \eqref{Pf_Pr:L1_Grad}, and $\|\varphi\|_{L^2(\Omega_\varepsilon)}\leq\|\varphi\|_{H^1(\Omega_\varepsilon)}$ we obtain
  \begin{align*}
    \|M(|\varphi|^2)\|_{L^2(\Gamma)} \leq c\|M(|\varphi|^2)\|_{W^{1,1}(\Gamma)} \leq c\varepsilon^{-1}\|\varphi\|_{L^2(\Omega_\varepsilon)}\|\varphi\|_{H^1(\Omega_\varepsilon)}.
  \end{align*}
  Finally, we apply the above inequality and \eqref{E:La_Surf} to \eqref{Pf_Pr:Est_L2} to get
  \begin{align*}
    \|\bar{\eta}\varphi\|_{L^2(\Omega_\varepsilon)}^2 \leq c\|\eta\|_{L^2(\Gamma)}\|\eta\|_{H^1(\Gamma)}\|\varphi\|_{L^2(\Omega_\varepsilon)}\|\varphi\|_{H^1(\Omega_\varepsilon)},
  \end{align*}
  which shows \eqref{E:Prod_Surf}.
\end{proof}

\begin{lemma} \label{L:Prod_Ua}
  For $\varphi\in H^1(\Omega_\varepsilon)$, $u\in H^1(\Omega_\varepsilon)^3$, and $u^a$ given by \eqref{E:Def_ExAve} we have
  \begin{align} \label{E:Prod_Ua}
    \bigl\|\,|u^a|\,\varphi\bigr\|_{L^2(\Omega_\varepsilon)} &\leq c\varepsilon^{-1/2}\|\varphi\|_{L^2(\Omega_\varepsilon)}^{1/2}\|\varphi\|_{H^1(\Omega_\varepsilon)}^{1/2}\|u\|_{L^2(\Omega_\varepsilon)}^{1/2}\|u\|_{H^1(\Omega_\varepsilon)}^{1/2}
  \end{align}
  with a constant $c>0$ independent of $\varepsilon$, $\varphi$, and $u$.
  If in addition $u\in H^2(\Omega_\varepsilon)^3$, then
  \begin{align} \label{E:Prod_Grad_Ua}
    \bigl\|\,|\nabla u^a|\,\varphi\bigr\|_{L^2(\Omega_\varepsilon)} &\leq c\varepsilon^{-1/2}\|\varphi\|_{L^2(\Omega_\varepsilon)}^{1/2}\|\varphi\|_{H^1(\Omega_\varepsilon)}^{1/2}\|u\|_{H^1(\Omega_\varepsilon)}^{1/2}\|u\|_{H^2(\Omega_\varepsilon)}^{1/2}.
  \end{align}
\end{lemma}

\begin{proof}
  Since
  \begin{align*}
    |u^a| \leq (1+|\Psi_\varepsilon|)\left|\overline{M_\tau u}\right| \leq c\left|\overline{Mu}\right| \quad\text{in}\quad \Omega_\varepsilon
  \end{align*}
  by \eqref{E:ExAux_Bound} and $|M_\tau u|=|PMu|\leq|Mu|$ on $\Gamma$, we have
  \begin{align*}
    \bigl\|\,|u^a|\,\varphi\bigr\|_{L^2(\Omega_\varepsilon)} \leq c\left\|\,\left|\overline{Mu}\right|\varphi\right\|_{L^2(\Omega_\varepsilon)}.
  \end{align*}
  Hence we get \eqref{E:Prod_Ua} by applying \eqref{E:Prod_Surf} with $\eta=|Mu|$ to the right-hand side of this inequality and using \eqref{E:Ave_Lp_Surf} and \eqref{E:Ave_Wmp_Surf}.

  Let us prove \eqref{E:Prod_Grad_Ua}.
  We differentiate both sides of \eqref{E:Def_ExAve} to get
  \begin{align*}
    \nabla u^a = \nabla\Bigl(\overline{M_\tau u}\Bigr)+\left[\left\{\nabla\Bigl(\overline{M_\tau u}\Bigr)\right\}\Psi_\varepsilon+(\nabla\Psi_\varepsilon)\overline{M_\tau u}\right]\otimes\bar{n}+\Bigl(\overline{M_\tau u}\cdot\Psi_\varepsilon\Bigr)\nabla\bar{n}
  \end{align*}
  in $\Omega_\varepsilon$.
  Hence by \eqref{E:ConDer_Bound}, \eqref{E:NorG_Bound}, \eqref{E:ExAux_Bound}, $M_\tau u=PMu$ on $\Gamma$, and $P\in C^4(\Gamma)^{3\times 3}$,
  \begin{align*}
    |\nabla u^a| \leq c\left(\left|\overline{M_\tau u}\right|+\left|\overline{\nabla_\Gamma M_\tau u}\right|\right) \leq c\left(\left|\overline{Mu}\right|+\left|\overline{\nabla_\Gamma Mu}\right|\right) \quad\text{in}\quad \Omega_\varepsilon.
  \end{align*}
  From this inequality it follows that
  \begin{align*}
    \bigl\|\,|\nabla u^a|\,\varphi\bigr\|_{L^2(\Omega_\varepsilon)} \leq c\left(\left\|\,\left|\overline{Mu}\right|\varphi\right\|_{L^2(\Omega_\varepsilon)}+\left\|\,\left|\overline{\nabla_\Gamma Mu}\right|\varphi\right\|_{L^2(\Omega_\varepsilon)}\right).
  \end{align*}
  We apply \eqref{E:Prod_Surf} with $\eta=|Mu|$ and $\eta=|\nabla_\Gamma Mu|$ to the right-hand side and then use \eqref{E:Ave_Lp_Surf} and \eqref{E:Ave_Wmp_Surf} to obtain \eqref{E:Prod_Grad_Ua}.
\end{proof}

Next we derive a Poincar\'{e} type inequality for the residual part $u^r=u-u^a$.
By Lemma~\ref{L:ExImp_Bo}, the average part $u^a$ given by \eqref{E:Def_ExAve} satisfies the impermeable boundary condition \eqref{E:Bo_Imp}.
Hence $u^r$ satisfies \eqref{E:Bo_Imp} if $u$ itself satisfies the same condition.
This observation is essential for the proof of the Poincar\'{e} type inequality for $u^r$.

\begin{lemma} \label{L:Po_Ur}
  Let $u\in H^1(\Omega_\varepsilon)^3$ satisfy \eqref{E:Bo_Imp}.
  Then
  \begin{align} \label{E:Po_Ur}
    \|u^r\|_{L^2(\Omega_\varepsilon)} \leq c\varepsilon\|u^r\|_{H^1(\Omega_\varepsilon)}
  \end{align}
  for $u^r=u-u^a$, where $c>0$ is a constant independent of $\varepsilon$ and $u$.
\end{lemma}

\begin{proof}
  We use the notation \eqref{E:Def_U_TN} for the tangential and normal components of a vector field on $\Omega_\varepsilon$.
  By the definition \eqref{E:Def_ExAve} of $u^a$,
  \begin{align} \label{Pf_PUr:UrT}
    u_\tau^r = \overline{P}u^r = \overline{P}u-\overline{P}u^a = \overline{P}u-\overline{M_\tau u} = u_\tau-\overline{Mu_\tau} \quad\text{in}\quad \Omega_\varepsilon.
  \end{align}
  This equality and \eqref{E:NorDer_Con} imply that
  \begin{align*}
    \partial_nu_\tau^r = \overline{P}\partial_nu^r = \partial_nu_\tau, \quad |\partial_nu_\tau| = \left|\overline{P}\partial_nu^r\right| \leq |\partial_nu^r| \quad\text{in}\quad \Omega_\varepsilon.
  \end{align*}
  From these relations and \eqref{E:Ave_Diff_Dom} we deduce that
  \begin{align} \label{Pf_PUr:Est_Tau}
    \|u_\tau^r\|_{L^2(\Omega_\varepsilon)} = \left\|u_\tau-\overline{Mu_\tau}\right\|_{L^2(\Omega_\varepsilon)} \leq c\varepsilon\|\partial_nu_\tau\|_{L^2(\Omega_\varepsilon)} \leq c\varepsilon\|\partial_nu^r\|_{L^2(\Omega_\varepsilon)}.
  \end{align}
  Moreover, $u^r=u-u^a$ satisfies \eqref{E:Bo_Imp} by the assumption on $u$, since $u^a$ satisfies \eqref{E:Bo_Imp} by Lemma~\ref{L:ExImp_Bo} and \eqref{E:Def_ExAve}.
  Hence we can apply \eqref{E:Poin_Nor} to $u^r$ to get
  \begin{align*}
    \|u_n^r\|_{L^2(\Omega_\varepsilon)} = \|u^r\cdot\bar{n}\|_{L^2(\Omega_\varepsilon)} \leq c\varepsilon\|u^r\|_{H^1(\Omega_\varepsilon)}.
  \end{align*}
  By this inequality, \eqref{Pf_PUr:Est_Tau}, and
  \begin{align*}
    \|u^r\|_{L^2(\Omega_\varepsilon)}^2 = \|u_\tau^r\|_{L^2(\Omega_\varepsilon)}^2+\|u_n^r\|_{L^2(\Omega_\varepsilon)}^2, \quad \|\partial_nu^r\|_{L^2(\Omega_\varepsilon)}\leq c\|u^r\|_{H^1(\Omega_\varepsilon)}
  \end{align*}
  we conclude that \eqref{E:Po_Ur} is valid.
\end{proof}

We can also show a Poincar\'{e} type inequality for $\nabla u^r$ if $u$ satisfies the divergence-free condition in $\Omega_\varepsilon$ and the slip boundary conditions \eqref{E:Bo_Slip} on $\Gamma_\varepsilon$.

\begin{lemma} \label{L:Po_Grad_Ur}
  Suppose that the inequalities \eqref{E:Fric_Upper} are valid and $u\in H^2(\Omega_\varepsilon)^3$ satisfies $\mathrm{div}\,u=0$ in $\Omega_\varepsilon$ and \eqref{E:Bo_Slip}.
  Then
  \begin{align} \label{E:Po_Grad_Ur}
    \|\nabla u^r\|_{L^2(\Omega_\varepsilon)} \leq c\left(\varepsilon\|u\|_{H^2(\Omega_\varepsilon)}+\|u\|_{L^2(\Omega_\varepsilon)}\right)
  \end{align}
  for $u^r=u-u^a$, where $c>0$ is a constant independent of $\varepsilon$ and $u$.
\end{lemma}

\begin{proof}
  As in the proof of Lemma~\ref{L:Po_Ur} we use the notation \eqref{E:Def_U_TN}.
  Noting that
  \begin{align*}
    u^r = u_\tau^r+u_n^r, \quad I_3 = \overline{P}+\overline{Q} \quad\text{in}\quad \Omega_\varepsilon,
  \end{align*}
  we split the gradient matrix of $u^r$ into
  \begin{align} \label{Pf_PGUr:Sum}
    \nabla u^r=\overline{P}\nabla u_\tau^r+\overline{Q}\nabla u_\tau^r+\overline{P}\nabla u_n^r+\overline{Q}\nabla u_n^r \quad\text{in}\quad \Omega_\varepsilon
  \end{align}
  and deal with the four terms on the right-hand side separately.

  First we estimate the $L^2(\Omega_\varepsilon)$-norm of $\overline{P}\nabla u_\tau^r$.
  Since
  \begin{align*}
    \nabla u_\tau^r = \nabla\Bigl(u_\tau-\overline{Mu_\tau}\Bigr) = \nabla u_\tau-\Bigl(I_3-d\overline{W}\Bigr)^{-1}\,\overline{\nabla_\Gamma Mu_\tau} \quad\text{in}\quad \Omega_\varepsilon
  \end{align*}
  by \eqref{E:ConDer_Dom} and \eqref{Pf_PUr:UrT}, we observe by \eqref{E:P_TGr} and \eqref{E:WReso_P} that
  \begin{align*}
    \overline{P}\nabla u_\tau^r = \Bigl(\overline{P}\nabla u_\tau-\overline{\nabla_\Gamma Mu_\tau}\Bigr)-\left\{I_3-\Bigl(I_3-d\overline{W}\Bigr)^{-1}\right\}\overline{\nabla_\Gamma Mu_\tau} \quad\text{in}\quad \Omega_\varepsilon.
  \end{align*}
  Hence we apply \eqref{E:Wein_Diff} with $|d|\leq c\varepsilon$ in $\Omega_\varepsilon$, \eqref{E:Con_Lp}, \eqref{E:Ave_Wmp_Surf}, and \eqref{E:ADD_Dom} to get
  \begin{align} \label{Pf_PGUr:L2_PT}
    \begin{aligned}
      \left\|\overline{P}\nabla u_\tau^r\right\|_{L^2(\Omega_\varepsilon)} &\leq \left\|\overline{P}\nabla u_\tau-\overline{\nabla_\Gamma Mu_\tau}\right\|_{L^2(\Omega_\varepsilon)}+c\varepsilon\left\|\overline{\nabla_\Gamma Mu_\tau}\right\|_{L^2(\Omega_\varepsilon)} \\
      &\leq c\varepsilon\|u_\tau\|_{H^2(\Omega_\varepsilon)} \\
      &\leq c\varepsilon\|u\|_{H^2(\Omega_\varepsilon)}.
    \end{aligned}
  \end{align}
  Here the last inequality follows from $u_\tau=\overline{P}u$ in $\Omega_\varepsilon$ and $P\in C^4(\Gamma)^{3\times3}$.

  Next we deal with $\overline{Q}\nabla u_\tau^r$.
  Since $(\nabla u_\tau^r)^T\bar{n}=(\bar{n}\cdot\nabla)u_\tau^r=\partial_nu_\tau^r$ in $\Omega_\varepsilon$, we have
  \begin{align*}
    \overline{Q}\nabla u_\tau^r = \bar{n}\otimes[(\nabla u_\tau^r)^T\bar{n}] = \bar{n}\otimes\partial_nu_\tau^r, \quad \left|\overline{Q}\nabla u_\tau^r\right| = |\partial_nu_\tau^r| \quad\text{in}\quad \Omega_\varepsilon.
  \end{align*}
  By \eqref{E:NorDer_Con} and \eqref{Pf_PUr:UrT} we also get
  \begin{align*}
    \partial_nu_\tau^r = \partial_n\Bigl(\overline{P}u-\overline{M_\tau u}\Bigr) = \overline{P}\partial_nu \quad\text{in}\quad \Omega_\varepsilon.
  \end{align*}
  From these relations we deduce that
  \begin{align*}
    \left\|\overline{Q}\nabla u_\tau^r\right\|_{L^2(\Omega_\varepsilon)} = \left\|\overline{P}\partial_nu\right\|_{L^2(\Omega_\varepsilon)} &\leq \left\|\overline{P}\partial_nu+\overline{W}u\right\|_{L^2(\Omega_\varepsilon)}+\left\|\overline{W}u\right\|_{L^2(\Omega_\varepsilon)} \\
    &\leq \left\|\overline{P}\partial_nu+\overline{W}u\right\|_{L^2(\Omega_\varepsilon)}+c\|u\|_{L^2(\Omega_\varepsilon)}.
  \end{align*}
  Moreover, since we assume that the inequalities \eqref{E:Fric_Upper} are valid and $u$ satisfies \eqref{E:Bo_Slip}, we can apply \eqref{E:PDnU_WU} to the first term on the last line.
  Hence we obtain
  \begin{align} \label{Pf_PGUr:L2_QT}
    \left\|\overline{Q}\nabla u_\tau^r\right\|_{L^2(\Omega_\varepsilon)} \leq c\left(\varepsilon\|u\|_{H^2(\Omega_\varepsilon)}+\|u\|_{L^2(\Omega_\varepsilon)}\right).
  \end{align}
  Let us estimate the $L^2(\Omega_\varepsilon)$-norm of $\overline{P}\nabla u_n^r$.
  Since $u_n^r=(u^r\cdot\bar{n})\bar{n}$ in $\Omega_\varepsilon$, we have
  \begin{align*}
    \overline{P}\nabla u_n^r= \Bigl[\overline{P}\nabla(u^r\cdot\bar{n})\Bigr]\otimes\bar{n}-(u^r\cdot\bar{n})\overline{P}\nabla\bar{n} \quad\text{in}\quad \Omega_\varepsilon.
  \end{align*}
  By this formula, \eqref{E:NorG_Bound}, and $|n|=1$ and $|P|=2$ on $\Gamma$,
  \begin{align*}
    \left\|\overline{P}\nabla u_n^r\right\|_{L^2(\Omega_\varepsilon)} \leq c\left(\left\|\overline{P}\nabla(u^r\cdot\bar{n})\right\|_{L^2(\Omega_\varepsilon)}+\|u^r\cdot\bar{n}\|_{L^2(\Omega_\varepsilon)}\right).
  \end{align*}
  Here $u^r$ satisfies \eqref{E:Bo_Imp} by the assumption on $u$ since $u^a$ satisfies \eqref{E:Bo_Imp} by Lemma~\ref{L:ExImp_Bo} and \eqref{E:Def_ExAve}.
  Hence we apply \eqref{E:Poin_Nor}, \eqref{E:Poin_Dnor}, and \eqref{E:Wmp_UaUr} to the above inequality to get
  \begin{align} \label{Pf_PGUr:L2_PN}
    \left\|\overline{P}\nabla u_n^r\right\|_{L^2(\Omega_\varepsilon)} \leq c\varepsilon\|u^r\|_{H^2(\Omega_\varepsilon)} \leq c\varepsilon\|u\|_{H^2(\Omega_\varepsilon)}.
  \end{align}
  Now let us consider $\overline{Q}\nabla u_n^r=\bar{n}\otimes\partial_nu_n^r$.
  Since
  \begin{align*}
    u_n^r = (u^r\cdot\bar{n})\bar{n} = \Bigl(u\cdot\bar{n}-\overline{M_\tau u}\cdot\Psi_\varepsilon\Bigr)\bar{n}, \quad \partial_nu_n^r = \Bigl(\partial_nu\cdot\bar{n}-\overline{M_\tau u}\cdot\partial_n\Psi_\varepsilon\Bigr)\bar{n}
  \end{align*}
  in $\Omega_\varepsilon$ by \eqref{E:NorDer_Con} and \eqref{E:Def_ExAve}, we have
  \begin{align*}
    \left|\overline{Q}\nabla u_n^r\right| &= |\partial_nu_n^r| = \left|\partial_nu\cdot\bar{n}-\overline{M_\tau u}\cdot\partial_n\Psi_\varepsilon\right| \\
    &\leq \left|\partial_nu\cdot\bar{n}-\frac{1}{\bar{g}}\overline{M_\tau u}\cdot\overline{\nabla_\Gamma g}\right|+\left|\overline{M_\tau u}\right|\left|\partial_n\Psi_\varepsilon-\frac{1}{\bar{g}}\overline{\nabla_\Gamma g}\right|
  \end{align*}
  in $\Omega_\varepsilon$.
  By this inequality, $|M_\tau u|\leq |Mu|$ on $\Gamma$, \eqref{E:ExAux_TNDer}, \eqref{E:Ave_Lp_Dom}, and \eqref{E:DnU_N_Ave},
  \begin{align} \label{Pf_PGUr:L2_QN}
    \begin{aligned}
      \left\|\overline{Q}\nabla u_n^r\right\|_{L^2(\Omega_\varepsilon)} &\leq \left\|\partial_nu\cdot\bar{n}-\frac{1}{\bar{g}}\overline{M_\tau u}\cdot\overline{\nabla_\Gamma g}\right\|_{L^2(\Omega_\varepsilon)}+c\varepsilon\left\|\overline{Mu}\right\|_{L^2(\Omega_\varepsilon)} \\
      &\leq c\varepsilon\|u\|_{H^2(\Omega_\varepsilon)}.
    \end{aligned}
  \end{align}
  Note that $u$ satisfies $\mathrm{div}\,u=0$ in $\Omega_\varepsilon$ and \eqref{E:Bo_Imp} by the assumption on $u$ and thus we can apply \eqref{E:DnU_N_Ave}.
  Finally, applying \eqref{Pf_PGUr:L2_PT}--\eqref{Pf_PGUr:L2_QN} to \eqref{Pf_PGUr:Sum} we obtain \eqref{E:Po_Grad_Ur}.
\end{proof}

As a consequence of Lemmas~\ref{L:Po_Ur} and~\ref{L:Po_Grad_Ur}, we obtain a good $L^\infty(\Omega_\varepsilon)$-estimate for the residual part $u^r$.

\begin{lemma} \label{L:Linf_Ur}
  Suppose that the inequalities \eqref{E:Fric_Upper} are valid and $u\in H^2(\Omega_\varepsilon)^3$ satisfies $\mathrm{div}\,u=0$ in $\Omega_\varepsilon$ and \eqref{E:Bo_Slip}.
  Then
  \begin{align} \label{E:Linf_Ur}
    \|u^r\|_{L^\infty(\Omega_\varepsilon)} \leq c\left(\varepsilon^{1/2}\|u\|_{H^2(\Omega_\varepsilon)}+\|u\|_{L^2(\Omega_\varepsilon)}^{1/2}\|u\|_{H^2(\Omega_\varepsilon)}^{1/2}\right)
  \end{align}
  for $u^r=u-u^a$, where $c>0$ is a constant independent of $\varepsilon$ and $u$.
\end{lemma}

\begin{proof}
  Since $u^r\in H^2(\Omega_\varepsilon)^3$ by Lemma~\ref{L:Wmp_UaUr}, we have
  \begin{multline*}
    \|u^r\|_{L^\infty(\Omega_\varepsilon)} \leq c\varepsilon^{-1/2}\|u^r\|_{L^2(\Omega_\varepsilon)}^{1/4}\|u^r\|_{H^2(\Omega_\varepsilon)}^{1/2} \\
    \times\left(\|u^r\|_{L^2(\Omega_\varepsilon)}+\varepsilon\|\partial_nu^r\|_{L^2(\Omega_\varepsilon)}+\varepsilon^2\|\partial_n^2u^r\|_{L^2(\Omega_\varepsilon)}\right)^{1/4}
  \end{multline*}
  by the anisotropic Agmon inequality \eqref{E:Agmon}.
  Moreover, by \eqref{E:Po_Ur} and \eqref{E:Po_Grad_Ur} we have
  \begin{align*}
    \|u^r\|_{L^2(\Omega_\varepsilon)} &\leq c\varepsilon\|u^r\|_{H^1(\Omega_\varepsilon)} \leq c\varepsilon\left(\|u^r\|_{L^2(\Omega_\varepsilon)}+\|\nabla u^r\|_{L^2(\Omega_\varepsilon)}\right) \\
    &\leq c\varepsilon\left(\varepsilon\|u\|_{H^2(\Omega_\varepsilon)}+\|u\|_{L^2(\Omega_\varepsilon)}\right)
  \end{align*}
  and
  \begin{align*}
    \|\partial_nu^r\|_{L^2(\Omega_\varepsilon)} &\leq c\|\nabla u^r\|_{L^2(\Omega_\varepsilon)} \leq c\left(\varepsilon\|u\|_{H^2(\Omega_\varepsilon)}+\|u\|_{L^2(\Omega_\varepsilon)}\right).
  \end{align*}
  We also observe by \eqref{E:Wmp_UaUr} that
  \begin{align*}
    \|\partial_n^2u^r\|_{L^2(\Omega_\varepsilon)} \leq c\|u^r\|_{H^2(\Omega_\varepsilon)} \leq c\|u\|_{H^2(\Omega_\varepsilon)}.
  \end{align*}
  From the above inequalities we deduce that
  \begin{align*}
    \|u^r\|_{L^\infty(\Omega_\varepsilon)} \leq c\left(\varepsilon\|u\|_{H^2(\Omega_\varepsilon)}+\|u\|_{L^2(\Omega_\varepsilon)}\right)^{1/2}\|u\|_{H^2(\Omega_\varepsilon)}^{1/2}.
  \end{align*}
  Using $(a+b)^{1/2}\leq a^{1/2}+b^{1/2}$ for $a,b\geq0$ to this inequality we obtain \eqref{E:Linf_Ur}.
\end{proof}

\section{Estimate for the trilinear term} \label{S:Tri}
The purpose of this section is to derive a good estimate for the trilinear term
\begin{align*}
  \bigl((u\cdot\nabla)u,A_\varepsilon u\bigr)_{L^2(\Omega_\varepsilon)}, \quad u\in D(A_\varepsilon),
\end{align*}
which is essential for the proof of the global existence of a strong solution to \eqref{E:NS_CTD}.
Throughout this section we impose Assumptions~\ref{Assump_1} and~\ref{Assump_2} and fix the constant $\varepsilon_0$ given in Lemma~\ref{L:Uni_aeps}.
For $\varepsilon\in(0,\varepsilon_0]$ let $\mathcal{H}_\varepsilon$ be the subspace of $L^2(\Omega_\varepsilon)^3$ given by \eqref{E:Def_Heps} and $A_\varepsilon$ the Stokes operator on $\mathcal{H}_\varepsilon$ introduced in Section~\ref{S:St_Op}.

\begin{lemma} \label{L:Tri_Est}
  For any $\alpha>0$ there exist constants $c_\alpha^1,c_\alpha^2>0$ such that
  \begin{multline} \label{E:Tri_Est}
    \left|\bigl((u\cdot\nabla)u,A_\varepsilon u\bigr)_{L^2(\Omega_\varepsilon)}\right| \leq \left(\alpha+c_\alpha^1\varepsilon^{1/2}\|u\|_{H^1(\Omega_\varepsilon)}\right)\|u\|_{H^2(\Omega_\varepsilon)}^2 \\
    +c_\alpha^2\left(\|u\|_{L^2(\Omega_\varepsilon)}^2\|u\|_{H^1(\Omega_\varepsilon)}^4+\varepsilon^{-1}\|u\|_{L^2(\Omega_\varepsilon)}^2\|u\|_{H^1(\Omega_\varepsilon)}^2\right)
  \end{multline}
  for all $\varepsilon\in(0,\varepsilon_0]$ and $u\in D(A_\varepsilon)$.
  (In fact, $c_\alpha^1$ does not depend on $\alpha$.)
\end{lemma}

To prove Lemma~\ref{L:Tri_Est} we give three auxiliary lemmas.
As in the previous sections, we denote by $\bar{\eta}=\eta\circ\pi$ the constant extension of a function $\eta$ on $\Gamma$.
Let $n_\varepsilon^0$ and $n_\varepsilon^1$ be the vector fields on $\Gamma$ given by \eqref{E:Def_NB} and
\begin{align*}
  W_\varepsilon^i(x) := -\{I_3-\bar{n}_\varepsilon^i(x)\otimes\bar{n}_\varepsilon^i(x)\}\nabla\bar{n}_\varepsilon^i(x), \quad x\in N,\,i=0,1.
\end{align*}
Also, for $x\in N$ let
\begin{align*}
  \begin{aligned}
    \tilde{n}_1(x) &:= \frac{1}{\varepsilon\bar{g}(x)}\bigl\{\bigl(d(x)-\varepsilon\bar{g}_0(x)\bigr)\bar{n}_\varepsilon^1(x)-\bigl(\varepsilon\bar{g}_1(x)-d(x)\bigr)\bar{n}_\varepsilon^0(x)\bigr\}, \\
    \tilde{n}_2(x) &:= \frac{1}{\varepsilon\bar{g}(x)}\left\{\bigl(d(x)-\varepsilon\bar{g}_0(x)\bigr)\frac{\gamma_\varepsilon^1}{\nu}\bar{n}_\varepsilon^1(x)+\bigl(\varepsilon\bar{g}_1(x)-d(x)\bigr)\frac{\gamma_\varepsilon^0}{\nu}\bar{n}_\varepsilon^0(x)\right\}, \\
    \widetilde{W}(x) &:= \frac{1}{\varepsilon\bar{g}(x)}\bigl\{\bigl(d(x)-\varepsilon\bar{g}_0(x)\bigr)W_\varepsilon^1(x)-\bigl(\varepsilon\bar{g}_1(x)-d(x)\bigr)W_\varepsilon^0(x)\bigr\}.
  \end{aligned}
\end{align*}
For a vector field $u$ on $\Omega_\varepsilon$ we define
\begin{align} \label{E:Def_Gu}
  G(u)(x) := 2\tilde{n}_1(x)\times\Bigl[\widetilde{W}(x)u(x)\Bigr]+\tilde{n}_2(x)\times u(x), \quad x\in\Omega_\varepsilon.
\end{align}

\begin{lemma}[{\cite{Miu_NSCTD_01}*{Lemma~7.2}}] \label{L:G_Bound}
  Suppose that the inequalities \eqref{E:Fric_Upper} are valid.
  Then there exists a constant $c>0$ independent of $\varepsilon$ such that
  \begin{align} \label{E:G_Bound}
    |G(u)| \leq c|u|, \quad |\nabla G(u)| \leq c(|u|+|\nabla u|) \quad\text{in}\quad \Omega_\varepsilon
  \end{align}
  for all $u\in C^1(\Omega_\varepsilon)^3$, where $G(u)$ is the vector field on $\Omega_\varepsilon$ given by \eqref{E:Def_Gu}.
\end{lemma}

\begin{lemma}[{\cite{Miu_NSCTD_01}*{Lemma~7.3}}] \label{L:IbP_Curl}
  The integration by parts formula
  \begin{multline} \label{E:IbP_Curl}
    \int_{\Omega_\varepsilon}\mathrm{curl}\,\mathrm{curl}\,u\cdot\Phi\,dx \\
    = -\int_{\Omega_\varepsilon}\mathrm{curl}\,G(u)\cdot\Phi\,dx+\int_{\Omega_\varepsilon}\{\mathrm{curl}\,u+G(u)\}\cdot\mathrm{curl}\,\Phi\,dx
  \end{multline}
  holds for all $u\in H^2(\Omega_\varepsilon)^3$ satisfying \eqref{E:Bo_Slip} and $\Phi\in L^2(\Omega_\varepsilon)^3$ with $\mathrm{curl}\,\Phi\in L^2(\Omega_\varepsilon)^3$, where $G(u)$ is the vector field on $\Omega_\varepsilon$ given by \eqref{E:Def_Gu}.
\end{lemma}

For the proofs of Lemmas~\ref{L:G_Bound} and~\ref{L:IbP_Curl} we refer to our first paper~\cite{Miu_NSCTD_01}.

We also present a useful estimate for the curl of the average part $u^a$.

\begin{lemma} \label{L:Tan_Curl_Ua}
  For $u\in C^1(\Omega_\varepsilon)^3$ let $u^a$ be given by \eqref{E:Def_ExAve}.
  Then
  \begin{align} \label{E:Tan_Curl_Ua}
    \left|\overline{P}\,\mathrm{curl}\,u^a\right| \leq c\left(\left|\overline{Mu}\right|+\varepsilon\left|\overline{\nabla_\Gamma Mu}\right|\right) \quad\text{in}\quad \Omega_\varepsilon,
  \end{align}
  where $c>0$ is a constant independent of $\varepsilon$ and $u$.
\end{lemma}

The proof of Lemma~\ref{L:Tan_Curl_Ua} is given in Appendix~\ref{S:Ap_Proofs}.

Now let us prove Lemma~\ref{L:Tri_Est}.
The main tools for the proof are the estimates for the Stokes and average operators given in Sections~\ref{S:St_Op} and~\ref{S:Ave}.

\begin{proof}[Proof of Lemma~\ref{L:Tri_Est}]
  The proof is basically the same as that of \cite{Ho10}*{Proposition~6.1}, but we require further calculations.

  Let $u\in D(A_\varepsilon)$.
  First note that $u\in H^2(\Omega_\varepsilon)^3$ and it satisfies $\mathrm{div}\,u=0$ in $\Omega_\varepsilon$ and the slip boundary conditions \eqref{E:Bo_Slip} by \eqref{E:Dom_St}, and thus we can apply all the lemmas in the previous sections to $u$.
  Let $u^a$ be the average part of $u$ given by \eqref{E:Def_ExAve}, $u^r:=u-u^a$ the residual part, and $\omega:=\mathrm{curl}\,u$.
  Since
  \begin{align*}
    (u\cdot\nabla)u = \omega\times u+\frac{1}{2}\nabla(|u|^2), \quad A_\varepsilon u \in \mathcal{H}_\varepsilon \subset L_\sigma^2(\Omega_\varepsilon), \quad \nabla(|u|^2) \in L_\sigma^2(\Omega_\varepsilon)^\perp,
  \end{align*}
  we have $(\nabla(|u|^2),A_\varepsilon u)_{L^2(\Omega_\varepsilon)}=0$ and thus
  \begin{align*}
    \bigl((u\cdot\nabla)u,A_\varepsilon u\bigr)_{L^2(\Omega_\varepsilon)} = (\omega\times u,A_\varepsilon u)_{L^2(\Omega_\varepsilon)} = J_1+J_2+J_3,
  \end{align*}
  where
  \begin{align*}
    J_1 &:= (\omega\times u^r,A_\varepsilon u)_{L^2(\Omega_\varepsilon)}, \\
    J_2 &:= (\omega\times u^a,A_\varepsilon u+\nu\Delta u)_{L^2(\Omega_\varepsilon)}, \\
    J_3 &:= (\omega\times u^a,-\nu\Delta u)_{L^2(\Omega_\varepsilon)}.
  \end{align*}
  Let us estimate $J_1$, $J_2$, and $J_3$ separately.
  By \eqref{E:Stokes_H2} and \eqref{E:Linf_Ur},
  \begin{align*}
    |J_1| &\leq \|u^r\|_{L^\infty(\Omega_\varepsilon)}\|\omega\|_{L^2(\Omega_\varepsilon)}\|A_\varepsilon u\|_{L^2(\Omega_\varepsilon)} \\
    &\leq c\left(\varepsilon^{1/2}\|u\|_{H^2(\Omega_\varepsilon)}+\|u\|_{L^2(\Omega_\varepsilon)}^{1/2}\|u\|_{H^2(\Omega_\varepsilon)}^{1/2}\right)\|u\|_{H^1(\Omega_\varepsilon)}\|u\|_{H^2(\Omega_\varepsilon)} \\
    &= c\varepsilon^{1/2}\|u\|_{H^1(\Omega_\varepsilon)}\|u\|_{H^2(\Omega_\varepsilon)}^2+c\|u\|_{L^2(\Omega_\varepsilon)}^{1/2}\|u\|_{H^1(\Omega_\varepsilon)}\|u\|_{H^2(\Omega_\varepsilon)}^{3/2}.
  \end{align*}
  To the last term we apply Young's inequality $ab\leq \alpha a^{4/3}+c_\alpha b^4$ to get
  \begin{align} \label{Pf_TE:Est_I1}
    |J_1| \leq \left(\alpha+c\varepsilon^{1/2}\|u\|_{H^1(\Omega_\varepsilon)}\right)\|u\|_{H^2(\Omega_\varepsilon)}^2+c_\alpha\|u\|_{L^2(\Omega_\varepsilon)}^2\|u\|_{H^1(\Omega_\varepsilon)}^4.
  \end{align}
  Next we deal with $J_2$.
  By \eqref{E:Prod_Ua} we have
  \begin{align} \label{Pf_TE:L2_Phi}
    \begin{aligned}
      \|\omega\times u^a\|_{L^2(\Omega_\varepsilon)} &\leq c\varepsilon^{-1/2}\|\omega\|_{L^2(\Omega_\varepsilon)}^{1/2}\|\omega\|_{H^1(\Omega_\varepsilon)}^{1/2}\|u\|_{L^2(\Omega_\varepsilon)}^{1/2}\|u\|_{H^1(\Omega_\varepsilon)}^{1/2} \\
      &\leq c\varepsilon^{-1/2}\|u\|_{L^2(\Omega_\varepsilon)}^{1/2}\|u\|_{H^1(\Omega_\varepsilon)}\|u\|_{H^2(\Omega_\varepsilon)}^{1/2}.
    \end{aligned}
  \end{align}
  From this inequality, \eqref{E:Comp_Sto_Lap}, and \eqref{E:St_Inter} it follows that
  \begin{align*}
    |J_2| &\leq \|\omega\times u^a\|_{L^2(\Omega_\varepsilon)}\|A_\varepsilon u+\nu\Delta u\|_{L^2(\Omega_\varepsilon)} \\
    &\leq c\varepsilon^{-1/2}\|u\|_{L^2(\Omega_\varepsilon)}^{1/2}\|u\|_{H^1(\Omega_\varepsilon)}^2\|u\|_{H^2(\Omega_\varepsilon)}^{1/2} \\
    &\leq c\varepsilon^{-1/2}\|u\|_{L^2(\Omega_\varepsilon)}\|u\|_{H^1(\Omega_\varepsilon)}\|u\|_{H^2(\Omega_\varepsilon)}.
  \end{align*}
   Applying Young's inequality $ab\leq \alpha a^2+c_\alpha b^2$ to the last line we further get
  \begin{align} \label{Pf_TE:Est_I2}
      |J_2| \leq \alpha\|u\|_{H^2(\Omega_\varepsilon)}^2+c_\alpha\varepsilon^{-1}\|u\|_{L^2(\Omega_\varepsilon)}^2\|u\|_{H^1(\Omega_\varepsilon)}^2.
  \end{align}
  The estimate for $J_3$ is more complicated.
  Let $\Phi:=\omega\times u^a$.
  Since $\omega\in H^1(\Omega_\varepsilon)^3$ and $u^a\in H^2(\Omega_\varepsilon)^3$, we have $\Phi\in H^1(\Omega_\varepsilon)^3$ by the Sobolev embeddings (see~\cite{AdFo03})
  \begin{align*}
    H^1(\Omega_\varepsilon)\hookrightarrow L^4(\Omega_\varepsilon), \quad H^2(\Omega_\varepsilon)\hookrightarrow L^\infty(\Omega_\varepsilon).
  \end{align*}
  Also, since $-\Delta u=\mathrm{curl}\,\mathrm{curl}\,u=\mathrm{curl}\,\omega$ by $\mathrm{div}\,u=0$ in $\Omega_\varepsilon$,
  \begin{align*}
    J_3 &= -\nu(\Delta u,\Phi)_{L^2(\Omega_\varepsilon)} = \nu(\mathrm{curl}\,\omega,\Phi)_{L^2(\Omega_\varepsilon)} \\
    &= -\nu(\mathrm{curl}\,G(u),\Phi)_{L^2(\Omega_\varepsilon)}+\nu(\omega+G(u),\mathrm{curl}\,\Phi)_{L^2(\Omega_\varepsilon)} = J_3^1+J_3^2+J_3^3
  \end{align*}
  by \eqref{E:IbP_Curl}.
  Here $G(u)$ is given by \eqref{E:Def_Gu} and
  \begin{align*}
    J_3^1 &:= -\nu(\mathrm{curl}\,G(u),\Phi)_{L^2(\Omega_\varepsilon)}, \\
    J_3^2 &:= \nu(G(u),\mathrm{curl}\,\Phi)_{L^2(\Omega_\varepsilon)}, \\
    J_3^3 &:= \nu(\omega,\mathrm{curl}\,\Phi)_{L^2(\Omega_\varepsilon)}.
  \end{align*}
  Noting that $\Phi=\omega\times u^a$, we apply \eqref{E:G_Bound} and \eqref{Pf_TE:L2_Phi} to $J_3^1$ to deduce that
  \begin{align*}
    |J_3^1| \leq c\|\nabla G(u)\|_{L^2(\Omega_\varepsilon)}\|\Phi\|_{L^2(\Omega_\varepsilon)} \leq c\varepsilon^{-1/2}\|u\|_{L^2(\Omega_\varepsilon)}^{1/2}\|u\|_{H^1(\Omega_\varepsilon)}^2\|u\|_{H^2(\Omega_\varepsilon)}^{1/2}.
  \end{align*}
  Then using \eqref{E:St_Inter} and Young's inequality $ab\leq \alpha a^2+c_\alpha b^2$ we get
  \begin{align} \label{Pf_TE:Est_J1}
    \begin{aligned}
    |J_3^1| &\leq c\varepsilon^{-1/2}\|u\|_{L^2(\Omega_\varepsilon)}\|u\|_{H^1(\Omega_\varepsilon)}\|u\|_{H^2(\Omega_\varepsilon)} \\
    &\leq \alpha\|u\|_{H^2(\Omega_\varepsilon)}^2+c_\alpha\varepsilon^{-1}\|u\|_{L^2(\Omega_\varepsilon)}^2\|u\|_{H^1(\Omega_\varepsilon)}^2.
    \end{aligned}
  \end{align}
  Let us estimate $J_3^2$.
  The curl of $\Phi=\omega\times u^a$ is bounded by
  \begin{align*}
    |\mathrm{curl}\,\Phi| \leq c(|\nabla\omega||u^a|+|\omega||\nabla u^a|) \leq c(|u^a||\nabla^2u|+|\nabla u^a||\nabla u|) \quad\text{in}\quad \Omega_\varepsilon.
  \end{align*}
  By this inequality, \eqref{E:G_Bound}, and H\"{o}lder's inequality we get
  \begin{align*}
    |J_3^2| &\leq c\int_{\Omega_\varepsilon}|u|(|u^a||\nabla^2u|+|\nabla u^a||\nabla u|)\,dx \\
    &\leq c\left(\|\,|u^a|\,|u|\,\|_{L^2(\Omega_\varepsilon)}\|\nabla^2u\|_{L^2(\Omega_\varepsilon)}+\|\,|\nabla u^a|\,|u|\,\|_{L^2(\Omega_\varepsilon)}\|\nabla u\|_{L^2(\Omega_\varepsilon)}\right).
  \end{align*}
  To the last line we apply \eqref{E:Prod_Ua} and \eqref{E:Prod_Grad_Ua} to obtain
  \begin{align*}
    |J_3^2| &\leq c\varepsilon^{-1/2}\left(\|u\|_{L^2(\Omega_\varepsilon)}\|u\|_{H^1(\Omega_\varepsilon)}\|u\|_{H^2(\Omega_\varepsilon)}+\|u\|_{L^2(\Omega_\varepsilon)}^{1/2}\|u\|_{H^1(\Omega_\varepsilon)}^2\|u\|_{H^2(\Omega_\varepsilon)}^{1/2}\right) \\
    &\leq c\varepsilon^{-1/2}\|u\|_{L^2(\Omega_\varepsilon)}\|u\|_{H^1(\Omega_\varepsilon)}\|u\|_{H^2(\Omega_\varepsilon)},
  \end{align*}
  where the second inequality follows from \eqref{E:St_Inter}.
  Hence Young's inequality yields
  \begin{align} \label{Pf_TE:Est_J2}
    |J_3^2| \leq \alpha\|u\|_{H^2(\Omega_\varepsilon)}^2+c_\alpha\varepsilon^{-1}\|u\|_{L^2(\Omega_\varepsilon)}^2\|u\|_{H^1(\Omega_\varepsilon)}^2.
  \end{align}
  To estimate $J_3^3=\nu(\omega,\mathrm{curl}\,\Phi)_{L^2(\Omega_\varepsilon)}$ we observe by $\Phi=\omega\times u^a$ that
  \begin{align*}
    \mathrm{curl}\,\Phi = (u^a\cdot\nabla)\omega-(\omega\cdot\nabla)u^a+(\mathrm{div}\,u^a)\omega-(\mathrm{div}\,\omega)u^a \quad\text{in}\quad \Omega_\varepsilon.
  \end{align*}
  Moreover, since $u^a$ satisfies $u^a\cdot n_\varepsilon=0$ on $\Gamma_\varepsilon$ by Lemma~\ref{L:ExImp_Bo} and \eqref{E:Def_ExAve}, we have
  \begin{align*}
    \int_{\Omega_\varepsilon}\omega\cdot(u^a\cdot\nabla)\omega\,dx = -\frac{1}{2}\int_{\Omega_\varepsilon}(\mathrm{div}\,u^a)|\omega|^2\,dx
  \end{align*}
  by integration by parts.
  By these equalities and $\mathrm{div}\,\omega=\mathrm{div}\,\mathrm{curl}\,u=0$ in $\Omega_\varepsilon$ we get
  \begin{align} \label{Pf_TE:J3_Exp}
    \begin{aligned}
      J_3^3 & = \nu(\omega,(u^a\cdot\nabla)\omega-(\omega\cdot\nabla)u^a+(\mathrm{div}\,u^a)\omega)_{L^2(\Omega_\varepsilon)} \\
      & = \frac{\nu}{2}(\mathrm{div}\,u^a,|\omega|^2)_{L^2(\Omega_\varepsilon)}-\nu(\omega,(\omega\cdot\nabla)u^a)_{L^2(\Omega_\varepsilon)}.
    \end{aligned}
  \end{align}
  Noting that $u^a=E_\varepsilon M_\tau u$ by the definition \eqref{E:Def_ExAve}, we write
  \begin{multline*}
    (\mathrm{div}\,u^a,|\omega|^2)_{L^2(\Omega_\varepsilon)} = \int_{\Omega_\varepsilon}\frac{1}{\bar{g}}\left(\overline{\mathrm{div}_\Gamma(gM_\tau u)}\right)|\omega|^2\,dx \\
    +\int_{\Omega_\varepsilon}\left(\mathrm{div}(E_\varepsilon M_\tau u)-\frac{1}{\bar{g}}\overline{\mathrm{div}_\Gamma(gM_\tau u)}\right)|\omega|^2\,dx
  \end{multline*}
  and apply \eqref{E:G_Inf}, \eqref{E:ExImp_Div}, and H\"{o}lder's inequality to the right-hand side to get
  \begin{align*}
    |(\mathrm{div}\,u^a,|\omega|^2)_{L^2(\Omega_\varepsilon)}| \leq c(K_1+\varepsilon K_2+\varepsilon K_3)\|\omega\|_{L^2(\Omega_\varepsilon)},
  \end{align*}
  where
  \begin{align*}
    K_1 &:= \left\|\,\left|\overline{\mathrm{div}_\Gamma(gM_\tau u)}\right|\,|\omega|\,\right\|_{L^2(\Omega_\varepsilon)}, \\
    K_2 &:= \left\|\,\left|\overline{M_\tau u}\right|\,|\omega|\,\right\|_{L^2(\Omega_\varepsilon)}, \\
    K_3 &:= \left\|\,\left|\overline{\nabla_\Gamma M_\tau u}\right|\,|\omega|\,\right\|_{L^2(\Omega_\varepsilon)}.
  \end{align*}
  To $K_1$ we apply \eqref{E:Prod_Surf} and use \eqref{E:ADiv_Tan_Lp} and \eqref{E:ADiv_Tan_W1p}.
  Then we have
  \begin{align*}
    K_1 &\leq c\|\mathrm{div}_\Gamma(gM_\tau u)\|_{L^2(\Gamma)}^{1/2}\|\mathrm{div}_\Gamma(gM_\tau u)\|_{H^1(\Gamma)}^{1/2}\|\omega\|_{L^2(\Omega_\varepsilon)}^{1/2}\|\omega\|_{H^1(\Omega_\varepsilon)}^{1/2} \\
    &\leq c\varepsilon^{1/2}\|u\|_{H^1(\Omega_\varepsilon)}\|u\|_{H^2(\Omega_\varepsilon)} \leq c\varepsilon^{1/2}\|u\|_{H^2(\Omega_\varepsilon)}^2.
  \end{align*}
  Also, by $M_\tau u=PMu$ on $\Gamma$, $P\in C^4(\Gamma)^{3\times3}$, \eqref{E:Ave_Lp_Surf}, and \eqref{E:Ave_Wmp_Surf} we see that
  \begin{align*}
    \|M_\tau u\|_{H^k(\Gamma)} &\leq c\|Mu\|_{H^k(\Gamma)} \leq c\varepsilon^{-1/2}\|u\|_{H^k(\Omega_\varepsilon)}, \\
    \|\nabla_\Gamma M_\tau u\|_{H^k(\Gamma)} &\leq c\|Mu\|_{H^{k+1}(\Gamma)} \leq c\varepsilon^{-1/2}\|u\|_{H^{k+1}(\Omega_\varepsilon)}
  \end{align*}
  for $k=0,1$ (with $H^0=L^2$).
  Using \eqref{E:Prod_Surf} and these inequalities we obtain
  \begin{align*}
    K_2 &\leq c\varepsilon^{-1/2}\|u\|_{L^2(\Omega_\varepsilon)}^{1/2}\|u\|_{H^1(\Omega_\varepsilon)}\|u\|_{H^2(\Omega_\varepsilon)}^{1/2} \leq c\varepsilon^{-1/2}\|u\|_{H^2(\Omega_\varepsilon)}^2, \\
    K_3 &\leq c\varepsilon^{-1/2}\|u\|_{H^1(\Omega_\varepsilon)}\|u\|_{H^2(\Omega_\varepsilon)} \leq c\varepsilon^{-1/2}\|u\|_{H^2(\Omega_\varepsilon)}^2.
  \end{align*}
  From these inequalities and $\|\omega\|_{L^2(\Omega_\varepsilon)}\leq c\|u\|_{H^1(\Omega_\varepsilon)}$ we deduce that
  \begin{align} \label{Pf_TE:DivUa_O}
    \begin{aligned}
      |(\mathrm{div}\,u^a,|\omega|^2)_{L^2(\Omega_\varepsilon)}| &\leq c(K_1+\varepsilon K_2+\varepsilon K_3)\|\omega\|_{L^2(\Omega_\varepsilon)} \\
      &\leq c\varepsilon^{1/2}\|u\|_{H^1(\Omega_\varepsilon)}\|u\|_{H^2(\Omega_\varepsilon)}^2.
    \end{aligned}
  \end{align}
  Let us estimate $(\omega,(\omega\cdot\nabla)u^a)_{L^2(\Omega_\varepsilon)}$.
  By $\omega=\mathrm{curl}\,u^r+\mathrm{curl}\,u^a$ we have
  \begin{align*}
    (\omega,(\omega\cdot\nabla)u^a)_{L^2(\Omega_\varepsilon)} = (\omega,(\mathrm{curl}\,u^r\cdot\nabla)u^a)_{L^2(\Omega_\varepsilon)}+(\omega,(\mathrm{curl}\,u^a\cdot\nabla)u^a)_{L^2(\Omega_\varepsilon)}.
  \end{align*}
  The first term on the right-hand side is bounded by
  \begin{align*}
    |(\omega,(\mathrm{curl}\,u^r\cdot\nabla)u^a)_{L^2(\Omega_\varepsilon)}| \leq c\|\nabla u^r\|_{L^2(\Omega_\varepsilon)}\|\,|\nabla u^a|\,|\omega|\,\|_{L^2(\Omega_\varepsilon)}.
  \end{align*}
  To the right-hand side we apply \eqref{E:Po_Grad_Ur} and
  \begin{align} \label{Pf_TE:GUa_O}
    \begin{aligned}
      \|\,|\nabla u^a|\,|\omega|\,\|_{L^2(\Omega_\varepsilon)} &\leq c\varepsilon^{-1/2}\|\omega\|_{L^2(\Omega_\varepsilon)}^{1/2}\|\omega\|_{H^1(\Omega_\varepsilon)}^{1/2}\|u\|_{H^1(\Omega_\varepsilon)}^{1/2}\|u\|_{H^2(\Omega_\varepsilon)}^{1/2} \\
      &\leq c\varepsilon^{-1/2}\|u\|_{H^1(\Omega_\varepsilon)}\|u\|_{H^2(\Omega_\varepsilon)}
    \end{aligned}
  \end{align}
  by \eqref{E:Prod_Grad_Ua}.
  Then we get
  \begin{multline} \label{Pf_TE:OUrUa}
    |(\omega,(\mathrm{curl}\,u^r\cdot\nabla)u^a)_{L^2(\Omega_\varepsilon)}| \\
    \leq c\left(\varepsilon^{1/2}\|u\|_{H^1(\Omega_\varepsilon)}\|u\|_{H^2(\Omega_\varepsilon)}^2+\varepsilon^{-1/2}\|u\|_{L^2(\Omega_\varepsilon)}\|u\|_{H^1(\Omega_\varepsilon)}\|u\|_{H^2(\Omega_\varepsilon)}\right).
  \end{multline}
  Also, noting that
  \begin{align*}
    \mathrm{curl}\,u^a = \overline{P}\,\mathrm{curl}\,u^a+(\mathrm{curl}\,u^a\cdot\bar{n})\bar{n}, \quad (\bar{n}\cdot\nabla)u^a = \partial_nu^a \quad\text{in}\quad \Omega_\varepsilon
  \end{align*}
  we decompose $(\omega,(\mathrm{curl}\,u^a\cdot\nabla)u^a)_{L^2(\Omega_\varepsilon)}$ into the sum of
  \begin{align*}
    L_1 := \left(\omega,\bigl((\overline{P}\,\mathrm{curl}\,u^a)\cdot\nabla\bigr)u^a\right)_{L^2(\Omega_\varepsilon)}, \quad L_2 := (\omega,(\mathrm{curl}\,u^a\cdot\bar{n})\partial_nu^a)_{L^2(\Omega_\varepsilon)}.
  \end{align*}
  To $L_1$ we apply \eqref{E:Tan_Curl_Ua} and H\"{o}lder's inequality to get
  \begin{align*}
    |L_1| &\leq c\int_{\Omega_\varepsilon}|\omega|\left(\left|\overline{Mu}\right|+\varepsilon\left|\overline{\nabla_\Gamma Mu}\right|\right)|\nabla u^a|\,dx \\
    &\leq c\|\,|\nabla u^a|\,|\omega|\,\|_{L^2(\Omega_\varepsilon)}\left(\left\|\overline{Mu}\right\|_{L^2(\Omega_\varepsilon)}+\varepsilon\left\|\overline{\nabla_\Gamma Mu}\right\|_{L^2(\Omega_\varepsilon)}\right).
  \end{align*}
  Hence from \eqref{E:Con_Lp}, \eqref{E:Ave_Lp_Surf}, \eqref{E:Ave_Wmp_Surf}, \eqref{Pf_TE:GUa_O}, and $\|u\|_{H^1(\Omega_\varepsilon)}\leq\|u\|_{H^2(\Omega_\varepsilon)}$ it follows that
  \begin{align*}
    |L_1| &\leq c\varepsilon^{-1/2}\|u\|_{H^1(\Omega_\varepsilon)}\|u\|_{H^2(\Omega_\varepsilon)}\left(\|u\|_{L^2(\Omega_\varepsilon)}+\varepsilon\|u\|_{H^1(\Omega_\varepsilon)}\right) \\
    &\leq c\left(\varepsilon^{1/2}\|u\|_{H^1(\Omega_\varepsilon)}\|u\|_{H^2(\Omega_\varepsilon)}^2+\varepsilon^{-1/2}\|u\|_{L^2(\Omega_\varepsilon)}\|u\|_{H^1(\Omega_\varepsilon)}\|u\|_{H^2(\Omega_\varepsilon)}\right).
  \end{align*}
  To estimate $L_2$ we see by the definition \eqref{E:Def_ExAve} of $u^a$, \eqref{E:NorDer_Con}, and \eqref{E:ExAux_Bound} that
  \begin{align*}
    |\partial_nu^a| = \left|\overline{M_\tau u}\cdot\partial_n\Psi_\varepsilon\right| \leq c\left|\overline{M_\tau u}\right| = c\left|\overline{PMu}\right| \leq c\left|\overline{Mu}\right| \quad\text{in}\quad \Omega_\varepsilon.
  \end{align*}
  By this inequality, $|\mathrm{curl}\,u^a\cdot\bar{n}|\leq c|\nabla u^a|$ in $\Omega_\varepsilon$, \eqref{E:Ave_Lp_Dom}, and \eqref{Pf_TE:GUa_O},
  \begin{align*}
    |L_2| &\leq c\int_{\Omega_\varepsilon}|\omega||\nabla u^a|\left|\overline{Mu}\right|\,dx \leq c\|\,|\nabla u^a|\,|\omega|\,\|_{L^2(\Omega_\varepsilon)}\left\|\overline{Mu}\right\|_{L^2(\Omega_\varepsilon)} \\
    &\leq c\varepsilon^{-1/2}\|u\|_{L^2(\Omega_\varepsilon)}\|u\|_{H^1(\Omega_\varepsilon)}\|u\|_{H^2(\Omega_\varepsilon)}.
  \end{align*}
  Applying the above estimates to $(\omega,(\mathrm{curl}\,u^a\cdot\nabla)u^a)_{L^2(\Omega_\varepsilon)}=L_1+L_2$ we obtain
  \begin{multline*}
    |(\omega,(\mathrm{curl}\,u^a\cdot\nabla)u^a)_{L^2(\Omega_\varepsilon)}| \\
    \leq c\left(\varepsilon^{1/2}\|u\|_{H^1(\Omega_\varepsilon)}\|u\|_{H^2(\Omega_\varepsilon)}^2+\varepsilon^{-1/2}\|u\|_{L^2(\Omega_\varepsilon)}\|u\|_{H^1(\Omega_\varepsilon)}\|u\|_{H^2(\Omega_\varepsilon)}\right).
  \end{multline*}
  From this inequality and \eqref{Pf_TE:OUrUa} we deduce that
  \begin{align*}
    &|(\omega,(\omega\cdot\nabla)u^a)_{L^2(\Omega_\varepsilon)}| \\
    &\qquad \leq |(\omega,(\mathrm{curl}\,u^r\cdot\nabla)u^a)_{L^2(\Omega_\varepsilon)}|+|(\omega,(\mathrm{curl}\,u^a\cdot\nabla)u^a)_{L^2(\Omega_\varepsilon)}| \\
    &\qquad \leq c\left(\varepsilon^{1/2}\|u\|_{H^1(\Omega_\varepsilon)}\|u\|_{H^2(\Omega_\varepsilon)}^2+\varepsilon^{-1/2}\|u\|_{L^2(\Omega_\varepsilon)}\|u\|_{H^1(\Omega_\varepsilon)}\|u\|_{H^2(\Omega_\varepsilon)}\right).
  \end{align*}
  Using Young's inequality $ab\leq \alpha a^2+c_\alpha b^2$ to the last term we further get
  \begin{multline*}
    |(\omega,(\omega\cdot\nabla)u^a)_{L^2(\Omega_\varepsilon)}| \\
    \leq \left(\alpha+c\varepsilon^{1/2}\|u\|_{H^1(\Omega_\varepsilon)}\right)\|u\|_{H^2(\Omega_\varepsilon)}^2+c_\alpha\varepsilon^{-1}\|u\|_{L^2(\Omega_\varepsilon)}^2\|u\|_{H^1(\Omega_\varepsilon)}^2.
  \end{multline*}
  We apply this inequality and \eqref{Pf_TE:DivUa_O} to \eqref{Pf_TE:J3_Exp} to show that
  \begin{align} \label{Pf_TE:Est_J3}
    \begin{aligned}
      |J_3^3| &\leq c\left(|(\mathrm{div}\,u^a,|\omega|^2)_{L^2(\Omega_\varepsilon)}|+|(\omega,(\omega\cdot\nabla)u^a)_{L^2(\Omega_\varepsilon)}|\right) \\
      &\leq c\left(\alpha+\varepsilon^{1/2}\|u\|_{H^1(\Omega_\varepsilon)}\right)\|u\|_{H^2(\Omega_\varepsilon)}^2+c_\alpha\varepsilon^{-1}\|u\|_{L^2(\Omega_\varepsilon)}^2\|u\|_{H^1(\Omega_\varepsilon)}^2.
    \end{aligned}
  \end{align}
  Since $J_3=J_3^1+J_3^2+J_3^3$, we see by \eqref{Pf_TE:Est_J1}, \eqref{Pf_TE:Est_J2}, and \eqref{Pf_TE:Est_J3} that
  \begin{align*}
    |J_3| \leq c\left(\alpha+\varepsilon^{1/2}\|u\|_{H^1(\Omega_\varepsilon)}\right)\|u\|_{H^2(\Omega_\varepsilon)}^2+c_\alpha\varepsilon^{-1}\|u\|_{L^2(\Omega_\varepsilon)}^2\|u\|_{H^1(\Omega_\varepsilon)}^2
  \end{align*}
  and this inequality combined with \eqref{Pf_TE:Est_I1} and \eqref{Pf_TE:Est_I2} yields
  \begin{align*}
    \left|\bigl((u\cdot\nabla)u,A_\varepsilon u\bigr)_{L^2(\Omega_\varepsilon)}\right| &\leq |J_1|+|J_2|+|J_3| \\
    &\leq \left(c_1\alpha+c_2\varepsilon^{1/2}\|u\|_{H^1(\Omega_\varepsilon)}\right)\|u\|_{H^2(\Omega_\varepsilon)}^2 \\
    &\qquad +c_\alpha\left(\|u\|_{L^2(\Omega_\varepsilon)}^2\|u\|_{H^1(\Omega_\varepsilon)}^4+\varepsilon^{-1}\|u\|_{L^2(\Omega_\varepsilon)}^2\|u\|_{H^1(\Omega_\varepsilon)}^2\right).
  \end{align*}
  Here $c_\alpha>0$ is a constant depending only on $\alpha$ and $c_1,c_2>0$ are constants independent of $\varepsilon$ and $\alpha$.
  Replacing $c_1\alpha$ by $\alpha$ in the above inequality we obtain \eqref{E:Tri_Est}.
\end{proof}

Finally, we fix $\alpha$ and write \eqref{E:Tri_Est} in terms of the Stokes operator $A_\varepsilon$.

\begin{lemma} \label{L:Tri_Est_A}
  There exist constants $d_1,d_2>0$ such that
  \begin{multline} \label{E:Tri_Est_A}
    \left|\bigl((u\cdot\nabla)u,A_\varepsilon u\bigr)_{L^2(\Omega_\varepsilon)}\right| \leq \left(\frac{1}{4}+d_1\varepsilon^{1/2}\|A_\varepsilon^{1/2}u\|_{L^2(\Omega_\varepsilon)}\right)\|A_\varepsilon u\|_{L^2(\Omega_\varepsilon)}^2 \\
    +d_2\left(\|u\|_{L^2(\Omega_\varepsilon)}^2\|A_\varepsilon^{1/2}u\|_{L^2(\Omega_\varepsilon)}^4+\varepsilon^{-1}\|u\|_{L^2(\Omega_\varepsilon)}^2\|A_\varepsilon^{1/2}u\|_{L^2(\Omega_\varepsilon)}^2\right)
  \end{multline}
  for all $\varepsilon\in(0,\varepsilon_0]$ and $u\in D(A_\varepsilon)$.
\end{lemma}

\begin{proof}
  Applying \eqref{E:Stokes_H1} and \eqref{E:Stokes_H2} to the right-hand side of \eqref{E:Tri_Est} we get
  \begin{multline*}
    \left|\bigl((u\cdot\nabla)u,A_\varepsilon u\bigr)_{L^2(\Omega_\varepsilon)}\right| \leq \left(c\alpha+d_\alpha^1\varepsilon^{1/2}\|A_\varepsilon^{1/2}u\|_{L^2(\Omega_\varepsilon)}\right)\|A_\varepsilon u\|_{L^2(\Omega_\varepsilon)}^2 \\
    +d_\alpha^2\left(\|u\|_{L^2(\Omega_\varepsilon)}^2\|A_\varepsilon^{1/2}u\|_{L^2(\Omega_\varepsilon)}^4+\varepsilon^{-1}\|u\|_{L^2(\Omega_\varepsilon)}^2\|A_\varepsilon^{1/2}u\|_{L^2(\Omega_\varepsilon)}^2\right)
  \end{multline*}
  with positive constants $c$, $d_\alpha^1$, and $d_\alpha^2$ independent of $\varepsilon$.
  We take $\alpha=1/4c$ in the above inequality to obtain \eqref{E:Tri_Est_A}.
\end{proof}

\section{Global existence and estimates of a strong solution} \label{S:GE}
Based on the results in the previous sections we prove Theorems~\ref{T:GE} and~\ref{T:UE}.
As in Section~\ref{S:Tri} we impose Assumptions~\ref{Assump_1} and~\ref{Assump_2} and fix the constant $\varepsilon_0$ given in Lemma~\ref{L:Uni_aeps}.
For $\varepsilon\in(0,\varepsilon_0]$ let $\mathcal{H}_\varepsilon$ and $\mathcal{V}_\varepsilon$ be the function spaces given by \eqref{E:Def_Heps} and $A_\varepsilon$ the Stokes operator on $\mathcal{H}_\varepsilon$.
We also write $\bar{\eta}=\eta\circ\pi$ for the constant extension of a function of $\eta$ on $\Gamma$ in the normal direction of $\Gamma$.

First we recall the well-known result on the local-in-time existence of a strong solution to the Navier--Stokes equations (see e.g.~\cites{BoFa13,CoFo88,So01,Te79}).

\begin{theorem} \label{T:LE}
  For a fixed $\varepsilon\in(0,\varepsilon_0]$ let
  \begin{align*}
    u_0^\varepsilon\in \mathcal{V}_\varepsilon, \quad f^\varepsilon\in L^\infty(0,\infty;L^2(\Omega_\varepsilon)^3).
  \end{align*}
  When the condition (A3) of Assumption~\ref{Assump_2} is imposed, suppose that $f^\varepsilon(t)\in\mathcal{R}_g^\perp$ for a.a. $t\in(0,\infty)$.
  Then there exists $T_0\in(0,\infty)$ depending on $\Omega_\varepsilon$, $\nu$, $u_0^\varepsilon$, and $f^\varepsilon$ such that the problem \eqref{E:NS_CTD} admits a strong solution $u^\varepsilon$ on $[0,T_0)$ satisfying
  \begin{align*}
    u^\varepsilon \in C([0,T];\mathcal{V}_\varepsilon)\cap L^2(0,T;D(A_\varepsilon))\cap H^1(0,T;\mathcal{H}_\varepsilon) \quad\text{for all}\quad T\in(0,T_0).
  \end{align*}
  If $u^\varepsilon$ is maximally defined on the time interval $[0,T_{\max})$ and $T_{\max}$ is finite, then
  \begin{align*}
    \lim_{t\to T_{\max}^-}\|A_\varepsilon^{1/2}u^\varepsilon(t)\|_{L^2(\Omega_\varepsilon)} = \infty.
  \end{align*}
\end{theorem}

Note that the assumption $f^\varepsilon(t)\in\mathcal{R}_g^\perp$ for a.a. $t\in(0,\infty)$ is required to recover the original problem \eqref{E:NS_CTD} properly from its abstract form (see Remark~\ref{R:Ext_Orth}).

To establish the global-in-time existence of the strong solution $u^\varepsilon$ we show that the $L^2(\Omega_\varepsilon)$-norm of $A_\varepsilon^{1/2}u^\varepsilon(t)$ is bounded uniformly in $t$.
We argue by a standard energy method and use the uniform Gronwall inequality (see~\cite{SeYo02}*{Lemma~D.3}).

\begin{lemma}[Uniform Gronwall inequality] \label{L:Uni_Gronwall}
  Let $z$, $\xi$, and $\zeta$ be nonnegative functions in $L_{loc}^1([0,T);\mathbb{R})$, $T\in(0,\infty]$.
  Suppose that $z\in C(0,T;\mathbb{R})$ and
  \begin{align*}
    \frac{dz}{dt}(t) \leq \xi(t)z(t)+\zeta(t) \quad\text{for a.a.}\quad t\in(0,T).
  \end{align*}
  Then $z\in L_{loc}^\infty(0,T;\mathbb{R})$ and
  \begin{align*}
    z(t_2) \leq \left(\frac{1}{t_2-t_1}\int_{t_1}^{t_2}z(s)\,ds+\int_{t_1}^{t_2}\zeta(s)\,ds\right)\exp\left(\int_{t_1}^{t_2}\xi(s)\,ds\right)
  \end{align*}
  for all $t_1,t_2\in(0,T)$ with $t_1<t_2$.
\end{lemma}

We also use an estimate for the duality product between a vector field on $\Omega_\varepsilon$ and the constant extension of a tangential vector field on $\Gamma$.

\begin{lemma} \label{L:Con_Dual}
  There exists a constant $c>0$ independent of $\varepsilon$ such that
  \begin{align} \label{E:Con_Dual}
    \left|(\bar{v},u)_{L^2(\Omega_\varepsilon)}\right| \leq c\varepsilon^{1/2}\|v\|_{H^{-1}(\Gamma,T\Gamma)}\|u\|_{H^1(\Omega_\varepsilon)}
  \end{align}
  for all $v\in L^2(\Gamma,T\Gamma)$ and $u\in H^1(\Omega_\varepsilon)^3$.
\end{lemma}

\begin{proof}
  We use the notation \eqref{E:Pull_Dom} and define
  \begin{align*}
    \eta(y) := \int_{\varepsilon g_0(y)}^{\varepsilon g_1(y)}u^\sharp(y,r)J(y,r)\,dr, \quad y\in\Gamma.
  \end{align*}
  In what follows, we suppress the arguments of functions.
  Let us show $\eta\in H^1(\Gamma)^3$.
  By \eqref{E:Jac_Bound_02}, H\"{o}lder's inequality, and \eqref{E:CoV_Equiv},
  \begin{align} \label{Pf_CDu:Eta_L2}
    \|\eta\|_{L^2(\Gamma)}^2 \leq \int_\Gamma\varepsilon g\left(\int_{\varepsilon g_0}^{\varepsilon g_1}|u^\sharp|^2\,dr\right)d\mathcal{H}^2 \leq c\varepsilon\|u\|_{L^2(\Omega_\varepsilon)}^2.
  \end{align}
  Also, by the same calculations as in the proof of Lemma~\ref{L:Ave_Der} we have
  \begin{align*}
    \nabla_\Gamma\eta = \int_{\varepsilon g_0}^{\varepsilon g_1}\left\{\frac{\partial}{\partial r}\Bigl(J\psi_\varepsilon^\sharp\otimes u^\sharp \Bigr)+J(B\nabla u)^\sharp+\nabla_\Gamma J\otimes u^\sharp\right\}\,dr \quad\text{on}\quad \Gamma,
  \end{align*}
  where $B$ and $\psi_\varepsilon$ are given by \eqref{E:Ave_Der_Aux}.
  By this equality, \eqref{E:Jac_Bound_02}, \eqref{E:ADA_Bound}, and \eqref{E:ADA_Grad_Bound},
  \begin{align*}
    |\nabla_\Gamma\eta| \leq c\int_{\varepsilon g_0}^{\varepsilon g_1}(|u^\sharp|+|(\nabla u)^\sharp|)\,dr \quad\text{on}\quad \Gamma.
  \end{align*}
  Hence H\"{o}lder's inequality and \eqref{E:CoV_Equiv} imply that
  \begin{align} \label{Pf_Cdu:Eta_H1}
    \|\nabla_\Gamma\eta\|_{L^2(\Gamma)}^2 \leq c\int_\Gamma\varepsilon g\left(\int_{\varepsilon g_0}^{\varepsilon g_1}(|u^\sharp|^2+|(\nabla u)^\sharp|^2)\,dr\right)d\mathcal{H}^2 \leq c\varepsilon\|u\|_{H^1(\Omega_\varepsilon)}^2.
  \end{align}
  From \eqref{Pf_CDu:Eta_L2} and \eqref{Pf_Cdu:Eta_H1} we deduce that $\eta\in H^1(\Gamma)^3$ and
  \begin{align} \label{Pf_Cdu:PEta}
    \|P\eta\|_{H^1(\Gamma)} \leq c\|\eta\|_{H^1(\Gamma)} \leq c\varepsilon^{1/2}\|u\|_{H^1(\Omega_\varepsilon)}.
  \end{align}
  Now we observe by \eqref{E:CoV_Dom} and $v\in L^2(\Gamma,T\Gamma)$ that
  \begin{align*}
    (\bar{v},u)_{L^2(\Omega_\varepsilon)} = \int_\Gamma v\cdot\left(\int_{\varepsilon g_0}^{\varepsilon g_1}u^\sharp J\,dr\right)d\mathcal{H}^2 = (v,\eta)_{L^2(\Gamma)} = (v,P\eta)_{L^2(\Gamma)}.
  \end{align*}
  Moreover, since $(v,P\eta)_{L^2(\Gamma)} = [v,P\eta]_{T\Gamma}$ by $P\eta\in H^1(\Gamma,T\Gamma)$ (see Section~\ref{SS:Pre_Surf}),
  \begin{align*}
    \left|(\bar{v},u)_{L^2(\Omega_\varepsilon)}\right| = \bigl|[v,P\eta]_{T\Gamma}\bigr| \leq \|v\|_{H^{-1}(\Gamma,T\Gamma)}\|P\eta\|_{H^1(\Gamma)}.
  \end{align*}
  Applying \eqref{Pf_Cdu:PEta} to this inequality we obtain \eqref{E:Con_Dual}.
\end{proof}

Now we are ready to establish the global existence of a strong solution to \eqref{E:NS_CTD}.

\begin{proof}[Proof of Theorem~\ref{T:GE}]
  We follow the idea of the proofs of~\cite{Ho10}*{Theorem~7.4} and~\cite{HoSe10}*{Theorem~3.1}.
  Under Assumptions~\ref{Assump_1} and~\ref{Assump_2}, let $\varepsilon_0$ and $d_1$ be the constants given in Lemmas~\ref{L:Uni_aeps} and~\ref{L:Tri_Est_A}, and
  \begin{align} \label{Pf_GE:C0}
    c_0 := \min\left\{1, \frac{C_1^2}{4C_2}, \frac{C_1^2}{4C_3}\right\}, \quad C_1 := \frac{1}{4d_1},
  \end{align}
  where $C_2$ and $C_3$ are positive constants fixed later.
  For $\varepsilon\in(0,\varepsilon_0]$ let
  \begin{align*}
    u_0^\varepsilon\in \mathcal{V}_\varepsilon, \quad f^\varepsilon\in L^\infty(0,\infty;L^2(\Omega_\varepsilon)^3)
  \end{align*}
  satisfy \eqref{E:GE_Data} (and $f^\varepsilon(t)\in\mathcal{R}_g^\perp$ for a.a. $t\in(0,\infty)$ when the condition (A3) of Assumption~\ref{Assump_2} is imposed) and $u^\varepsilon$ be the strong solution to \eqref{E:NS_CTD} on the maximal time interval $[0,T_{\max})$ given in Theorem~\ref{T:LE}.
  In what follows, we write $c$ for a general positive constant independent of $\varepsilon$, $c_0$, and $T_{\max}$.
  First note that
  \begin{align} \label{Pf_GE:Ini_StL2}
    \|A_\varepsilon^{1/2}u_0^\varepsilon\|_{L^2(\Omega_\varepsilon)}^2 \leq C_2\|u_0^\varepsilon\|_{H^1(\Omega_\varepsilon)}^2 \leq C_2c_0\varepsilon^{-1}
  \end{align}
  by \eqref{E:GE_Data} and \eqref{E:Stokes_H1} with a constant $C_2>0$ independent of $\varepsilon$ and $c_0$.
  Also,
  \begin{align} \label{Pf_GE:Ini_Data}
    \begin{aligned}
      \|u_0^\varepsilon\|_{L^2(\Omega_\varepsilon)}^2 &\leq c\left(\left\|u_0^\varepsilon-\overline{M_\tau u_0^\varepsilon}\right\|_{L^2(\Omega_\varepsilon)}^2+\left\|\overline{M_\tau u_0^\varepsilon}\right\|_{L^2(\Omega_\varepsilon)}^2\right) \\
      &\leq c\left(\varepsilon^2\|u_0^\varepsilon\|_{H^1(\Omega_\varepsilon)}^2+\varepsilon\|M_\tau u_0^\varepsilon\|_{L^2(\Gamma)}^2\right)
    \end{aligned}
  \end{align}
  by \eqref{E:Con_Lp} and \eqref{E:AveT_Diff_Dom}, and thus it follows from \eqref{E:GE_Data} that
  \begin{align} \label{Pf_GE:Ini_L2}
    \|u_0^\varepsilon\|_{L^2(\Omega_\varepsilon)}^2 \leq cc_0.
  \end{align}
  To prove $T_{\max}=\infty$ we first derive estimates for
  \begin{align*}
    \|u^\varepsilon(t)\|_{L^2(\Omega_\varepsilon)}^2, \quad \int_t^{\min\{t+1,T_{\max}\}}\|A_\varepsilon^{1/2}u^\varepsilon(s)\|_{L^2(\Omega_\varepsilon)}^2\,ds, \quad t\in[0,T_{\max})
  \end{align*}
  with constants explicitly depending on $\varepsilon$.
  Taking the $L^2(\Omega_\varepsilon)$-inner product of
  \begin{align} \label{Pf_GE:Ab_NS}
    \partial_tu^\varepsilon+A_\varepsilon u^\varepsilon+\mathbb{P}_\varepsilon(u^\varepsilon\cdot\nabla)u^\varepsilon = \mathbb{P}_\varepsilon f^\varepsilon \quad\text{on}\quad (0,T_{\max})
  \end{align}
  with $u^\varepsilon$ and using \eqref{E:L2in_Ahalf} and
  \begin{align*}
    (\mathbb{P}_\varepsilon(u^\varepsilon\cdot\nabla)u^\varepsilon,u^\varepsilon)_{L^2(\Omega_\varepsilon)} &= \bigl((u^\varepsilon\cdot\nabla)u^\varepsilon,u^\varepsilon\bigr)_{L^2(\Omega_\varepsilon)} = 0
  \end{align*}
  by integration by parts, $\mathrm{div}\,u^\varepsilon=0$ in $\Omega_\varepsilon$, and $u^\varepsilon\cdot n_\varepsilon=0$ on $\Gamma_\varepsilon$ we get
  \begin{align} \label{Pf_GE:L2_Inner}
    \frac{1}{2}\frac{d}{dt}\|u^\varepsilon\|_{L^2(\Omega_\varepsilon)}^2+\|A_\varepsilon^{1/2}u^\varepsilon\|_{L^2(\Omega_\varepsilon)}^2 = (\mathbb{P}_\varepsilon f^\varepsilon,u^\varepsilon)_{L^2(\Omega_\varepsilon)} \quad\text{on}\quad (0,T_{\max}).
  \end{align}
  We split the right-hand side into
  \begin{align*}
    J_1 := \Bigl(\mathbb{P}_\varepsilon f^\varepsilon,u^\varepsilon-\overline{M_\tau u^\varepsilon}\Bigr)_{L^2(\Omega_\varepsilon)}, \quad J_2 := \Bigl(\mathbb{P}_\varepsilon f^\varepsilon,\overline{M_\tau u^\varepsilon}\Bigr)_{L^2(\Omega_\varepsilon)}
  \end{align*}
  and estimate them separately.
  By \eqref{E:AveT_Diff_Dom} we have
  \begin{align*}
    |J_1| \leq \|\mathbb{P}_\varepsilon f^\varepsilon\|_{L^2(\Omega_\varepsilon)}\left\|u^\varepsilon-\overline{M_\tau u^\varepsilon}\right\|_{L^2(\Omega_\varepsilon)} \leq c\varepsilon\|\mathbb{P}_\varepsilon f^\varepsilon\|_{L^2(\Omega_\varepsilon)}\|u^\varepsilon\|_{H^1(\Omega_\varepsilon)}.
  \end{align*}
  To $J_2$ we use \eqref{E:AveT_Inner} and \eqref{E:Con_Dual} to get
  \begin{align*}
    |J_2| &\leq \left|\Bigl(\overline{M_\tau\mathbb{P}_\varepsilon f^\varepsilon},u^\varepsilon\Bigr)_{L^2(\Omega_\varepsilon)}\right|+\left|\Bigl(\mathbb{P}_\varepsilon f^\varepsilon,\overline{M_\tau u^\varepsilon}\Bigr)_{L^2(\Omega_\varepsilon)}-\Bigl(\overline{M_\tau\mathbb{P}_\varepsilon f^\varepsilon},u^\varepsilon\Bigr)_{L^2(\Omega_\varepsilon)}\right| \\
    &\leq c\left(\varepsilon^{1/2}\|M_\tau\mathbb{P}_\varepsilon f^\varepsilon\|_{H^{-1}(\Gamma,T\Gamma)}\|u^\varepsilon\|_{H^1(\Omega_\varepsilon)}+\varepsilon\|\mathbb{P}_\varepsilon f^\varepsilon\|_{L^2(\Omega_\varepsilon)}\|u^\varepsilon\|_{L^2(\Omega_\varepsilon)}\right).
  \end{align*}
  We apply these estimates to $(\mathbb{P}_\varepsilon f^\varepsilon,u^\varepsilon)_{L^2(\Omega_\varepsilon)}=J_1+J_2$ and use
  \begin{align} \label{Pf_GE:L2_Po}
    \|u^\varepsilon\|_{L^2(\Omega_\varepsilon)} \leq \|u^\varepsilon\|_{H^1(\Omega_\varepsilon)} \leq c\|A_\varepsilon^{1/2}u^\varepsilon\|_{L^2(\Omega_\varepsilon)},
  \end{align}
  where the second inequality is due to \eqref{E:Stokes_H1}, and Young's inequality to obtain
  \begin{multline*}
    |(\mathbb{P}_\varepsilon f^\varepsilon,u^\varepsilon)_{L^2(\Omega_\varepsilon)}| \\
    \leq \frac{1}{2}\|A_\varepsilon^{1/2}u^\varepsilon\|_{L^2(\Omega_\varepsilon)}^2+c\left(\varepsilon^2\|\mathbb{P}_\varepsilon f^\varepsilon\|_{L^2(\Omega_\varepsilon)}^2+\varepsilon\|M_\tau\mathbb{P}_\varepsilon f^\varepsilon\|_{H^{-1}(\Gamma,T\Gamma)}^2\right).
  \end{multline*}
  From this inequality and \eqref{Pf_GE:L2_Inner} we deduce that
  \begin{multline} \label{Pf_GE:L2_InEst1}
    \frac{d}{dt}\|u^\varepsilon\|_{L^2(\Omega_\varepsilon)}^2+\|A_\varepsilon^{1/2}u^\varepsilon\|_{L^2(\Omega_\varepsilon)}^2 \\
    \leq c\left(\varepsilon^2\|\mathbb{P}_\varepsilon f^\varepsilon\|_{L^2(\Omega_\varepsilon)}^2+\varepsilon\|M_\tau\mathbb{P}_\varepsilon f^\varepsilon\|_{H^{-1}(\Gamma,T\Gamma)}^2\right) \quad\text{on}\quad (0,T_{\max}).
  \end{multline}
  By \eqref{Pf_GE:L2_Po} we further get
  \begin{multline*}
    \frac{d}{dt}\|u^\varepsilon\|_{L^2(\Omega_\varepsilon)}^2+\frac{1}{a_1}\|u^\varepsilon\|_{L^2(\Omega_\varepsilon)}^2 \\
    \leq c\left(\varepsilon^2\|\mathbb{P}_\varepsilon f^\varepsilon\|_{L^2(\Omega_\varepsilon)}^2+\varepsilon\|M_\tau\mathbb{P}_\varepsilon f^\varepsilon\|_{H^{-1}(\Gamma,T\Gamma)}^2\right) \quad\text{on}\quad (0,T_{\max})
  \end{multline*}
  with a constant $a_1>0$ independent of $\varepsilon$, $c_0$, and $T_{\max}$.
  For each $t\in(0,T_{\max})$ we multiply both sides of this inequality at $s\in(0,t)$ by $e^{(s-t)/a_1}$ and integrate them over $(0,t)$.
  Then we have
  \begin{multline} \label{Pf_GE:LinfL2}
    \|u^\varepsilon(t)\|_{L^2(\Omega_\varepsilon)}^2 \leq e^{-t/a_1}\|u_0^\varepsilon\|_{L^2(\Omega_\varepsilon)}^2 \\
    +ca_1(1-e^{-t/a_1})\left(\varepsilon^2\|\mathbb{P}_\varepsilon f^\varepsilon\|_{L^\infty(0,\infty;L^2(\Omega_\varepsilon))}^2+\varepsilon\|M_\tau\mathbb{P}_\varepsilon f^\varepsilon\|_{L^\infty(0,\infty;H^{-1}(\Gamma,T\Gamma))}^2\right).
  \end{multline}
  Also, integrating \eqref{Pf_GE:L2_InEst1} over $(t,t_\ast)$ with $t_\ast:=\min\{t+1,T_{\max}\}$ we deduce that
  \begin{multline} \label{Pf_GE:L2H1}
    \int_t^{t_\ast}\|A_\varepsilon^{1/2}u^\varepsilon(s)\|_{L^2(\Omega_\varepsilon)}^2\,ds \leq \|u_0^\varepsilon\|_{L^2(\Omega_\varepsilon)}^2 \\
    +c\left(\varepsilon^2\|\mathbb{P}_\varepsilon f^\varepsilon\|_{L^\infty(0,\infty;L^2(\Omega_\varepsilon))}^2+\varepsilon\|M_\tau\mathbb{P}_\varepsilon f^\varepsilon\|_{L^\infty(0,\infty;H^{-1}(\Gamma,T\Gamma))}^2\right).
  \end{multline}
  Hence we apply \eqref{E:GE_Data} and \eqref{Pf_GE:Ini_L2} to the right-hand sides of \eqref{Pf_GE:LinfL2} and \eqref{Pf_GE:L2H1} to get
  \begin{align} \label{Pf_GE:L2_Ener}
    \|u^\varepsilon(t)\|_{L^2(\Omega_\varepsilon)}^2+\int_t^{t_\ast}\|A_\varepsilon^{1/2}u^\varepsilon(s)\|_{L^2(\Omega_\varepsilon)}^2\,ds \leq cc_0 \quad\text{for all}\quad t\in[0,T_{\max}).
  \end{align}
  Next we show that $\|A_\varepsilon^{1/2}u^\varepsilon(t)\|_{L^2(\Omega_\varepsilon)}$ is uniformly bounded in $t\in[0,T_{\max})$ (note that it is continuous on $[0,T_{\max})$ by $u^\varepsilon\in C([0,T_{\max});\mathcal{V}_\varepsilon)$).
  Our goal is to prove
  \begin{align} \label{Pf_GE:Goal}
    \varepsilon^{1/2}\|A_\varepsilon^{1/2}u^\varepsilon(t)\|_{L^2(\Omega_\varepsilon)} < C_1 = \frac{1}{4d_1} \quad\text{for all}\quad t\in[0,T_{\max}).
  \end{align}
  If \eqref{Pf_GE:Goal} is valid, then Theorem~\ref{T:LE} implies that $T_{\max}=\infty$, i.e. the strong solution $u^\varepsilon$ exists on the whole time interval $[0,\infty)$.
  First note that \eqref{Pf_GE:Goal} is valid at $t=0$ by \eqref{Pf_GE:C0} and \eqref{Pf_GE:Ini_StL2}.
  Let us prove \eqref{Pf_GE:Goal} for all $t\in(0,T_{\max})$ by contradiction.
  Assume to the contrary that there exists $T\in(0,T_{\max})$ such that
  \begin{align}
    \varepsilon^{1/2}\|A_\varepsilon^{1/2}u^\varepsilon(t)\|_{L^2(\Omega_\varepsilon)} &< C_1 \quad\text{for all}\quad t\in[0,T), \label{Pf_GE:Cont_Ine} \\
    \varepsilon^{1/2}\|A_\varepsilon^{1/2}u^\varepsilon(T)\|_{L^2(\Omega_\varepsilon)} &= C_1. \label{Pf_GE:Cont_Eq}
  \end{align}
  We consider \eqref{Pf_GE:Ab_NS} on $(0,T]$ and take its $L^2(\Omega_\varepsilon)$-inner product with $A_\varepsilon u^\varepsilon$ to get
  \begin{multline} \label{PF_GE:H1_Inner}
    \frac{1}{2}\frac{d}{dt}\|A_\varepsilon^{1/2}u^\varepsilon\|_{L^2(\Omega_\varepsilon)}^2+\|A_\varepsilon u^\varepsilon\|_{L^2(\Omega_\varepsilon)}^2 \\
    \leq \left|\bigl((u^\varepsilon\cdot\nabla)u^\varepsilon,A_\varepsilon u^\varepsilon\bigr)_{L^2(\Omega_\varepsilon)}\right|+|(\mathbb{P}_\varepsilon f^\varepsilon,A_\varepsilon u^\varepsilon)_{L^2(\Omega_\varepsilon)}| \quad\text{on}\quad (0,T]
  \end{multline}
  by \eqref{E:L2in_Ahalf}.
  By \eqref{E:Tri_Est_A} and \eqref{Pf_GE:Cont_Ine}--\eqref{Pf_GE:Cont_Eq} with $C_1=1/4d_1$ we have
  \begin{multline*}
    \left|\bigl((u^\varepsilon\cdot\nabla)u^\varepsilon,A_\varepsilon u^\varepsilon\bigr)_{L^2(\Omega_\varepsilon)}\right| \leq \frac{1}{2}\|A_\varepsilon u^\varepsilon\|_{L^2(\Omega_\varepsilon)}^2 \\
    +d_2\left(\|u^\varepsilon\|_{L^2(\Omega_\varepsilon)}^2\|A_\varepsilon^{1/2}u^\varepsilon\|_{L^2(\Omega_\varepsilon)}^4+\varepsilon^{-1}\|u^\varepsilon\|_{L^2(\Omega_\varepsilon)}^2\|A_\varepsilon^{1/2}u^\varepsilon\|_{L^2(\Omega_\varepsilon)}^2\right)
  \end{multline*}
  on $(0,T]$.
  Also, Young's inequality yields
  \begin{align*}
    |(\mathbb{P}_\varepsilon f^\varepsilon,A_\varepsilon u^\varepsilon)_{L^2(\Omega_\varepsilon)}| \leq \frac{1}{4}\|A_\varepsilon u^\varepsilon\|_{L^2(\Omega_\varepsilon)}^2+\|\mathbb{P}_\varepsilon f^\varepsilon\|_{L^2(\Omega_\varepsilon)}^2.
  \end{align*}
  Applying these inequalities to the right-hand side of \eqref{PF_GE:H1_Inner} we obtain
  \begin{align} \label{Pf_GE:H1_InEst1}
    \frac{d}{dt}\|A_\varepsilon^{1/2}u^\varepsilon\|_{L^2(\Omega_\varepsilon)}^2+\frac{1}{2}\|A_\varepsilon u^\varepsilon\|_{L^2(\Omega_\varepsilon)}^2 \leq \xi\|A_\varepsilon^{1/2}u^\varepsilon\|_{L^2(\Omega_\varepsilon)}^2+\zeta \quad\text{on}\quad (0,T],
  \end{align}
  where the functions $\xi$ and $\zeta$ are given by
  \begin{align} \label{Pf_GE:Def_Xi}
    \begin{aligned}
      \xi(t) &:= 2d_2\|u^\varepsilon(t)\|_{L^2(\Omega_\varepsilon)}^2\|A_\varepsilon^{1/2}u^\varepsilon(t)\|_{L^2(\Omega_\varepsilon)}^2, \\
     \zeta(t) &:= 2\left(d_2\varepsilon^{-1}\|u^\varepsilon(t)\|_{L^2(\Omega_\varepsilon)}^2\|A_\varepsilon^{1/2}u^\varepsilon(t)\|_{L^2(\Omega_\varepsilon)}^2+\|\mathbb{P}_\varepsilon f^\varepsilon(t)\|_{L^2(\Omega_\varepsilon)}^2\right)
    \end{aligned}
  \end{align}
  for $t\in(0,T]$.
  By \eqref{Pf_GE:L2_Ener}, \eqref{Pf_GE:Cont_Ine}, and \eqref{Pf_GE:Cont_Eq} we see that
  \begin{align*}
    \xi \leq cc_0\varepsilon^{-1}, \quad \zeta \leq c\left(c_0\varepsilon^{-1}\|A_\varepsilon^{1/2}u^\varepsilon\|_{L^2(\Omega_\varepsilon)}^2+\|\mathbb{P}_\varepsilon f^\varepsilon\|_{L^2(\Omega_\varepsilon)}^2\right) \quad\text{on}\quad (0,T].
  \end{align*}
  Applying these inequalities to \eqref{Pf_GE:H1_InEst1} we have
  \begin{multline} \label{Pf_GE:H1_InEst2}
    \frac{d}{dt}\|A_\varepsilon^{1/2}u^\varepsilon\|_{L^2(\Omega_\varepsilon)}^2+\frac{1}{2}\|A_\varepsilon u^\varepsilon\|_{L^2(\Omega_\varepsilon)}^2 \\
    \leq c\left(c_0\varepsilon^{-1}\|A_\varepsilon^{1/2}u^\varepsilon\|_{L^2(\Omega_\varepsilon)}^2+\|\mathbb{P}_\varepsilon f^\varepsilon\|_{L^2(\Omega_\varepsilon)}^2\right) \quad\text{on}\quad (0,T].
  \end{multline}
  From \eqref{E:Stokes_Po} and \eqref{Pf_GE:H1_InEst2} we further deduce that
  \begin{align*}
    \frac{d}{dt}\|A_\varepsilon^{1/2}u^\varepsilon\|_{L^2(\Omega_\varepsilon)}^2+\frac{1}{a_2}\|A_\varepsilon^{1/2}u^\varepsilon\|_{L^2(\Omega_\varepsilon)}^2 \leq c\left(c_0\varepsilon^{-1}\|A_\varepsilon^{1/2}u^\varepsilon\|_{L^2(\Omega_\varepsilon)}^2+\|\mathbb{P}_\varepsilon f^\varepsilon\|_{L^2(\Omega_\varepsilon)}^2\right)
  \end{align*}
  on $(0,T]$ with a constant $a_2>0$ independent of $\varepsilon$, $c_0$, and $T$.
  For $t\in(0,T]$ we multiply both sides of the above inequality at $s\in(0,t)$ by $e^{(s-t)/a_2}$ and integrate them over $(0,t)$.
  Then we have
  \begin{multline} \label{Pf_GE:LinfH1_Int}
    \|A_\varepsilon^{1/2}u^\varepsilon(t)\|_{L^2(\Omega_\varepsilon)}^2 \\
    \leq e^{-t/a_2}\|A_\varepsilon^{1/2}u_0^\varepsilon\|_{L^2(\Omega_\varepsilon)}^2+cc_0\varepsilon^{-1}\int_0^te^{(s-t)/a_2}\|A_\varepsilon^{1/2}u^\varepsilon(s)\|_{L^2(\Omega_\varepsilon)}^2\,ds \\
    +ca_2(1-e^{-t/a_2})\|\mathbb{P}_\varepsilon f^\varepsilon\|_{L^\infty(0,\infty;L^2(\Omega_\varepsilon))}^2.
  \end{multline}
  When $t\leq T_\ast:=\min\{1,T\}$ we apply \eqref{E:GE_Data}, \eqref{Pf_GE:Ini_StL2}, and \eqref{Pf_GE:L2_Ener} to \eqref{Pf_GE:LinfH1_Int} to get
  \begin{align} \label{Pf_GE:LinfH1_1}
    \|A_\varepsilon^{1/2}u^\varepsilon(t)\|_{L^2(\Omega_\varepsilon)}^2 \leq cc_0(1+c_0)\varepsilon^{-1} \leq cc_0\varepsilon^{-1} \quad\text{for all}\quad t\in(0,T_\ast].
  \end{align}
  Note that $c_0\leq1$ by \eqref{Pf_GE:C0}.
  Next we assume $T\geq 1$ and derive an estimate similar to \eqref{Pf_GE:LinfH1_1} for $t\in[1,T]$.
  Since
  \begin{align*}
    \frac{d}{dt}\|A_\varepsilon^{1/2}u^\varepsilon\|_{L^2(\Omega_\varepsilon)}^2 \leq \xi\|A_\varepsilon^{1/2}u^\varepsilon\|_{L^2(\Omega_\varepsilon)}^2+\zeta \quad\text{on}\quad (0,T]
  \end{align*}
  by \eqref{Pf_GE:H1_InEst1}, we can use Lemma~\ref{L:Uni_Gronwall} with $z(t)=\|A_\varepsilon^{1/2}u^\varepsilon(t)\|_{L^2(\Omega_\varepsilon)}^2$ to obtain
  \begin{multline} \label{Pf_GE:Uni_Gron}
    \|A_\varepsilon^{1/2}u^\varepsilon(t)\|_{L^2(\Omega_\varepsilon)}^2 \\
    \leq \left(\int_{t-1}^t\|A_\varepsilon^{1/2}u^\varepsilon(s)\|_{L^2(\Omega_\varepsilon)}^2\,ds+\int_{t-1}^t\zeta(s)\,ds\right)\exp\left(\int_{t-1}^t\xi(s)\,ds\right)
  \end{multline}
  for all $t\in[1,T]$.
  Moreover, the functions $\xi$ and $\zeta$ given by \eqref{Pf_GE:Def_Xi} satisfy
  \begin{align*}
    \int_{t-1}^t\xi(s)\,ds &\leq cc_0\int_{t-1}^t\|A_\varepsilon^{1/2}u^\varepsilon(s)\|_{L^2(\Omega_\varepsilon)}^2\,ds \leq c, \\
    \int_{t-1}^t\zeta(s)\,ds &\leq c\left(c_0\varepsilon^{-1}\int_{t-1}^t\|A_\varepsilon^{1/2}u^\varepsilon(s)\|_{L^2(\Omega_\varepsilon)}^2\,ds+\|\mathbb{P}_\varepsilon f^\varepsilon\|_{L^\infty(0,\infty;L^2(\Omega_\varepsilon))}^2\right) \\
    &\leq cc_0\varepsilon^{-1}
  \end{align*}
  by \eqref{E:GE_Data}, \eqref{Pf_GE:L2_Ener}, and $c_0\leq1$.
  Using these inequalities and \eqref{Pf_GE:L2_Ener} to \eqref{Pf_GE:Uni_Gron} we have
  \begin{align} \label{Pf_GE:LinfH1_2}
    \|A_\varepsilon^{1/2}u^\varepsilon(t)\|_{L^2(\Omega_\varepsilon)}^2 \leq cc_0\varepsilon^{-1} \quad\text{for all}\quad t\in[1,T].
  \end{align}
  Now we combine \eqref{Pf_GE:LinfH1_1} and \eqref{Pf_GE:LinfH1_2} to observe that
  \begin{align*}
    \|A_\varepsilon^{1/2}u^\varepsilon(t)\|_{L^2(\Omega_\varepsilon)}^2 \leq C_3c_0\varepsilon^{-1} \quad\text{for all}\quad t\in(0,T]
  \end{align*}
  with a constant $C_3>0$ independent of $\varepsilon$, $c_0$, and $T$.
  Hence if we define the constant $c_0$ by \eqref{Pf_GE:C0}, then by setting $t=T$ in the above inequality we get
  \begin{align*}
    \|A_\varepsilon^{1/2}u^\varepsilon(T)\|_{L^2(\Omega_\varepsilon)}^2 \leq \frac{C_1^2\varepsilon^{-1}}{4}, \quad\text{i.e.}\quad \varepsilon^{1/2}\|A_\varepsilon^{1/2}u^\varepsilon(T)\|_{L^2(\Omega_\varepsilon)} \leq \frac{C_1}{2} < C_1,
  \end{align*}
  which contradicts with \eqref{Pf_GE:Cont_Eq}.
  Therefore, \eqref{Pf_GE:Goal} is valid for all $t\in[0,T_{\max})$ and we conclude by Theorem~\ref{T:LE} that $T_{\max}=\infty$, i.e. the strong solution $u^\varepsilon$ to \eqref{E:NS_CTD} exists on the whole time interval $[0,\infty)$.
\end{proof}

Using the inequalities given in the proof of Theorem~\ref{T:GE}, we can also show the estimates \eqref{E:UE_L2} and \eqref{E:UE_H1} for a strong solution to \eqref{E:NS_CTD}.

\begin{proof}[Proof of Theorem~\ref{T:UE}]
  Let $\varepsilon_0$ and $c_0$ be the constants given in Lemma~\ref{L:Uni_aeps} and Theorem~\ref{T:GE}.
  Since $\alpha$ and $\beta$ are positive we can take $\varepsilon_1\in(0,\varepsilon_0]$ such that
  \begin{align*}
    c_1\varepsilon^\alpha\leq c_0, \quad c_2\varepsilon^\beta\leq c_0 \quad\text{for all}\quad \varepsilon\in(0,\varepsilon_1].
  \end{align*}
  Hence for $\varepsilon\in(0,\varepsilon_1]$ if $u_0^\varepsilon$ and $f^\varepsilon$ satisfy \eqref{E:UE_Data} then the inequality \eqref{E:GE_Data} holds and Theorem~\ref{T:GE} gives the existence of a global strong solution $u^\varepsilon$ to \eqref{E:NS_CTD}.

  Let us derive the estimates \eqref{E:UE_L2} and \eqref{E:UE_H1} for the strong solution $u^\varepsilon$.
  Hereafter we denote by $c$ a general positive constant independent of $\varepsilon$.
  First note that
  \begin{align} \label{Pf_UE:Ini_L2}
    \|u_0^\varepsilon\|_{L^2(\Omega_\varepsilon)}^2 \leq c(\varepsilon^{1+\alpha}+\varepsilon^\beta)
  \end{align}
  by \eqref{E:UE_Data} and \eqref{Pf_GE:Ini_Data}.
  We apply this inequality and \eqref{E:UE_Data} to \eqref{Pf_GE:LinfL2} to get
  \begin{align} \label{Pf_UE:U_LinfL2}
    \|u^\varepsilon(t)\|_{L^2(\Omega_\varepsilon)}^2 \leq c(\varepsilon^{1+\alpha}+\varepsilon^\beta) \quad\text{for all}\quad t\geq 0.
  \end{align}
  Also, integrating \eqref{Pf_GE:L2_InEst1} over $[0,t]$ and using \eqref{E:UE_Data} and \eqref{Pf_UE:Ini_L2} we have
  \begin{align} \label{Pf_UE:U_L2H1}
    \int_0^t\|A_\varepsilon^{1/2}u^\varepsilon(s)\|_{L^2(\Omega_\varepsilon)}^2\,ds \leq c(\varepsilon^{1+\alpha}+\varepsilon^\beta)(1+t) \quad\text{for all}\quad t\geq 0.
  \end{align}
  Combining \eqref{Pf_UE:U_LinfL2} and \eqref{Pf_UE:U_L2H1} with \eqref{E:Stokes_H1} we obtain \eqref{E:UE_L2}.

  Next let us prove \eqref{E:UE_H1}.
  From \eqref{E:UE_Data} and \eqref{E:Stokes_H1} it follows that
  \begin{align} \label{Pf_UE:Ini_L2_StL2}
    \|A_\varepsilon^{1/2}u_0^\varepsilon\|_{L^2(\Omega_\varepsilon)}^2 \leq c\varepsilon^{-1+\alpha}.
  \end{align}
  Also, we use \eqref{E:UE_Data} and \eqref{Pf_UE:Ini_L2} to \eqref{Pf_GE:L2H1} to deduce that
  \begin{align} \label{Pf_UE:L2H1_Short}
    \int_t^{t+1}\|A_\varepsilon^{1/2}u^\varepsilon(s)\|_{L^2(\Omega_\varepsilon)}^2\,ds \leq c(\varepsilon^{1+\alpha}+\varepsilon^\beta) \quad\text{for all}\quad t\geq 0.
  \end{align}
  Note that $t_\ast=\min\{t+1,T_{\max}\}=t+1$ in \eqref{Pf_GE:L2H1} by $T_{\max}=\infty$.
  Since \eqref{Pf_GE:Goal} and \eqref{Pf_UE:U_LinfL2} are valid for all $t\geq 0$, we can derive \eqref{Pf_GE:LinfH1_Int} for all $t\geq 0$ as in the proof of Theorem~\ref{T:GE}.
  When $t\in[0,1]$, we apply \eqref{E:UE_Data}, \eqref{Pf_UE:U_L2H1}, and \eqref{Pf_UE:Ini_L2_StL2} to \eqref{Pf_GE:LinfH1_Int} to get
  \begin{align} \label{Pf_UE:U_LinfH1_1}
    \|A_\varepsilon^{1/2}u^\varepsilon(t)\|_{L^2(\Omega_\varepsilon)}^2 \leq c(\varepsilon^{-1+\alpha}+\varepsilon^{-1+\beta}) \quad\text{for all}\quad t\in[0,1].
  \end{align}
  Let $t\geq1$.
  In \eqref{Pf_GE:Uni_Gron} the functions $\xi$ and $\zeta$ given by \eqref{Pf_GE:Def_Xi} satisfy
  \begin{align*}
    \int_{t-1}^t\xi(s)\,ds &\leq c\int_{t-1}^t\|A_\varepsilon^{1/2}u^\varepsilon(s)\|_{L^2(\Omega_\varepsilon)}^2\,ds \leq c, \\
    \int_{t-1}^t\zeta(s)\,ds &\leq c\left(\varepsilon^{-1}\int_{t-1}^t\|A_\varepsilon^{1/2}u^\varepsilon(s)\|_{L^2(\Omega_\varepsilon)}^2\,ds+\|\mathbb{P}_\varepsilon f^\varepsilon\|_{L^\infty(0,\infty;L^2(\Omega_\varepsilon))}^2\right) \\
    &\leq c(\varepsilon^{-1+\alpha}+\varepsilon^{-1+\beta})
  \end{align*}
  by \eqref{E:UE_Data}, \eqref{Pf_GE:L2_Ener}, and \eqref{Pf_UE:L2H1_Short}.
  Applying these estimates and \eqref{Pf_UE:L2H1_Short} to \eqref{Pf_GE:Uni_Gron} we get
  \begin{align} \label{Pf_UE:U_LinfH1_2}
    \|A_\varepsilon^{1/2}u^\varepsilon(t)\|_{L^2(\Omega_\varepsilon)}^2 \leq c(\varepsilon^{-1+\alpha}+\varepsilon^{-1+\beta}) \quad\text{for all}\quad t\geq 1.
  \end{align}
  By \eqref{E:Stokes_H1}, \eqref{Pf_UE:U_LinfH1_1}, and \eqref{Pf_UE:U_LinfH1_2} we obtain the first inequality of \eqref{E:UE_H1}.
  To prove the second one we see that \eqref{Pf_GE:H1_InEst2} holds on $(0,\infty)$ since \eqref{Pf_GE:L2_Ener} and \eqref{Pf_GE:Goal} are valid on $(0,\infty)$.
  Thus we integrate \eqref{Pf_GE:H1_InEst2} over $(0,t)$ and use \eqref{E:UE_Data}, \eqref{Pf_UE:U_L2H1}, and \eqref{Pf_UE:Ini_L2_StL2} to get
  \begin{align*}
    \int_0^t\|A_\varepsilon u^\varepsilon(s)\|_{L^2(\Omega_\varepsilon)}^2\,ds \leq c(\varepsilon^{-1+\alpha}+\varepsilon^{-1+\beta})(1+t) \quad\text{for all}\quad t\geq 0.
  \end{align*}
  This inequality combined with \eqref{E:Stokes_H2} yields the second inequality of \eqref{E:UE_H1}.
\end{proof}

\begin{appendix}
\section{Notations on vectors and matrices} \label{S:Ap_Vec}
In this appendix we fix notations on vectors and matrices.
For $m\in\mathbb{N}$ we consider a vector $a\in\mathbb{R}^m$ as a column vector
\begin{align*}
  a =
  \begin{pmatrix}
    a_1 \\ \vdots \\ a_m
  \end{pmatrix}
  = (a_1, \cdots, a_m)^T
\end{align*}
and denote the $i$-th component of $a$ by $a_i$ or sometimes by $a^i$ or $[a]_i$ for $i=1,\dots,m$.
A matrix $A\in\mathbb{R}^{l\times m}$ with $l,m\in\mathbb{N}$ is expressed as
\begin{align*}
  A = (A_{ij})_{i,j} =
  \begin{pmatrix}
    A_{11} & \cdots & A_{1m} \\
    \vdots & & \vdots \\
    A_{l1} & \cdots & A_{lm}
  \end{pmatrix}
\end{align*}
and the $(i,j)$-entry of $A$ is denoted by $A_{ij}$ or sometimes by $[A]_{ij}$ for $i=1,\dots,l$ and $j=1,\dots m$.
We denote the transpose of $A$ by $A^T$ and, when $l=m$, the symmetric part of $A$ by $A_S:=(A+A^T)/2$.
Also, we write $I_m$ for the $m\times m$ identity matrix.
The tensor product of $a\in\mathbb{R}^l$ and $b\in\mathbb{R}^m$ is defined as
\begin{align*}
  a\otimes b := (a_ib_j)_{i,j} =
  \begin{pmatrix}
    a_1b_1 & \cdots & a_1b_m \\
    \vdots & & \vdots \\
    a_lb_1 & \cdots & a_lb_m
  \end{pmatrix}, \quad
  a =
  \begin{pmatrix}
    a_1 \\ \vdots \\ a_l
  \end{pmatrix}, \quad
  b =
  \begin{pmatrix}
    b_1 \\ \vdots \\ b_m
  \end{pmatrix}.
\end{align*}
For three-dimensional vector fields $u=(u_1,u_2,u_3)^T$ and $\varphi$ on an open set in $\mathbb{R}^3$ let
\begin{gather*}
  \nabla u :=
  \begin{pmatrix}
    \partial_1u_1 & \partial_1u_2 & \partial_1u_3 \\
    \partial_2u_1 & \partial_2u_2 & \partial_2u_3 \\
    \partial_3u_1 & \partial_3u_2 & \partial_3u_3
  \end{pmatrix}, \quad
  |\nabla^2u|^2 := \sum_{i,j,k=1}^3|\partial_i\partial_ju_k|^2 \quad\left(\partial_i := \frac{\partial}{\partial x_i}\right), \\
  (\varphi\cdot\nabla)u :=
  \begin{pmatrix}
    \varphi\cdot\nabla u_1 \\
    \varphi\cdot\nabla u_2 \\
    \varphi\cdot\nabla u_3
  \end{pmatrix}
  = (\nabla u)^T\varphi.
\end{gather*}
Also, we define the inner product of $3\times 3$ matrices $A$ and $B$ and the norm of $A$ by
\begin{align*}
  A: B := \mathrm{tr}[A^TB] = \sum_{i=1}^3AE_i\cdot BE_i, \quad |A| := \sqrt{A:A},
\end{align*}
where $\{E_1,E_2,E_3\}$ is an orthonormal basis of $\mathbb{R}^3$.
Note that $A:B$ does not depend on a choice of $\{E_1,E_2,E_3\}$.
In particular, taking the standard basis of $\mathbb{R}^3$ we get
\begin{align*}
  A:B = \sum_{i,j=1}^3A_{ij}B_{ij} = B:A = A^T:B^T, \quad AB:C = A:CB^T = B:A^TC
\end{align*}
for $A,B,C\in\mathbb{R}^{3\times3}$.
Also, for $a,b\in\mathbb{R}^3$ we have $|a\otimes b|=|a||b|$.

\section{Proofs of auxiliary lemmas} \label{S:Ap_Proofs}
The purpose of this appendix is to present the proofs of Lemmas~\ref{L:La_Surf},~\ref{L:Agmon}, and~\ref{L:Tan_Curl_Ua}.
First we prove Lemma~\ref{L:La_Surf} after giving two auxiliary statements.
Recall that $\Gamma$ is a two-dimensional closed surface in $\mathbb{R}^3$ of class $C^5$.

\begin{lemma}[{\cite{Miu_NSCTD_01}*{Lemma~B.4}}] \label{L:Lp_Loc}
  Let $U$ be an open set in $\mathbb{R}^2$, $\mu\colon U\to\Gamma$ a $C^5$ local parametrization of $\Gamma$, and $\mathcal{K}$ a compact subset of $U$.
  For $p\in[1,\infty]$ if $\eta\in L^p(\Gamma)$ is supported in $\mu(\mathcal{K})$, then $\eta^\flat:=\eta\circ\mu\in L^p(U)$ and
  \begin{align} \label{E:Lp_Loc}
    c^{-1}\|\eta^\flat\|_{L^p(U)} \leq \|\eta\|_{L^p(\Gamma)} \leq c\|\eta^\flat\|_{L^p(U)}.
  \end{align}
  If in addition $\eta\in W^{1,p}(\Gamma)$, then $\eta^\flat\in W^{1,p}(U)$ and
  \begin{align} \label{E:W1p_Loc}
    c^{-1}\|\nabla_s\eta^\flat\|_{L^p(U)} \leq \|\nabla_\Gamma\eta\|_{L^p(\Gamma)} \leq c\|\nabla_s\eta^\flat\|_{L^p(U)},
  \end{align}
  where $\nabla_s\eta^\flat=(\partial_{s_1}\eta^\flat,\partial_{s_2}\eta^\flat)^T$ is the gradient of $\eta^\flat$ in $s\in\mathbb{R}^2$.
\end{lemma}

We omit the proof of Lemma~\ref{L:Lp_Loc} since it is given in our first paper~\cite{Miu_NSCTD_01}.

\begin{lemma} \label{L:La_R2}
  Let $U$ be an open set in $\mathbb{R}^2$.
  Then
  \begin{align} \label{E:La_R2}
    \|\varphi\|_{L^4(U)} \leq \sqrt{2}\|\varphi\|_{L^2(U)}^{1/2}\|\nabla_s\varphi\|_{L^2(U)}^{1/2}
  \end{align}
  for all $\varphi\in H_0^1(U)$.
\end{lemma}

The inequality \eqref{E:La_R2} is the well-known Ladyzhenskaya inequality on $\mathbb{R}^2$ (see~\cite{La69}*{Chapter~1, Section~1.1, Lemma~1}).
We give its proof for the readers' convenience.

\begin{proof}
  By a density argument, it is sufficient to prove \eqref{E:La_R2} for all $\varphi\in C_c^\infty(U)$.
  We extend $\varphi$ to $\mathbb{R}^2$ by setting it to zero outside $U$.
  Then
  \begin{align*}
    |\varphi(s_1,s_2)|^2 = \int_{-\infty}^{s_1}\frac{\partial}{\partial t_1}\bigl(|\varphi(t_1,s_2)|^2\bigr)\,dt_1 = 2\int_{-\infty}^{s_1}\varphi(t_1,s_2)\partial_{t_1}\varphi(t_1,s_2)\,dt_1
  \end{align*}
  for each $s=(s_1,s_2)\in\mathbb{R}^2$.
  Thus H\"{o}lder's inequality implies
  \begin{align*}
    |\varphi(s_1,s_2)|^2 \leq 2\left(\int_{-\infty}^\infty|\varphi(t_1,s_2)|^2\,dt_1\right)^{1/2}\left(\int_{-\infty}^\infty|\partial_{t_1}\varphi(t_1,s_2)|^2\,dt_1\right)^{1/2}.
  \end{align*}
  Similarly, we obtain
  \begin{align*}
    |\varphi(s_1,s_2)|^2 \leq 2\left(\int_{-\infty}^\infty|\varphi(s_1,t_2)|^2\,dt_2\right)^{1/2}\left(\int_{-\infty}^\infty|\partial_{t_2}\varphi(s_1,t_2)|^2\,dt_2\right)^{1/2}.
  \end{align*}
  From the above two inequalities we deduce that
  \begin{multline} \label{Pf_LaR2:Ineq}
    \int_{\mathbb{R}^2}|\varphi(s_1,s_2)|^4\,ds_1\,ds_2 \\
    \leq 4\int_{-\infty}^\infty\left(\int_{-\infty}^\infty|\varphi(t_1,s_2)|^2\,dt_1\right)^{1/2}\left(\int_{-\infty}^\infty|\partial_{t_1}\varphi(t_1,s_2)|^2\,dt_1\right)^{1/2}ds_2 \\
    \times \int_{-\infty}^\infty\left(\int_{-\infty}^\infty|\varphi(s_1,t_2)|^2\,dt_2\right)^{1/2}\left(\int_{-\infty}^\infty|\partial_{t_2}\varphi(s_1,t_2)|^2\,dt_2\right)^{1/2}ds_1.
  \end{multline}
  We again use H\"{o}lder's inequality to get
  \begin{multline*}
    \int_{-\infty}^\infty\left(\int_{-\infty}^\infty|\varphi(t_1,s_2)|^2\,dt_1\right)^{1/2}\left(\int_{-\infty}^\infty|\partial_{t_1}\varphi(t_1,s_2)|^2\,dt_1\right)^{1/2}ds_2 \\
    \begin{aligned}
      &\leq \left(\int_{-\infty}^\infty\int_{-\infty}^\infty|\varphi(t_1,s_2)|^2\,dt_1ds_2\right)^{1/2}\left(\int_{-\infty}^\infty\int_{-\infty}^\infty|\partial_{t_1}\varphi(t_1,s_2)|^2\,dt_1ds_2\right)^{1/2} \\
      &\leq \|\varphi\|_{L^2(\mathbb{R}^2)}\|\nabla_s\varphi\|_{L^2(\mathbb{R}^2)}
    \end{aligned}
  \end{multline*}
  and a similar inequality for the last line of \eqref{Pf_LaR2:Ineq}.
  By these inequalities and \eqref{Pf_LaR2:Ineq},
  \begin{align*}
    \|\varphi\|_{L^4(\mathbb{R}^2)}^4 \leq 4\|\varphi\|_{L^2(\mathbb{R}^2)}^2\|\nabla_s\varphi\|_{L^2(\mathbb{R}^2)}^2.
  \end{align*}
  Since $\varphi\in C_c^\infty(U)$ is compactly supported in $U$, this inequality implies \eqref{E:La_R2}.
\end{proof}

\begin{proof}[Proof of Lemma~\ref{L:La_Surf}]
  Since $\Gamma$ is compact, by a standard localization argument with a partition of unity on $\Gamma$ we may assume that there exist an open set $U$ in $\mathbb{R}^2$, a local parametrization $\mu\colon U\to\Gamma$ of $\Gamma$, and a compact subset $\mathcal{K}$ of $U$ such that $\eta$ is supported in $\mu(\mathcal{K})$.
  Then $\eta^\flat:=\eta\circ\mu$ is supported in $\mathcal{K}$.
  Moreover, since $\eta\in H^1(\Gamma)$, we have $\eta^\flat\in H_0^1(U)$ by Lemma~\ref{L:Lp_Loc} and thus we can use \eqref{E:La_R2} to $\eta^\flat$ to get
  \begin{align*}
    \|\eta^\flat\|_{L^4(U)} \leq \sqrt{2}\|\eta^\flat\|_{L^2(U)}^{1/2}\|\nabla_s\eta^\flat\|_{L^2(U)}^{1/2}.
  \end{align*}
  Applying \eqref{E:Lp_Loc} and \eqref{E:W1p_Loc} to this inequality we obtain \eqref{E:La_Surf}.
\end{proof}

Next we present the proof of Lemma~\ref{L:Agmon}.

\begin{proof}[Proof of Lemma~\ref{L:Agmon}]
  To prove \eqref{E:Agmon} we use the anisotropic Agmon inequality
  \begin{align} \label{Pf_A:Ans_Agm}
    \|\Phi\|_{L^\infty(V)} \leq c\|\Phi\|_{L^2(V)}^{1/4}\prod_{i=1}^3\left(\|\Phi\|_{L^2(V)}+\|\partial_i\Phi\|_{L^2(V)}+\|\partial_i^2\Phi\|_{L^2(V)}\right)^{1/4}
  \end{align}
  for $V=(0,1)^3$ and $\Phi\in H^2(V)$ (see~\cite{TeZi96}*{Proposition~2.2}).
  For this purpose, we use a partition of unity on $\Gamma$ to localize a function on $\Omega_\varepsilon$.

  Since $\Gamma$ is compact and without boundary, we can take a finite number of bounded open sets in $\mathbb{R}^2$ and local parametrizations of $\Gamma$
  \begin{align*}
    U_k \subset \mathbb{R}^2, \quad \mu_k\colon U_k\to\Gamma, \quad k=1,\dots,k_0
  \end{align*}
  such that $\{\mu_k(U_k)\}_{k=1}^{k_0}$ is an open covering of $\Gamma$.
  Let $\{\eta_k\}_{k=1}^{k_0}$ be a partition of unity on $\Gamma$ subordinate to $\{\mu_k(U_k)\}_{k=1}^{k_0}$.
  We may assume that $\eta_k$ is supported in $\mu_k(\mathcal{K}_k)$ with some compact subset $\mathcal{K}_k$ of $U_k$ for each $k=1,\dots,k_0$.
  Let $\bar{\eta}_k:=\eta_k\circ\pi$ be the constant extension of $\eta_k$ and
  \begin{multline*}
    \zeta_k(s) := \mu_k(s')+\varepsilon\{(1-s_3)g_0(\mu_k(s'))+s_3g_1(\mu_k(s'))\}n(\mu_k(s')), \\
    s = (s',s_3) \in V_k := U_k\times(0,1),\, k=1,\dots,k_0.
  \end{multline*}
  Then $\{\zeta_k(V_k)\}_{k=1}^{k_0}$ is an open covering of $\Omega_\varepsilon$ and $\{\bar{\eta}_k\}_{k=1}^{k_0}$ is a partition of unity on $\Omega_\varepsilon$ subordinate to $\{\zeta_k(V_k)\}_{k=1}^{k_0}$.
  For $\varphi\in H^2(\Omega_\varepsilon)$ we define
  \begin{align*}
    \varphi_k := \bar{\eta}_k\varphi \quad\text{on}\quad \Omega_\varepsilon, \, k=1,\dots,k_0.
  \end{align*}
  Then $\varphi_k$ is supported in $\zeta_k(\mathcal{K}_k\times(0,1))\subset\zeta_k(V_k)$ and
  \begin{align*}
    \partial_n\varphi_k = \bar{\eta}_k\partial_n\varphi, \quad \partial_n^2\varphi_k = \bar{\eta}_k\partial_n^2\varphi \quad\text{in}\quad \Omega_\varepsilon, \, k=1,\dots,k_0
  \end{align*}
  by \eqref{E:NorDer_Con}.
  Therefore, if we prove
  \begin{multline} \label{Pf_A:Agm_Loc}
    \|\varphi_k\|_{L^\infty(\zeta_k(V_k))} \leq c\varepsilon^{-1/2}\|\varphi_k\|_{L^2(\zeta_k(V_k))}^{1/4}\|\varphi_k\|_{H^2(\zeta_k(V_k))}^{1/2} \\
    \times\left(\|\varphi_k\|_{L^2(\zeta_k(V_k))}+\varepsilon\|\partial_n\varphi_k\|_{L^2(\zeta_k(V_k))}+\varepsilon^2\|\partial_n^2\varphi_k\|_{L^2(\zeta_k(V_k))}\right)^{1/4}
  \end{multline}
  for all $k=1,\dots,k_0$, then we obtain \eqref{E:Agmon} for $\varphi$.

  Let us show \eqref{Pf_A:Agm_Loc}.
  In what follows, we fix and suppress the index $k$.
  Hence we assume that $\varphi\in H^2(\Omega_\varepsilon)$ is supported in $\zeta(\mathcal{K}\times(0,1))\subset\zeta(V)$ with some compact subset $\mathcal{K}$ of $U$.
  Taking $U$ small and scaling it, we may further assume that
  \begin{align*}
    U = (0,1)^2, \quad V = U\times(0,1) = (0,1)^3.
  \end{align*}
  The local parametrization $\zeta$ of $\Omega_\varepsilon$ is of the form
  \begin{align} \label{Pf_A:Def_LP}
    \zeta(s) = \mu(s')+h_\varepsilon(s)n(\mu(s')), \quad s = (s',s_3) \in V = U\times(0,1),
  \end{align}
  where $\mu\colon U\to\Gamma$ is a $C^5$ local parametrization of $\Gamma$ and
  \begin{align} \label{Pf_A:Def_heps}
    h_\varepsilon(s) := \varepsilon\{(1-s_3)g_0(\mu(s'))+s_3g_1(\mu(s'))\}.
  \end{align}
  Since $g_0$, $g_1$, and $n$ are of class $C^4$ on $\Gamma$, $\mathcal{K}$ is compact in $U$, and $h_\varepsilon$ is an affine function of $s_3$, there exists a constant $c>0$ independent of $\varepsilon$ such that
  \begin{align} \label{Pf_A:GrZ_Bd}
    |\partial_{s_i}\zeta(s)| \leq c, \quad |\partial_{s_i}\partial_{s_j}\zeta(s)| \leq c, \quad s\in\mathcal{K}\times(0,1),\,i,j=1,2,3.
  \end{align}
  Let $\nabla_s\zeta$ be the gradient matrix of $\zeta$ in $s\in\mathbb{R}^3$, $J$ the function given by \eqref{E:Def_Jac}, and $\theta$ the Riemannian metric of $\Gamma$ given by
  \begin{align} \label{Pf_A:Riem}
    \begin{gathered}
      \theta(s') := \nabla_{s'}\mu(s')\{\nabla_{s'}\mu(s')\}^T, \quad s'\in U, \\
      \nabla_{s'}\mu :=
      \begin{pmatrix}
        \partial_{s_1}\mu_1 & \partial_{s_1}\mu_2 & \partial_{s_1}\mu_3 \\
        \partial_{s_2}\mu_1 & \partial_{s_2}\mu_2 & \partial_{s_2}\mu_3
      \end{pmatrix}.
    \end{gathered}
  \end{align}
  Then
  \begin{align} \label{Pf_A:Det_Z}
    \det\nabla_s\zeta(s) = \varepsilon g(\mu(s'))J(\mu(s'),h_\varepsilon(s))\sqrt{\det\theta(s')}, \quad s=(s',s_3)\in V,
  \end{align}
  which we prove at the end of the proof.
  Moreover, since $\det\theta$ is continuous and strictly positive on $U$ and $\mathcal{K}$ is compact in $U$, we have
  \begin{align*}
    \det\theta(s') \geq c, \quad s'\in\mathcal{K}.
  \end{align*}
  Applying this inequality, \eqref{E:G_Inf}, and \eqref{E:Jac_Bound_02} to \eqref{Pf_A:Det_Z} we obtain
  \begin{align} \label{Pf_A:DetZ_Bd}
    \det\nabla_s\zeta(s) \geq c\varepsilon, \quad s\in \mathcal{K}\times(0,1).
  \end{align}
  Now let $\Phi:=\varphi\circ\zeta$ on $V$.
  Then
  \begin{align} \label{Pf_A:Linf}
    \|\Phi\|_{L^\infty(V)} = \|\varphi\|_{L^\infty(\zeta(V))}
  \end{align}
  and $\Phi$ is supported in $\mathcal{K}\times(0,1)$ since $\varphi$ is supported in $\zeta(\mathcal{K}\times(0,1))$.
  Also, since
  \begin{align} \label{Pf_A:Change}
    \int_{\zeta(V)}\varphi(x)\,dx = \int_V\Phi(s)\det\nabla_s\zeta(s)\,ds,
  \end{align}
  we observe by \eqref{Pf_A:DetZ_Bd} that
  \begin{align} \label{Pf_A:L2}
    \|\Phi\|_{L^2(V)} \leq c\varepsilon^{-1/2}\|\varphi\|_{L^2(\zeta(V))}.
  \end{align}
  We differentiate $\Phi(s)=\varphi(\zeta(s))$ in $s\in V$.
  Then
  \begin{align*}
    \partial_{s_i}\Phi(s) &= \partial_{s_i}\zeta(s)\cdot\nabla\varphi(\zeta(s)), \\
    \partial_{s_i}^2\Phi(s) &= \partial_{s_i}^2\zeta(s)\cdot\nabla\varphi(\zeta(s))+\partial_{s_i}\zeta(s)\cdot\nabla^2\varphi(\zeta(s))\partial_{s_i}\zeta(s)
  \end{align*}
  for $s\in V$ and $i=1,2$, and
  \begin{align*}
    \partial_{s_3}\Phi(s) = \varepsilon g(\mu(s'))\partial_n\varphi(\zeta(s)), \quad \partial_{s_3}^2\Phi(s) = \varepsilon^2g(\mu(s'))^2\partial_n^2\varphi(\zeta(s))
  \end{align*}
  for $s=(s',s_3)\in V$.
  Hence \eqref{Pf_A:GrZ_Bd} and the boundedness of $g$ on $\Gamma$ imply that
  \begin{align*}
    |\partial_{s_i}\Phi(s)| &\leq c|\nabla\varphi(\zeta(s))|, \\
    |\partial_{s_i}^2\Phi(s)| &\leq c(|\nabla\varphi(\zeta(s))|+|\nabla^2\varphi(\zeta(s))|), \\
    |\partial_{s_3}^k\Phi(s)| &\leq c\varepsilon^k|\partial_n^k\varphi(\zeta(s))|,
  \end{align*}
  for $i,k=1,2$ and $s\in\mathcal{K}\times(0,1)$.
  Since $\Phi$ is supported in $\mathcal{K}\times(0,1)$, we deduce from these inequalities and \eqref{Pf_A:Change} that
  \begin{align} \label{Pf_A:Hk}
    \begin{aligned}
      \|\partial_{s_i}^k\Phi\|_{L^2(V)} &\leq c\varepsilon^{-1/2}\|\varphi\|_{H^k(\zeta(V))}, \\
      \|\partial_{s_3}^k\Phi\|_{L^2(V)} &\leq c\varepsilon^{k-1/2}\|\partial_n^k\varphi\|_{L^2(\zeta(V))}
    \end{aligned}
  \end{align}
  for $i,k=1,2$ and thus $\Phi\in H^2(V)$.
  Hence we can apply \eqref{Pf_A:Ans_Agm} to $\Phi=\varphi\circ\zeta$ and use \eqref{Pf_A:Linf}, \eqref{Pf_A:L2}, and \eqref{Pf_A:Hk} to obtain \eqref{Pf_A:Agm_Loc}.

  It remains to show the formula \eqref{Pf_A:Det_Z}.
  In what follows, we use the notation
  \begin{align*}
    \eta^\flat(s') := \eta(\mu(s')), \quad s'\in U
  \end{align*}
  for a function $\eta$ on $\Gamma$.
  Note that, since $\eta(\mu(s'))=\bar{\eta}(\mu(s'))$ by $\mu(s')\in\Gamma$,
  \begin{align} \label{Pf_A:Gr_Flat}
    \partial_{s_i}\eta^\flat(s') = \partial_{s_i}\mu(s')\cdot\nabla\bar{\eta}(\mu(s')) = \partial_{s_i}\mu(s')\cdot\nabla_\Gamma\eta(\mu(s'))
  \end{align}
  for $i=1,2$ by \eqref{E:ConDer_Surf}.
  By \eqref{Pf_A:Def_LP} and \eqref{Pf_A:Def_heps} we have
  \begin{align*}
    \zeta(s) = \mu(s')+\varepsilon\{(1-s_3)g_0^\flat(s')+s_3g_1^\flat(s')\}n^\flat(s'), \quad s = (s',s_3) \in V = U\times(0,1).
  \end{align*}
  We differentiate $\zeta(s)$ and apply \eqref{Pf_A:Gr_Flat} and $-\nabla_\Gamma n=W=W^T$ on $\Gamma$ to get
  \begin{align} \label{Pf_A:Grad_Z}
    \begin{aligned}
      \partial_{s_i}\zeta(s) &= \{I_3-h_\varepsilon(s)W^\flat(s')\}\partial_{s_i}\mu(s')+\eta_\varepsilon^i(s)n^\flat(s'), \quad i=1,2, \\
      \partial_{s_3}\zeta(s) &= \varepsilon g^\flat(s')n^\flat(s')
    \end{aligned}
  \end{align}
  for $s=(s',s_3)\in V$, where $h_\varepsilon(s)$ is given by \eqref{Pf_A:Def_heps} and
  \begin{align*}
    \eta_\varepsilon^i(s) := \varepsilon\partial_{s_i}\mu(s')\cdot\{(1-s_3)(\nabla_\Gamma g_0)^\flat(s')+s_3(\nabla_\Gamma g_1)^\flat(s')\}, \quad i=1,2.
  \end{align*}
  From now on, we fix and suppress the arguments $s'$ and $s$.
  By \eqref{Pf_A:Grad_Z} we have
  \begin{align*}
    \nabla_s\zeta =
    \begin{pmatrix}
      \partial_{s_1}\zeta_1 & \partial_{s_1}\zeta_2 & \partial_{s_1}\zeta_3 \\
      \partial_{s_2}\zeta_1 & \partial_{s_2}\zeta_2 & \partial_{s_2}\zeta_3 \\
      \partial_{s_3}\zeta_1 & \partial_{s_3}\zeta_2 & \partial_{s_3}\zeta_3
    \end{pmatrix} =
    \begin{pmatrix}
      \nabla_{s'}\mu(I_3-h_\varepsilon W^\flat)^T+\eta_\varepsilon\otimes n^\flat \\
      \varepsilon g^\flat(n^\flat)^T
    \end{pmatrix}.
  \end{align*}
  Here we consider $n^\flat\in\mathbb{R}^3$ and $\eta_\varepsilon:=(\eta_\varepsilon^1,\eta_\varepsilon^2)^T\in\mathbb{R}^2$ as column vectors.
  Since $\partial_{s_1}\mu$ and $\partial_{s_2}\mu$ are tangent to $\Gamma$ at $\mu(s')$ we have $(\nabla_{s'}\mu)n^\flat=0$.
  Moreover,
  \begin{align*}
    W^\flat n^\flat=0, \quad (\eta_\varepsilon\otimes n^\flat)n^\flat=|n^\flat|^2\eta_\varepsilon=\eta_\varepsilon, \quad (\eta_\varepsilon\otimes n^\flat)(n^\flat\otimes \eta_\varepsilon) = \eta_\varepsilon\otimes\eta_\varepsilon.
  \end{align*}
  From these equalities and the symmetry of the matrix $W^\flat$ it follows that
  \begin{align*}
    \nabla_s\zeta(\nabla_s\zeta)^T =
    \begin{pmatrix}
      \nabla_{s'}\mu(I_3-h_\varepsilon W^\flat)^2(\nabla_{s'}\mu)^T+\eta_\varepsilon\otimes\eta_\varepsilon & \varepsilon g^\flat\eta_\varepsilon \\
      \varepsilon g^\flat\eta_\varepsilon^T & \varepsilon^2(g^\flat)^2
    \end{pmatrix}.
  \end{align*}
  Hence by elementary row operations we have
  \begin{align*}
    \det[\nabla_s\zeta(\nabla_s\zeta)^T] &= \det
    \begin{pmatrix}
      \nabla_{s'}\mu(I_3-h_\varepsilon W^\flat)^2(\nabla_{s'}\mu)^T+\eta_\varepsilon\otimes\eta_\varepsilon & \varepsilon g^\flat\eta_\varepsilon \\
      \varepsilon g^\flat\eta_\varepsilon^T & \varepsilon^2(g^\flat)^2
    \end{pmatrix} \\
    &= \det
    \begin{pmatrix}
      \nabla_{s'}\mu(I_3-h_\varepsilon W^\flat)^2(\nabla_{s'}\mu)^T & 0 \\
      \varepsilon g^\flat\eta_\varepsilon^T & \varepsilon^2 (g^\flat)^2
    \end{pmatrix} \\
    &= \varepsilon^2 (g^\flat)^2\det[\nabla_{s'}\mu(I_3-h_\varepsilon W^\flat)^2(\nabla_{s'}\mu)^T].
  \end{align*}
  Since $\det[\nabla_s\zeta(\nabla_s\zeta)^T]=(\det\nabla_s\zeta)^2$, the above equality implies that
  \begin{align} \label{Pf_A:DZ_Sq}
    (\det\nabla_s\zeta)^2 = \varepsilon^2(g^\flat)^2\det[\nabla_{s'}\mu(I_3-h_\varepsilon W^\flat)^2(\nabla_{s'}\mu)^T].
  \end{align}
  To compute the right-hand side we define $3\times 3$ matrices
    \begin{align*}
    A :=
    \begin{pmatrix}
      \nabla_{s'}\mu \\
      (n^\flat)^T
    \end{pmatrix}, \quad
    A_h :=
    \begin{pmatrix}
      \nabla_{s'}\mu(I_3-h_\varepsilon W^\flat) \\
      (n^\flat)^T
    \end{pmatrix}.
  \end{align*}
  Then by $(\nabla_{s'}\mu)n^\flat=0$, $W^\flat n^\flat=0$, the symmetry of $W^\flat$, and \eqref{Pf_A:Riem} we have
  \begin{gather*}
    A_h = A(I_3-h_\varepsilon W^\flat), \\
    AA^T =
    \begin{pmatrix}
      \theta & 0 \\
      0 & 1
    \end{pmatrix}, \quad
    A_hA_h^T =
    \begin{pmatrix}
      \nabla_{s'}\mu(I_3-h^\flat W^\flat)^2(\nabla_{s'}\mu)^T & 0 \\
      0 & 1
    \end{pmatrix}.
  \end{gather*}
  Noting that $A$ and $I_3-h_\varepsilon W^\flat$ are $3\times3$ matrices, we use these equalities to get
  \begin{align*}
    \begin{aligned}
      \det[\nabla_{s'}\mu(I_3-h_\varepsilon W^\flat)^2(\nabla_{s'}\mu)^T] &= \det[A_hA_h^T] = \det[A(I_3-h_\varepsilon W^\flat)^2A^T] \\
      &= \det[(I_3-h_\varepsilon W^\flat)^2]\det[AA^T] \\
      &= J(\mu,h_\varepsilon)^2\det\theta,
    \end{aligned}
  \end{align*}
  where the last equality follows from $\det(I_3-h_\varepsilon W^\flat)=J(\mu,h_\varepsilon)$.
  From this equality and \eqref{Pf_A:DZ_Sq} we deduce that
  \begin{align*}
    (\det\nabla_s\zeta)^2 = \varepsilon^2(g^\flat)^2J(\mu,h_\varepsilon)^2\det\theta.
  \end{align*}
  This equality yields \eqref{Pf_A:Det_Z} since $g^\flat$ and $J(\mu,h_\varepsilon)$ are positive by \eqref{E:G_Inf} and \eqref{E:Jac_Bound_02}.
\end{proof}

Finally, let us prove Lemma~\ref{L:Tan_Curl_Ua}.
To this end, we give an auxiliary result.

\begin{lemma} \label{L:Curl_Exp}
  Let $E_1$, $E_2$, and $E_3$ be vector fields on an open subset $U$ of $\mathbb{R}^3$ such that $\{E_1(x),E_2(x),E_3(x)\}$ is an orthonormal basis of $\mathbb{R}^3$ for each $x\in\mathbb{R}^3$ and
  \begin{align*}
    E_1\times E_2 = E_3, \quad E_2\times E_3 = E_1, \quad E_3\times E_1 = E_2 \quad\text{in}\quad U.
  \end{align*}
  Then for $u\in C^1(U)^3$ we have
  \begin{multline} \label{E:Curl_Exp}
    \mathrm{curl}\,u = \{(E_2\cdot\nabla)u\cdot E_3-(E_3\cdot\nabla)u\cdot E_2\}E_1 \\
    +\{(E_3\cdot\nabla)u\cdot E_1-(E_1\cdot\nabla)u\cdot E_3\}E_2 \\
    +\{(E_1\cdot\nabla)u\cdot E_2-(E_2\cdot\nabla)u\cdot E_1\}E_3 \quad\text{in}\quad U.
  \end{multline}
\end{lemma}

\begin{proof}
  By the assumption, $\mathrm{curl}\,u=\sum_{i=1}^3(\mathrm{curl}\,u\cdot E_i)E_i$.
  Since $E_1=E_2\times E_3$,
  \begin{align*}
    \mathrm{curl}\,u\cdot E_1 &= \mathrm{curl}\,u\cdot(E_2\times E_3) = E_2\cdot(E_3\times\mathrm{curl}\,u) \\
    &= E_2\cdot\{(\nabla u)E_3-(\nabla u)^TE_3\} \\
    &= (\nabla u)^TE_2\cdot E_3-(\nabla u)^TE_3\cdot E_2 \\
    &= (E_2\cdot\nabla)u\cdot E_3-(E_3\cdot\nabla)u\cdot E_2.
  \end{align*}
  Calculating $\mathrm{curl}\,u\cdot E_i$, $i=2,3$ in the same way we obtain \eqref{E:Curl_Exp}.
\end{proof}

\begin{proof}[Proof of Lemma~\ref{L:Tan_Curl_Ua}]
  Let $u\in C^1(\Omega_\varepsilon)^3$ and $u^a$ be given by \eqref{E:Def_ExAve}.
  Since the surface $\Gamma$ is compact and without boundary, we can take a finite number of relatively open subsets $O_k$ of $\Gamma$ and pairs of tangential vector fields $\{\tau_1^k,\tau_2^k\}$ on $O_k$, $k=1,\dots,k_0$ such that $\Gamma=\bigcup_{k=1}^{k_0}O_k$, the triplet $\{\tau_1^k,\tau_2^k,n\}$ forms an orthonormal basis of $\mathbb{R}^3$ on $O_k$, and
  \begin{align*}
    \tau_1^k\times\tau_2^k = n, \quad \tau_2^k\times n = \tau_1^k, \quad n\times\tau_1^k = \tau_2^k \quad\text{on}\quad O_k
  \end{align*}
  for each $k=1,\dots,k_0$.
  Then since $\Omega_\varepsilon=\bigcup_{k=1}^{k_0}U_k$ with
  \begin{align*}
    U_k:=\{y+rn(y)\mid y\in O_k,\,r\in(\varepsilon g_0(y),\varepsilon g_1(y))\}, \quad k=1,\dots,k_0,
  \end{align*}
  it is sufficient to show \eqref{E:Tan_Curl_Ua} in $U_k$ for each $k=1,\dots,k_0$.

  From now on, we fix and suppress the index $k$ and carry out calculations in $U$ unless otherwise stated.
  We apply \eqref{E:Curl_Exp} to $u^a$ with
  \begin{align*}
    E_1 := \bar{\tau}_1, \quad E_2 := \bar{\tau}_2, \quad E_3 := \bar{n}
  \end{align*}
  and use $P\tau_i=\tau_i$ for $i=1,2$ and $Pn=0$ on $O$ to get
  \begin{align*}
    \overline{P}\,\mathrm{curl}\,u^a = \{(\bar{\tau}_2\cdot\nabla)u^a\cdot\bar{n}-(\bar{n}\cdot\nabla)u^a\cdot\bar{\tau}_2\}\bar{\tau}_1+\{(\bar{n}\cdot\nabla)u^a\cdot\bar{\tau}_1-(\bar{\tau}_1\cdot\nabla)u^a\cdot\bar{n}\}\bar{\tau}_2.
  \end{align*}
  By this equality, $(\bar{n}\cdot\nabla)u^a=\partial_nu^a$, and $|\bar{\tau}_1|=|\bar{\tau}_2|=|\bar{n}|=1$ we get
  \begin{align} \label{Pf_TCUa:Est}
    \left|\overline{P}\,\mathrm{curl}\,u^a\right| \leq c\left(|\partial_nu^a|+|(\bar{\tau}_1\cdot\nabla)u^a\cdot\bar{n}|+|(\bar{\tau}_2\cdot\nabla)u^a\cdot\bar{n}|\right).
  \end{align}
  Let us estimate each term on the right-hand side.
  By \eqref{E:NorDer_Con} and \eqref{E:Def_ExAve} we have
  \begin{align*}
    \partial_nu^a = \partial_n\Bigl[\overline{M_\tau u}+\Bigl(\overline{M_\tau u}\cdot\Psi_\varepsilon\Bigr)\bar{n}\Bigr] = \Bigl(\overline{M_\tau u}\cdot\partial_n\Psi_\varepsilon\Bigr)\bar{n}.
  \end{align*}
  Hence it follows from \eqref{E:ExAux_Bound} that
  \begin{align} \label{Pf_TCUa:DN}
    |\partial_nu^a| = \left|\overline{M_\tau u}\cdot\partial_n\Psi_\varepsilon\right| \leq c\left|\overline{M_\tau u}\right| = c\left|\overline{PMu}\right| \leq c\left|\overline{Mu}\right|.
  \end{align}
  To estimate the other terms we set
  \begin{align} \label{Pf_TCUa:Def_TN}
    u_\tau^a := \overline{P}u^a = \overline{M_\tau u}, \quad u_n^a := (u^a\cdot\bar{n})\bar{n} = \Bigl(\overline{M_\tau u}\cdot\Psi_\varepsilon\Bigr)\bar{n}
  \end{align}
  so that $u^a=u_\tau^a+u_n^a$.
  Let $i=1,2$.
  Since $u_\tau^a\cdot\bar{n}=0$ in $U$, we have
  \begin{align*}
    (\bar{\tau}_i\cdot\nabla)u_\tau^a\cdot\bar{n} = (\bar{\tau}_i\cdot\nabla)(u_\tau^a\cdot\bar{n})-u_\tau^a\cdot(\bar{\tau}_i\cdot\nabla)\bar{n} = -u_\tau^a\cdot(\bar{\tau}_i\cdot\nabla)\bar{n}.
  \end{align*}
  Hence by \eqref{E:NorG_Bound} and $|\bar{\tau}_i|=1$ we get
  \begin{align} \label{Pf_TCUa:Utaua}
    |(\bar{\tau}_i\cdot\nabla)u_\tau^a\cdot\bar{n}| \leq c|u_\tau^a| \leq c\left|\overline{Mu}\right|, \quad i=1,2.
  \end{align}
  Next we deal with $(\bar{\tau}_i\cdot\nabla)u_n^a\cdot\bar{n}$.
  Since $\tau_i=P\tau_i$, $P=P^T$, and $|\tau_i|=1$ on $O$,
  \begin{align} \label{Pf_TCUa:Der_Uan}
    |(\bar{\tau}_i\cdot\nabla)u_n^a| = \left|(\nabla u_n^a)^T\overline{P}\bar{\tau}_i\right| = \left|\Bigl(\overline{P}\nabla u_n^a\Bigr)^T\bar{\tau}_i\right| \leq \left|\overline{P}\nabla u_n^a\right|.
  \end{align}
  Moreover, by the definition \eqref{Pf_TCUa:Def_TN} of $u_n^a$ we have
  \begin{align*}
    \overline{P}\nabla u_n^a = \left[\left\{\overline{P}\nabla\Bigl(\overline{M_\tau u}\Bigr)\right\}\Psi_\varepsilon+\Bigl(\overline{P}\nabla\Psi_\varepsilon\Bigr)\overline{M_\tau u}\right]\otimes\bar{n}+\Bigl(\overline{M_\tau u}\cdot\Psi_\varepsilon\Bigr)\overline{P}\nabla\bar{n}
  \end{align*}
  and thus we deduce from \eqref{E:ConDer_Bound}, \eqref{E:NorG_Bound}, \eqref{E:ExAux_Bound}, and \eqref{E:ExAux_TNDer} that
  \begin{align} \label{Pf_TCUa:PGr_Uan}
    \left|\overline{P}\nabla u_n^a\right| \leq c\varepsilon\left(\left|\overline{M_\tau u}\right|+\left|\overline{\nabla_\Gamma M_\tau u}\right|\right) \leq c\varepsilon\left(\left|\overline{Mu}\right|+\left|\overline{\nabla_\Gamma Mu}\right|\right).
  \end{align}
  In the last inequality we also used $M_\tau u=PMu$ on $\Gamma$ and the $C^4$-regularity of $P$ on $\Gamma$.
  By \eqref{Pf_TCUa:Der_Uan} and \eqref{Pf_TCUa:PGr_Uan} we observe that
  \begin{align} \label{Pf_TCUa:Uan}
    |(\bar{\tau}_i\cdot\nabla)u_n^a\cdot\bar{n}| \leq |(\bar{\tau}_i\cdot\nabla)u_n^a| \leq c\varepsilon\left(\left|\overline{Mu}\right|+\left|\overline{\nabla_\Gamma Mu}\right|\right), \quad i=1,2.
  \end{align}
  Noting that $u^a=u_\tau^a+u_n^a$, we conclude by \eqref{Pf_TCUa:Est}, \eqref{Pf_TCUa:DN}, \eqref{Pf_TCUa:Utaua}, and \eqref{Pf_TCUa:Uan} that the inequality \eqref{E:Tan_Curl_Ua} holds in $U$.
\end{proof}

\end{appendix}

\section*{Acknowledgments}
This work is an expanded version of a part of the doctoral thesis of the author~\cite{Miu_DT} completed under the supervision of Professor Yoshikazu Giga at the University of Tokyo.
The author is grateful to him for his valuable comments on this work and also would like to thank Mr. Yuuki Shimizu for fruitful discussions on Killing vector fields on surfaces.

The work of the author was supported by Grant-in-Aid for JSPS Fellows No. 16J02664 and No. 19J00693, and by the Program for Leading Graduate Schools, MEXT, Japan.


\end{document}